\documentclass[12pt,reqno]{amsart}

\usepackage[fleqn,tbtags]{mathtools}
\usepackage{amssymb}
\usepackage{amsthm}
\usepackage{iftex}
\ifPDFTeX
  \usepackage[T1]{fontenc}
  \IfFileExists{mlmodern.sty}{\usepackage{mlmodern}}{\usepackage{lmodern}}
\else
  \usepackage{fontspec}
  \usepackage{unicode-math}
\fi
\usepackage{graphicx}
\usepackage[shortlabels]{enumitem}
\usepackage{microtype}
\usepackage[nodisplayskipstretch]{setspace}
\usepackage{xcolor}
\usepackage{tikz}
\usetikzlibrary{arrows.meta,positioning,calc}
\usepackage[margin=1.03in]{geometry}
\usepackage[colorlinks=true,linkcolor=red,citecolor=blue,urlcolor=blue,
            hypertexnames=false]{hyperref}
\usepackage[nameinlink,capitalize,noabbrev]{cleveref}

\theoremstyle{plain}
\newtheorem{theorem}{Theorem}[section]
\newtheorem{proposition}[theorem]{Proposition}
\newtheorem{lemma}[theorem]{Lemma}
\newtheorem{corollary}[theorem]{Corollary}
\theoremstyle{definition}
\newtheorem{definition}[theorem]{Definition}
\theoremstyle{remark}
\newtheorem{remark}[theorem]{Remark}

\newtheorem{conjecture}[theorem]{Conjecture}
\newtheorem{problem}[theorem]{Problem}

\theoremstyle{plain}
\newtheorem{maintheorem}{Theorem}

\newtheorem{mainbandotheorem}{Theorem}

\newtheorem{mainpropernesstheorem}{Theorem}

\newtheorem{mainmatsushimatheorem}{Theorem}

\newcommand{\ddc}{dd^c}

\newenvironment{thmenumerate}
 {\begin{enumerate}
  
  \setlength{\itemsep}{0.35em}
  \setlength{\parsep}{0pt}
  \setlength{\topsep}{0.35em}}
 {\end{enumerate}}

\title[Uniqueness of KE currents]{A Solution to Berndtsson's Problem \\[0.3em]
and Uniqueness of Twisted
KE Currents}

\author[Y. Li]{Yinji Li}
\address{Yinji Li: Institute of Mathematics\\Academy of Mathematics and Systems Science\\Chinese Academy of
		Sciences\\Beijing\\100190\\P. R. China}
\email{liyinji@amss.ac.cn}
\author[H. Sun]{Haoyuan Sun}
\address{Haoyuan Sun: School of Mathematical Sciences\\ Beijing Normal University\\ Beijing 100875\\ P. R. China}
\email{202531130037@mail.bnu.edu.cn}
\author[Z. Wang]{Zhiwei Wang}
\address{Zhiwei Wang: Laboratory of Mathematics and Complex Systems (Ministry of Education)\\ School of Mathematical Sciences\\ Beijing Normal University\\ Beijing 100875\\ P. R. China}
\email{zhiwei@bnu.edu.cn}

\author[X. Zhou]{Xiangyu Zhou}
\address{Xiangyu Zhou: Institute of Mathematics\\ Academy of Mathematics and Systems Science\\
	and Hua Loo--Keng Key Laboratory of Mathematics\\ Chinese Academy of Sciences\\ Beijing 100190\\ P. R. China}
\email{xyzhou@math.ac.cn}

\begin{document}
\raggedbottom
\begin{abstract}
We prove the uniqueness conjecture on the twisted K\"ahler--Einstein currents in big cohomology classes, by solving a problem of Bo Berndtsson posed in his work on the uniqueness theorem of K\"aher--Einstein metrics.
\end{abstract}

\maketitle

\tableofcontents

\section{Introduction}

The study of canonical metrics is a central area in K\"ahler geometry. 
Different versions of the
Yau--Tian--Donaldson conjecture relate the existence of canonical metrics on Fano manifolds to suitable algebro-geometric stability conditions.  These conjectures have shaped one of the major research directions in modern K\"ahler geometry and have led to significant developments at the interface of differential geometry, geometric analysis, several complex variables, complex geometry, and algebraic geometry, see, for example
\cite{Din88,DT92,Yau93,Tia97,Don02,LX14,CDS15,Tia15,Ber16,DS16,
BHJ17,DR17,CSW18,BX19,BBJ21,LTW21,Li22,LTW22,LXZ22,Zhou24}
and the references therein.

The present paper discusses a uniqueness problem for twisted
K\"ahler--Einstein currents in transcendental big classes on compact
K\"ahler manifolds\footnote{For a concurrent small-parameter
uniqueness result for Monge--Amp\`ere mean-field equations in big
classes, obtained by a quantitative stability method, see
Dang--Zhang--Zhou \cite{DZZ26}.}. Its natural framework is the pluripotential
theory for big classes, including non-pluripolar products, finite-energy
spaces, and their metric and geodesic geometry; see, for example,
\cite{BT87,EGZ09,BEGZ10,BBGZ13,DDL18a,DDL18b,BBEGZ19}.  The threshold
invariants relevant below arise from the analytic and valuative
theories of the $\alpha$-- and $\delta$--invariants
\cite{Tia87,Fuj16,FO18,Fuj19,BJ20,Zha24}.

The role of holomorphic symmetry in the uniqueness problem is classical:
Matsushima's theorem \cite{Mat57} gives reductivity of the
automorphism Lie algebra, Bando--Mabuchi theorem \cite{BM87}
identifies K\"ahler--Einstein metrics modulo the identity component
of the automorphism group.  
An analytic approach to these uniqueness
problems emerged from the geometry of the space of K\"ahler
potentials.  Semmes \cite{Sem92} and Donaldson \cite{Don99} related
its geodesic equation to the homogeneous complex Monge--Amp\`ere
equation, and Chen \cite{Che00} constructed $C^{1,1}$ geodesics
between smooth K\"ahler potentials.  

\vspace{0.4cm}
In his important work \cite{Ber15b}, Berndtsson provided a pluripotential approach to generalize the
classical Bando--Mabuchi uniqueness theorem.
In \cite[Theorem~1.2]{Ber15b}, he assumed that
$X$ is a projective manifold, $-K_X$ carries a smooth semipositive metric whose curvature form is $\theta$,
and that
\[
 H^{0,1}(X)=0.
\]
Berndtsson established and proved a generalized Bando--Mabuchi uniqueness theorem by showing that if the subgeodesic $\phi_s\in\operatorname{PSH}(X,\theta)$ is uniformly bounded, then the affinity of $s\mapsto-\log\int_Xe^{-\phi_s}\Omega$ produces a holomorphic vector field $V$, whose flow joins $\theta+dd^c\phi_s$. 
Naturally, Berndtsson \cite[page 5-6]{Ber15b} posed the following problem:
\begin{problem}[Berndtsson]
\label{que:Bern}
\leavevmode
    Assuming only that
$$
e^{-\phi_s}\in L^1,\qquad \forall s\in I,
$$
and without assuming \(H^{0,1}(X)=0\), can one still obtain
 a holomorphic vector field with flow $G_s$ such that $G_s^*(\theta+dd^c\phi_s)=\theta+dd^c\phi_0$?
\end{problem}

Our first main theorem settles Problem~\ref{que:Bern} affirmatively.  It requires neither $H^{0,1}(X)=0$ nor any restriction on
the singularity type of the subgeodesic, and applies even when
$-K_X$ is merely pseudoeffective.
\vspace{0.5cm}

Let $X$ be a connected compact K\"ahler manifold with a smooth positive volume form $\Omega$ and
$I\subset\mathbb R$ be an open interval, set
$S_I:=I+i\mathbb R$, and denote by
$p_X:X\times S_I\to X$ the projection.  Let $\theta:=\operatorname{Ric}(\Omega)$, a smooth
closed real $(1,1)$-form. A family
$(\phi_s)_{s\in I}\subset\operatorname{PSH}(X,\theta)$ is called a
$\theta$-psh subgeodesic when
\[
 \phi(x,\tau):=\phi_{\operatorname{Re}\tau}(x)
\]
is $p_X^*\theta$-psh on $X\times S_I$.  
We suppress the
factor $2\pi$ in Chern--Weil representatives, so that
$\operatorname{Ric}(\Omega)$ represents $c_1(-K_X)$. 

\begin{maintheorem}
\label{thm:main}
Let $X$ be a connected compact K\"ahler manifold, $\Omega$ be a
smooth positive volume form with
$\theta:=\operatorname{Ric}(\Omega)\in c_1(-K_X)$.
Assume that $-K_X$ is pseudoeffective.  Let
$(\phi_s)_{s\in I}\subset\operatorname{PSH}(X,\theta)$ be a
$\theta$-psh subgeodesic such that $e^{-\phi_s}\in L^1(X,\Omega)$ for
every $s\in I$, and set
\[
 F(s):=-\log\int_Xe^{-\phi_s}\Omega.
\]
If $F$ is affine, there is a holomorphic vector field
$\mathcal V\in H^0(X,T_X^{1,0})$ so that if $(G_t)$ is the real flow of
$\operatorname{Re}\mathcal V$ and we extend $\mathcal V$ trivially to
$X\times S_I$, then
\begin{thmenumerate}
\item
\begin{equation}
 \iota_{\partial/\partial\tau+\mathcal V/2}
 \bigl(p_X^*\theta+\ddc_{x,\tau}(\phi-F)\bigr)=0
 \label{eq:main-full-nullity}
\end{equation}
as a current on $X\times S_I$;
\item for all $s,s_0\in I$,
\begin{equation}
 G_{s-s_0}^*\bigl(\theta+\ddc_X\phi_s\bigr)
 =\theta+\ddc_X\phi_{s_0};
 \label{eq:main-current-transport}
\end{equation}
\item for all $s,s_0\in I$,
\begin{equation}
 G_{s-s_0}^*\bigl(e^{F(s)-\phi_s}\Omega\bigr)
 =e^{F(s_0)-\phi_{s_0}}\Omega.
 \label{eq:main-measure-transport}
\end{equation}
\end{thmenumerate}
\end{maintheorem}

The hypotheses in Theorem~\ref{thm:main} are essentially optimal.
$-K_X$ being pseudoeffective is the weakest assumption to produce singular positive metric, while the slice-wise
integrability is necessary to define $F$.

\vspace{0.6cm}

As applications of Theorem~\ref{thm:main}, we solve the uniqueness conjecture on the twisted K\"ahler-Einstein currents posed in Darvas--Zhang's important work
\cite{DZ24}.
Let
$(X,\omega)$ be a compact K\"ahler manifold, $\Omega$ be a smooth
volume form and $\theta$ be a smooth closed $(1,1)$-form representing
a big cohomology class
$[\theta]\in H^{1,1}(X,\mathbb{R})$.  Let
$\eta:=\operatorname{Ric}(\Omega)-\theta$ and assume there exists a
$\eta$-psh function $\psi$ so that
\[
 \eta_\psi
 :=\operatorname{Ric}(\Omega)-\theta+\ddc\psi\geq0,
 \qquad
 e^{-\psi}\in L^1(X,\Omega).
\]
Consider the $\lambda$-Ding functional $D^{\lambda}_\psi$ on the space of $\theta$-psh functions with finite energy:
\[
 D^{\lambda}_\psi(\varphi):=-\frac{1}{\lambda}\log\int_Xe^{-\lambda\varphi-\psi}\Omega-E_\theta(\varphi),\qquad \varphi\in\mathcal{E}^1(X,\theta).
\]
By strong openness theorem \cite{GZ15}, there exists $\lambda>0$ s.t. $D^{\lambda}_{\psi}$ is finite on $\mathcal{E}^1(X,\theta)$.
Darvas--Zhang proved that the delta invariant
$\delta_\psi([\theta])>1$ implies the properness of $D^1_{\psi}$. Consequently, they proved that  there exists a current
$T=\theta+dd^cu$ with minimal singularities satisfying the
$\eta_\psi$-twisted K\"ahler--Einstein equation:
\begin{equation}
 \operatorname{Ric}(T)=T+\eta_{\psi},
 \label{eq:twisted_KE}
\end{equation}
which is equivalent to the solvability of the following twisted
complex Monge--Amp\`ere equation:
\begin{equation}
 \langle(\theta+dd^cu)^n\rangle=e^{-u-\psi}\Omega.
 \label{eq:twisted_CMA}
\end{equation}
Conversely, Darvas--Zhang \cite[Proposition~5.9]{DZ24} proved that the unique solvability of
equation \eqref{eq:twisted_CMA} implies
$\delta_\psi([\theta])>1$
and raised the following conjectures going back to Berndtsson \cite[page 4]{Ber15b}.
\begin{conjecture}[Berndtsson, Darvas--Zhang]
\label{ques:BDZ-uniqueness}
\leavevmode
\begin{enumerate}[
    label=\textup{(\roman*)},
    leftmargin=*,
    itemsep=0.5em
]
    \item There exists a Bando--Mabuchi type uniqueness theorem
          for twisted K\"ahler--Einstein currents.

    \item If $\delta_{\psi}([\theta])>1$, the twisted complex
          Monge--Ampère equation \eqref{eq:twisted_CMA}
          admits a unique solution.
\end{enumerate}
\end{conjecture}

Dervan--Reboulet \cite{DR24} posed the following closely related problems.  Assume $-K_X$ is big, $\Omega$ is a smooth positive volume form with $\operatorname{Ric}(\Omega):=\theta$. 
Let $V_{\theta}:=\sup\{\varphi\in\operatorname{PSH}(X,\theta):\varphi\leq0\}$ be the psh upper envelope and assume it is klt, namely the multiplier ideal sheaf $\mathcal{I}(V_{\theta})=\mathcal{O}_X$.
\begin{problem}[Dervan-Reboulet]
    \label{ques:DR-strictly convex}
\leavevmode
\begin{enumerate}[
    label=\textup{(\roman*)},
    leftmargin=*,
    itemsep=0.5em
]
    \item The $1$-Ding functional $D^1(\varphi)=-\log\int_Xe^{-\varphi}\Omega-E_{\theta}(\varphi)$ is strictly convex along weak geodesics in $\mathcal{E}^1(X,\theta)$ in the absence of holomorphic vector fields.
    \item If $\operatorname{Aut}(X)$ is finite, the K\"ahler--Einstein current in $c_1(-K_X)$ is unique.
\end{enumerate}
\end{problem}
We observe that these problems are naturally related to the equality case of Berndtsson's log-integral convexity theorem \cite{Ber15b}. Theorem~\ref{thm:main} and its applications would give
positive answers to all of these conjectures and problems.  
\vspace{0.4cm}

We assume that
\[
 e^{-V_\theta-\psi}\in L^1(X,\Omega).
\]
A potential $\varphi\in\mathcal E^1(X,\theta)$ is called a normalized solution if it satisfies
\begin{equation}
 \left\langle(\theta+\ddc\varphi)^n\right\rangle
 =
 e^{-\varphi-\psi}\Omega.
 \label{eq:normalized-twisted-CMA}
\end{equation}
Note that the right-hand side is integrable by the strong openness theorem \cite{GZ15} (cf. \cite[Proposition 2.5]{DZ24}).
$T_\varphi:=\theta+\ddc\varphi$ is called the current of normalized solution and we denote the set of such currents by $\mathcal S(\theta,\psi)$. 

\begin{mainbandotheorem}
\label{thm:main-big-BM}
Let $\varphi_0,\varphi_1$ be two normalized solutions and $(\varphi_s)_{0\leq s\leq1}$ be the finite-energy weak geodesic joining them.
Then every $\varphi_s$ is a
normalized solution. Put $T_s:=T_{\varphi_s}=\theta+\ddc\varphi_s$, there exists
$\mathcal V\in H^0(X,T_X^{1,0})$ and its holomorphic flow
$\Lambda:\mathbb C\to\operatorname{Aut}^0(X)$, generated by
$\mathcal V/2$, such that
\begin{thmenumerate}
\item $\iota_{\mathcal V}\eta_{\psi}=0$ and
$\Lambda_z^*\eta_{\psi}=\eta_{\psi}$ for every $z\in\mathbb C$;
\item for $s,s_0\in[0,1]$,
\[
 \Lambda_{s-s_0}^*T_s=T_{s_0},
 \qquad
 \Lambda_{s-s_0}^*\langle T_s^n\rangle
 =\langle T_{s_0}^n\rangle;
\]
\item for $s\in[0,1]$ and $t\in\mathbb{R}$,
\[\Lambda_{it}^*T_s=T_s.\]
\end{thmenumerate}
\end{mainbandotheorem}

\begin{remark}
    Theorem~\ref{thm:main-big-BM} solves Conjecture~\ref{ques:BDZ-uniqueness} (i). Theorem~\ref{thm:main} and Theorem~\ref{thm:main-big-BM} solve Problem~\ref{ques:DR-strictly convex}, see Corollary~\ref{cor:DR-uniqueness}.
\end{remark}
\begin{remark}
\label{rem:lambda-scaling}
Theorem~\ref{thm:main-big-BM} applies to twisted K\"ahler-Einstein currents satisfying $\operatorname{Ric}(T_\varphi)=\lambda T_{\varphi}+\eta_\psi$ for arbitrary \(\lambda>0\) by standard scaling. 
Assume $[\theta]\in H^{1,1}(X,\mathbb{R})$ is big and $\psi$ is quasi-plurisubharmonic so that 
\begin{equation*}
\begin{split}
    & \eta_{\lambda}:=\operatorname{Ric}(\Omega)-\lambda\theta,\\
    &\eta_{\lambda,\psi}
 :=
 \operatorname{Ric}(\Omega)-\lambda\theta+\ddc\psi\geq0.
\end{split}
\end{equation*}
We consider the normalized equation: 
\[
 \left\langle(\theta+\ddc\varphi)^n\right\rangle
 =
 e^{-\lambda\varphi-\psi}\Omega,\qquad \varphi\in\mathcal{E}^1(X,\theta),
\]
and use $S_{\lambda}(\theta,\psi)$ to denote the set of currents of normalized solution.
Set
\[
 \theta_{\lambda}:=\lambda\theta,\qquad
 \varphi_{\lambda}:=\lambda\varphi,\qquad
 \psi_{\lambda}:=\psi-n\log\lambda.
\]
Then
$
 \eta_{1,\psi_{\lambda}}:=\operatorname{Ric}(\Omega)-\theta_{\lambda}
 +\ddc\psi_{\lambda}
 =\eta_{\lambda,\psi}$
and the above equation is equivalent to
\[
 \langle
 (\theta_{\lambda}+\ddc\varphi_{\lambda})^n
 \rangle
 =
 e^{-\varphi_{\lambda}-\psi_{\lambda}}\Omega.
\]
Theorem~\ref{thm:main-big-BM} applies to
\(S(\theta_{\lambda},\psi_{\lambda})\) and gives the corresponding
Bando--Mabuchi--type theorem for $S_{\lambda}(\theta,\psi)$.
\end{remark}
\vspace{0.4cm}
Theorem~\ref{thm:main-big-BM} requires no stability assumption: whenever
K\"ahler--Einstein currents exist, it describes all their nonuniqueness in terms of
holomorphic flows annihilating the twist.  Under the uniform stability
condition \(\delta_\psi([\theta])>1\), this orbit description can be upgraded to actual
uniqueness:

\begin{mainpropernesstheorem}
\label{thm:main-DZ-uniqueness}
Let $X$ be a connected compact K\"ahler manifold, let $\theta$ be a
smooth closed real $(1,1)$-form whose class is big, and let $\Omega$ be
a smooth positive volume form.  Suppose that $\psi$ is quasi-psh,
\[
 \eta_{\psi}:=\operatorname{Ric}(\Omega)-\theta+\ddc\psi\geq0,
 \qquad e^{-\psi}\in L^1(X,\Omega).
\]
Then the following are equivalent:
\begin{thmenumerate}
\item $\delta_\psi([\theta])>1$;
\item $\sup\bigl\{\lambda:D^{\lambda}_\psi \text{ is proper}\bigl\}>1$;
\item the equation
\[
 \left\langle(\theta+\ddc\varphi)^n\right\rangle
 =e^{-\varphi-\psi}\Omega
\]
admits a unique $\theta$-psh solution with minimal singularities.
\end{thmenumerate}
\end{mainpropernesstheorem}
Theorem~\ref{thm:main-DZ-uniqueness} solves Conjecture~\ref{ques:BDZ-uniqueness} (ii).

\vspace{0.4cm}
The orbit description also determines the infinitesimal symmetry of
the twisted equation.  Assume that
$\mathcal S:=\mathcal S(\theta,\psi)\ne\varnothing$, fix
$T\in\mathcal S$ and set
\[
 \widehat K_T:=\{g\in\operatorname{Aut}^0(X):g^*T=T\},
 \qquad
 \mathfrak g_\eta
 :=\{V\in H^0(X,T_X^{1,0}):\iota_V\eta=0\},
\]
where $\operatorname{Aut}^0(X)$ is the connected component of the identity element in the automorphism group $\operatorname{Aut}(X)$. Let
\[
 \mathfrak k_T^{\mathrm{null}}
 :=\mathfrak g_\eta\cap
   \operatorname{Lie}_{\mathbb R}(\widehat K_T).
\]

\begin{mainmatsushimatheorem}
\label{thm:main-big-Matsushima}
The Lie algebra $\mathfrak k_T^{\mathrm{null}}$ is a compact real form
of $\mathfrak g_\eta$.  More precisely,
\begin{equation}
 \mathfrak g_\eta=\mathfrak k_T^{\mathrm{null}}
          \oplus i\mathfrak k_T^{\mathrm{null}}
 \label{eq:null-Lie-real-form}
\end{equation}
as a direct sum of real vector spaces.  Equivalently,
\[
 \mathfrak g_\eta\simeq
 \mathfrak k_T^{\mathrm{null}}\otimes_{\mathbb R}\mathbb C
\]
as complex Lie algebras.  In particular, $\mathfrak g_\eta$ is
reductive.
\end{mainmatsushimatheorem}

When $\eta=0$, the null algebra is the full algebra of holomorphic
vector fields, and Theorem~\ref{thm:main-big-Matsushima} recovers the
classical Matsushima theorem \cite{Mat57}.

\vspace{0.6cm}

We now briefly outline the proofs.  The two essential analytic inputs
in the proof of Theorem~\ref{thm:main} are openness theorem \cite{Ber15a, GZ15} and a refined version of Berndtsson's approach \cite{Ber15b}.

For an affine subgeodesic \(\phi=(\phi_s)\) in
Theorem~\ref{thm:main}, the openness theorem yields that on
every \(J\Subset I\),
\[
 e^{-p\phi}\in L^1(X\times J)
 \qquad\text{for some }p>1.
\]
Together with the convexity of \(\phi\) on the slices
\(\{x\}\times I\), the exponent \(p>1\) further yields the kinetic
estimate
\[
 |\dot Q|^2e^{-Q}\in L^1(X\times J).
\]
The exponent $p>1$ is also crucial to obtain good convergence properties
for the regularization sequence \(q_j\) of \(Q\).

For each \(q_j\), we solve the auxiliary weighted \(D'\)-equation in
the fixed exact range, taking its \(L^2\)-minimal solution \(v_j\).
The uniform \(L^2\)-estimate in terms of
\(\bigl((q_j)_\tau u\bigr)^{\perp_{j,s}}\) follows from the dual form
of \cite[Lemma~6.3]{Ber15b}, ultimately based on H\"ormander's
\(L^2\)-estimates \cite{Hor65,Hor90}.  Retaining this exact-range
normalization in the limiting process is one ingredient that removes
the assumption \(H^{0,1}(X)=0\).

The estimate in
Lemma~\ref{lem:weighted-graph-no-concentration}, which again relies
essentially on H\"ormander's \(L^2\)-estimates
\cite{Hor65,Hor90}, shows that the \(L^2\)-mass of the
\(D'\)-solutions \(v_j\) cannot concentrate near the upper-level sets
of the Lelong number, thereby allowing the curvature loss in
\(\ddc q_j\) to be controlled in the limit.  The compactness of
\(v_je^{-q_j/2}\) in \(L^2\) space and the estimates of
\(\|\bar\partial_Xv_j\|_{L^2}\) produce holomorphic vector fields
\(V_s\) on almost every fiber \(X_s\).

In Berndtsson's approach, the vanishing of \(H^{0,1}(X)\) is used
both in the construction of \(v_j\) and, more substantially, to turn
certain \(\bar\partial\)-closed forms into
\(\bar\partial\)-exact ones when establishing holomorphic dependence
on the parameter.  Instead of using \(H^{0,1}(X)=0\), we use a family
of exact test forms in Proposition~\ref{prop:smooth-exact-tests} to
show that our holomorphic vector fields are actually independent of
the parameter \(s\in I\), hence become a single vector field
\(\mathcal V\) on \(X\).
Some Lie derivative calculations of currents give nullity equation \eqref{eq:main-full-nullity} and the transport statements in Theorem~\ref{thm:main}.

\vspace{0.3cm}
For Theorem~\ref{thm:main-big-BM}, the finite-energy geodesic between
two solutions consists of minimizers of the $1$-Ding functional.
Applying Theorem~\ref{thm:main} to this geodesic yields the orbit
theorem.  For Theorem~\ref{thm:main-DZ-uniqueness}, let $(\varphi_s)_{s\in[0,1]}$ be the weak geodesic joining two solutions.
For
almost every $s$, the velocity $\dot\varphi_s$ generates
auxiliary geodesic line consisting of normalized solutions. $\delta_\psi([\theta])>1$ implies the properness of $1$-Ding functional, which
forces this geodesic line to have zero \(d_1\)-speed. 
Consequently $\varphi_s$ is constant.  Finally,
Theorem~\ref{thm:main-big-Matsushima} follows from compactness of the
current stabilizer and a compact-real-form argument.

\vspace{0.3cm}
The organization of the remaining part of this paper is as follows.
Notations and preliminaries are contained in Section~\ref{sec:setup}.
Preparations of the proof of Theorem~\ref{thm:main} occupies
Sections~\ref{sec:pass1}--\ref{sec:smooth-exact-tests} and the solution of Berndtsson's problem is given in Section~\ref{sec:full-current-nullity}. 
Three applications including proofs of Theorems~\ref{thm:main-big-BM}, ~\ref{thm:main-DZ-uniqueness}, ~\ref{thm:main-big-Matsushima} are provided in
Sections~\ref{sec:twisted-data}, ~\ref{sec:DZ-properness-uniqueness}, ~\ref{sec:null-automorphism-group} respectively.


\subsection*{Acknowledgements}
This research is supported by the National Key R\&D Program of China
(Grant Nos.~2021YFA1002600 and 2021YFA1003100). Z.~Wang and X.~Zhou
are partially supported by grants from the National Natural Science
Foundation of China (NSFC) (Nos.~12571085 and 12288201), respectively.
Z.~Wang is also supported by the Fundamental Research Funds for the
Central Universities. 

\section{Notations and preliminaries}
\label{sec:setup}

\subsection{Pluripotential conventions and subgeodesics}

Throughout this paper, $X$ denotes a connected compact K\"ahler manifold of complex
dimension $n$, $\omega$ is a fixed K\"ahler form, and
$dV_\omega:=\omega^n/n!$.  We use the normalization
\[
 d^c=\frac{i}{2}(\bar\partial-\partial),
 \qquad
 \ddc=i\partial\bar\partial.
\]
If $\Omega$ is a smooth positive volume form and, in local
holomorphic coordinates,
\[
 \Omega=c_n f\,dz^1\wedge\cdots\wedge dz^n
 \wedge d\bar z^1\wedge\cdots\wedge d\bar z^n,
 \qquad c_n=i^{n^2},
\]
we set
\[
 \operatorname{Ric}(\Omega):=-\ddc\log f.
\]
This is a globally defined smooth closed real $(1,1)$-form.  We omit
the factor $2\pi$, so it represents $c_1(-K_X)$.

For a smooth closed real $(1,1)$-form $\theta$ on $X$, we write
\[
 \operatorname{PSH}(X,\theta)
 :=\bigl\{u  \text{ is quasi-plurisubharmonic}:
                   \theta+\ddc u\geq0\bigr\}.
\]
The class $[\theta]$ is pseudoeffective precisely when
$\operatorname{PSH}(X,\theta)\ne\varnothing$.  

Let $I\subset\mathbb R$ be an open interval and set
\[
 S_I:=I+i\mathbb R,
 \qquad \tau=s+it,
 \qquad p_X:X\times S_I\longrightarrow X.
\]
Our complex derivative convention is
\[
 \frac{\partial}{\partial\tau}
 =\frac12\left(\frac{\partial}{\partial s}
              -i\frac{\partial}{\partial t}\right),
 \qquad
 \frac{\partial}{\partial\bar\tau}
 =\frac12\left(\frac{\partial}{\partial s}
              +i\frac{\partial}{\partial t}\right).
\]
Thus, for a function independent of $t$,
$q_\tau=q_{\bar\tau}=\dot q/2$ whenever the derivatives exist.
When no confusion is possible, we write $S=S_I$.

\begin{definition}
\label{def:subgeodesic}
A family
$(\phi_s)_{s\in I}\subset\operatorname{PSH}(X,\theta)$ is a
$\theta$-psh subgeodesic if 
\[
 \phi(x,\tau):=\phi_{\operatorname{Re}\tau}(x)
\]
is a $p_X^*\theta$-plurisubharmonic function on $X\times S_I$.
\end{definition}
For each $x$ such that
$s\mapsto\phi(x,s)$ is not identically $-\infty$, the function is convex; we use $\dot \phi_s$ to denote its almost-everywhere derivative.  

For a function $P$ on $X\times S_I$, we write
$P_s:=P(\cdot,s)$; for a sequence of functions $P_j$ on $X\times S_I$,
we write $P_{j,s}:=P_j(\cdot,s)$.  

Let $(\phi_s)$ be the subgeodesic in Theorem~\ref{thm:main} and set
\begin{equation}
 F(s):=-\log\int_Xe^{-\phi_s}\Omega,
 \qquad
 Q:=\phi-F,
 \qquad
 \mu_s:=e^{-Q_s}\Omega.
 \label{eq:normalized-data}
\end{equation}
We extend $F$ to $X\times S_I$ trivially by $F(s+it):=F(s)$. Since $F$ is affine, $\ddc_\tau F=0$, so $Q$ is again a
$\theta$-psh subgeodesic and
\begin{equation}
 \int_Xd\mu_s=1,
 \qquad
 \Theta_Q:=p_X^*\theta+\ddc_{x,\tau}Q
           =p_X^*\theta+\ddc_{x,\tau}\phi\geq0.
 \label{eq:normalized-curvature}
\end{equation}

\subsection{The anticanonical calculus under Chern connections}

Put $L:=-K_X$ and let $h_\Omega$ be the smooth Hermitian metric on
$L$ induced by $\Omega$.  Its Chern curvature is denoted by
\[
 \theta=\operatorname{Ric}(\Omega).
\]
The canonical identity section
\[
 u\in H^0(X,K_X\otimes L)=H^0(X,\mathcal O_X)
\]
is normalized by
\[
 c_n[u,u]_{h_\Omega}=\Omega.
\]
The bracket $[\,\cdot,\cdot\,]_{h_\Omega}$ means that the
$L$-components are paired using $h_\Omega$ and the scalar form
components are wedged.

Write $D'_X:=D'_{h_\Omega,X}$, the $(1,0)$-part of the Chern connection of $h_{\Omega}$.  For a smooth real function $q$ on
$X$, let
\[
 D'_{q,X}:=D'_X-\partial_Xq\wedge
\]
denote the $(1,0)$-part of the Chern connection of
$h_\Omega e^{-q}$.  
For an $L$-valued differential form $\alpha$, the
weighted norm of $\alpha$ is defined as
\[
 \|\alpha\|_q^2
 :=\int_X|\alpha|_{\omega,h_\Omega}^2e^{-q}\,dV_\omega,
\]
and we use $\bar\partial_q^*$ to denote the Hilbert adjoint of $\bar\partial$
for the metric $(\omega,h_\Omega e^{-q})$.  When $q=0$, the subscript
is omitted. Furthermore, if $\alpha$ (resp. $\beta$) is a $L$-valued $(p,q)$-form (resp. $(p+1,q)$-form) on $X$ with $L^1$-coefficients, we say $D'_{q,X}\alpha=\beta$ in the sense of distribution if $\langle \alpha,(D'_{q,X})^*\gamma\rangle=\langle\beta,\gamma\rangle$ for every smooth $L$-valued $(p+1,q)$-form $\gamma$.

For a smooth real function $Q$ on $X\times S_I$, the $(1,0)$-part of the Chern connection of the metric
$p_X^*h_\Omega e^{-Q}$ on $p_X^*L$ is
\[
 D'_Q=D'_{Q,X}+D'_{Q,S},
 \qquad
 D'_{Q,X}=D'_{h_\Omega,X}-\partial_XQ\wedge,
 \qquad
 D'_{Q,S}=\partial_S-\partial_SQ\wedge.
\]

Contraction with the identity section $u$ identifies $(1,0)$-vector fields with
$L$-valued $(n-1,0)$-forms: every smooth $(1,0)$-vector field $V$ corresponds to a smooth $L$-valued $(n-1,0)$-form $v$ by
\[
 v=-\iota_{V}u.
\]
The Lefschetz map $L_{\omega}:\Lambda^{n-1,0}T^*X\to \Lambda^{n,1} T^*X$ is an isometry and we have
\begin{equation}
 \left\langle \omega\wedge v_1,
                 \omega\wedge v_2\right\rangle_{\omega,h_\Omega}
 dV_\omega
 =\langle v_1,v_2\rangle_{\omega,h_\Omega}dV_\omega
 =\langle V_1,V_2\rangle_\omega\Omega.
 \label{eq:contraction-identity}
\end{equation}

\vspace{0.3cm}
\section{Integrability and kinetic estimates}
\label{sec:pass1}

\subsection{\texorpdfstring{Integrability and kinetic estimates of $Q$}{Ingegrability and kinetic estimates of Q}}

The first lemma collects the integrability and kinetic estimates of $Q$, where the openness theorem \cite{GZ15,Ber15a} plays the key role. 

\begin{lemma}\label{lem:kinetic}
Let $J\Subset I:=(0,1)$ be an open interval.  Then the following
statements hold.
\begin{enumerate}
\item  There is a number $p>1$ depending on $J$ such that
      \begin{equation}
       \int_{J}\!\int_X e^{-pQ_s}\,\Omega\,ds<\infty.
       \label{eq:kinetic-strong-openness}
      \end{equation}
\item There is a complete pluripolar set $P\subset X$ so that for every $x\in X\backslash P$, $s\mapsto Q_s(x)$
      is convex on $I$.
      If $\dot Q_s(x):=\frac{\partial}{\partial s}Q(x,s+it)$, then
      \begin{equation}
       \int_{J}\!\int_X |\dot Q_s|^2e^{-Q_s}\,\Omega\,ds<\infty.
       \label{eq:canonical-kinetic-bound}
      \end{equation}
\item The following chain rule holds a.e. and in the sense of distributions:
      \begin{equation}
       \partial_s(e^{-Q})=-\dot Qe^{-Q}.
       \label{eq:canonical-chain-rule}
      \end{equation}
      Moreover, if
      $s<t$, $s,t\in I$, then
      \begin{equation}
       \|e^{-Q_t}-e^{-Q_s}\|_{L^1(X,\Omega)}
       \leq
       \int_s^t\!\int_X|\dot Q_r|e^{-Q_r}\,\Omega\,dr.
       \label{eq:density-ac-estimate}
      \end{equation}
      In particular, $s\mapsto e^{-Q_s} $ is continuous in $L^1(X,\Omega)$.
\item For almost every $s\in I$,
      \begin{equation}
       \int_X\dot Q_s e^{-Q_s}\,\Omega=0.
       \label{eq:centered-limit-speed}
      \end{equation}
\end{enumerate}
\end{lemma}

For a function $P$ resp. a sequence of functions $P_j$ on $X\times S$, we will use $P_s$ resp. $P_{j,s}$ to denote the slice $P(\cdot,s)$ resp. $P_j(\cdot,s)$.
\begin{proof}
The normalization $\int_Xe^{-Q_s}\Omega=1$ and Tonelli's theorem give
\[
 \int_{J}\!\int_Xe^{-Q_s}\,\Omega\,ds=|J|<\infty.
\]
Fix a point of $X\times S$.  Let $\rho_X$ be a local potential of $\theta$, namely
$\ddc\rho_X=\theta$.  Then $Q+p_X^*\rho_X$ is plurisubharmonic.  Since $\rho_X$
is smooth, we get that
	$e^{-(Q+p_X^*\rho_X)}$ is
locally integrable.  The openness theorem \cite{GZ15,Ber15a}, applied to the plurisubharmonic function $Q+p_X^*\rho_X$, gives a number
$p>1$ so that $e^{-pQ}$ is locally integrable.
Since $J$ is relatively compact and $Q$ is $\mathrm{Im}\tau$-invariant, \eqref{eq:kinetic-strong-openness} is proved.

\vspace{0.3cm}
Fix $x\in X$.
The restriction of $Q$ to $\{x\}\times S$ is either finite and convex or identically $-\infty$. Hence
$P:=\{x\in X:Q(x,\tau)=-\infty\}$ is independent of $\tau$.
The function $\dot Q_+(x,s):=\lim_{t\rightarrow0^+}\frac{Q(x,s+t)-Q(x,s)}{t}$ is finite on $(X\backslash P)\times S$ and is the limit of a decreasing sequence of measurable functions, hence is measurable.
The same property holds for $\dot Q_-(x,s):=\lim_{t\rightarrow0^-}\frac{Q(x,s+t)-Q(x,s)}{t}$.
Since $Q(x,\cdot)$ is convex, we know $\{s:\dot Q_+(x,s)\neq \dot Q_-(x,s)\}$ is countable for every $x$. Furthermore, the set $\{(x,s):\dot Q_+(x,s)\neq \dot Q_-(x,s)\}$ is measurable and Fubini's theorem then gives an a.e. well-defined measurable function $\dot Q$ on $X\times S$.

Choose $J'\Subset I$ and $d>0$ so that $J\pm d\subset J'$.
We choose
\[
 C_1:=1+\sup_{X\times \overline{J'}}Q.
\]
For $x\in X\backslash P$ and $s\in J$ such that $Q(x,\cdot)$ is differentiable, we have
\begin{equation}
 \frac{Q_s(x)-Q_{s-d}(x)}d
 \leq \dot Q_s(x)
 \leq \frac{Q_{s+d}(x)-Q_s(x)}d.
 \label{eq:secants-before-bound}
\end{equation}
Since $Q_{s-d}(x),Q_{s+d}(x)\leq C_1$, the two
inequalities become
\[
 -\frac{C_1-Q_s(x)}d
 \leq\dot Q_s(x)
 \leq\frac{C_1-Q_s(x)}d,
\]
and hence
\begin{equation}
 |\dot Q_s(x)|\leq d^{-1}(C_1-Q_s(x)).
 \label{eq:canonical-secant-bound}
\end{equation}

For every $y\leq C_1$,
\[
 (C_1-y)^2e^{-y}
 =e^{-py}(C_1-y)^2e^{(p-1)y}
 \leq C_{p,C_1}e^{-py},
\]
Combining this inequality with
\eqref{eq:canonical-secant-bound} and
\eqref{eq:kinetic-strong-openness} proves
\eqref{eq:canonical-kinetic-bound}.  It also gives
\begin{equation}
 \int_{J}\!\int_X|\dot Q_s|e^{-Q_s}\,\Omega\,ds<\infty
 \label{eq:canonical-first-order-integrability}
\end{equation}
by the Cauchy--Schwarz inequality.

\vspace{0.3cm}
We move on to prove the chain rule.  For every \(x\in X\setminus P\),
the finite convex function \(Q(x,\cdot)\) is locally Lipschitz on
\(I\).  Hence \(s\mapsto e^{-Q_s(x)}\) is locally absolutely
continuous and
\[
 \frac{d}{ds}e^{-Q_s(x)}
 =
 -\dot Q_s(x)e^{-Q_s(x)}
\]
for almost every \(s\in I\).  Let
\(\psi\in C_c^\infty(X\times J)\).  For every \(x\in X\setminus P\),
one-dimensional integration by parts gives
\[
 -\int_J e^{-Q_s(x)}\partial_s\psi(x,s)\,ds
 =
 \int_J-\dot Q_s(x)e^{-Q_s(x)}\psi(x,s)\,ds.
\]
Both sides are integrable with respect to \(x\), by the normalization
and \eqref{eq:canonical-first-order-integrability}.  Since \(P\) is
pluripolar and hence \(\Omega\)-negligible, Fubini's theorem yields
\begin{align*}
 \left\langle\partial_s(e^{-Q}),\psi\right\rangle
 &=
 -\int_{X\times J}
 e^{-Q}\partial_s\psi\,\Omega\,ds\\
 &=
 -\int_{X\times J}
 \dot Qe^{-Q}\psi\,\Omega\,ds.
\end{align*}
Thus $\partial_s(e^{-Q})=-\dot Qe^{-Q}$ in the sense of distributions on \(X\times J\).

\vspace{0.3cm}
Now we prove \eqref{eq:density-ac-estimate}:
\begin{align*}
    \left\|e^{-Q_t}-e^{-Q_s}\right\|_{L^1(X,\Omega)}=\int_X\left|\int_s^t-\dot Q_r e^{-Q_r}dr\right|\Omega
    \leq
    \int_X\int_s^t\left|\dot Q_r e^{-Q_r}\right|dr\Omega.
\end{align*}
It follows from the fact that $s\mapsto \int_X|\dot Qe^{-Q}|\Omega$ is integrable that $s\mapsto e^{-Q_s}\in L^1(X,\Omega)$ is continuous.

\vspace{0.3cm}
Finally, we prove that, for almost every $s$,
\begin{align*}
    \frac{d}{ds}\int_Xe^{-Q_s}\Omega=-\int_X\dot Q_s e^{-Q_s}\Omega.
\end{align*}
In fact, for $t\neq0$ \eqref{eq:canonical-chain-rule} implies that
\begin{align*}
	    \frac{1}{t}\int_X\bigl(e^{-Q_{s+t}}-e^{-Q_s}\bigr)\Omega
    =\frac{1}{t}\int_0^t\int_X-\dot Q_{s+r} e^{-Q_{s+r}}\Omega dr
\end{align*}
The function $s\mapsto \int_X\dot Q_se^{-Q_s}\Omega$ is locally integrable.
This implies that, for almost every $s$, we have $\lim_{t\rightarrow 0}\frac{1}{t}\int_0^t\int_X\dot Q_{s+r}e^{-Q_{s+r}}\Omega\,dr=\int_X\dot Q_s e^{-Q_s}\Omega$.
On the other hand, $\int_Xe^{-Q_{s+t}}\Omega=\int_Xe^{-Q_s}\Omega$ for every $t$; we therefore infer \eqref{eq:centered-limit-speed}.
\end{proof}

\vspace{0.4cm}
\subsection{\texorpdfstring{Uniform integrability and kinetic estimates of regularization sequence of $Q$}{Uniform integrability and kinetic estimates of regularization sequence of Q}}
\label{sec:pass2}

For a bounded interval $J\Subset I$, define an annulus $A_J:=\{\zeta\in\mathbb{C}^*: \log|\zeta|\in J\}$.
The map $\zeta=e^\tau$ identifies the quotient of
$J+i\mathbb R$ by $2\pi i\mathbb Z$ with $A_J$.  Since $Q$ is independent
of $\operatorname{Im}\tau$, it descends to an $S^1$-invariant function on
$X\times A_J$, which we still denote by $Q$ for simplicity.  
We use this quotient only to work on relatively compact
cylinders and apply Demailly's approximation theorem. 


Fix, in this section, three intervals
$J_0\Subset J_1\Subset J_2\Subset I$, and put
$Y_k:=X\times A_{J_k}$, $k=0,1,2$.
Let \(\zeta\) denote the coordinate on the annulus.  Using the fixed
K\"ahler form \(\omega\) on \(X\), equip \(Y_k\) with the
\(S^1\)-invariant product K\"ahler form
\[
 \widehat\omega
 :=
 p_X^*\omega
 +
 i\,\frac{d\zeta\wedge d\bar\zeta}{|\zeta|^2}.
\]
The metric $i\,\frac{d\zeta\wedge d\bar\zeta}{|\zeta|^2}$ on $A_{J_k}$ comes from the Euclidean metric $id\tau\wedge d\bar\tau$ on $S$.



\begin{lemma}[$S^1$-equivariant Demailly regularization]
\label{lem:relative-demailly-smoothing}
Let $T:=\Theta_Q=p_X^*\theta+\ddc Q$ be a closed positive current. There are
constants $c_*,C_0>0$, depending only on $T$ and $(X,\omega)$,
respectively; a sequence $\varepsilon_j\searrow0$; and
$S^1$-invariant functions
\[
 q_j\in C^\infty(Y_1),\qquad
 \ell_j\in C^0(Y_1,[0,c_*]),
\]
such that
\begin{align}
 q_j&\searrow Q &&\text{pointwise on }Y_1,
 \label{eq:demailly-monotone}\\
 \Theta_{q_j}:=p_X^*\theta+\ddc q_j
 &\geq-\ell_jC_0\widehat{\omega}-\varepsilon_j\widehat\omega,
 \label{eq:demailly-lelong-loss}\\
 \ell_j(y)&\searrow\nu(T,y)
 &&\text{for every }y\in Y_1.
 \label{eq:lelong-loss-limit}
\end{align}
In particular, for a constant $C$ independent of $j$,
\begin{equation}
 \Theta_{q_j}\geq-C\widehat\omega
 \quad\text{on }Y_1.
 \label{eq:uniform-semipositivity-loss}
\end{equation}
\end{lemma}

\begin{proof}

Choose a number $c_*>\sup_{y\in\overline{Y_2}}\nu(T,y)$.
Applying Demailly's regularization theorem on a neighborhood of $\overline{Y_1}$ contained in $Y_2$ in \cite{Dem92}, 
we obtain a constant $C_0>0$, smooth functions $q_j^0\searrow Q$, continuous functions
$\lambda_j\searrow\nu(T,\cdot)$, and numbers
$\varepsilon_j\downarrow0$, with
\begin{equation}
 p_X^*\theta+\ddc q_j^0
 \geq-\lambda_jC_0\widehat\omega
       -\varepsilon_j\widehat\omega \qquad \mathrm{on} \ Y_1.
 \label{eq:demailly-before-average}
\end{equation}

 Define $\widetilde\lambda_j:=\min\{\lambda_j,c_*\}$ and set
\begin{align*}
 q_j(y)
 &:=
 \frac1{2\pi}\int_0^{2\pi}
 q_j^0(e^{i\theta} y)\,d\theta,\\
 \ell_j(y)
 &:=
 \frac1{2\pi}\int_0^{2\pi}
 \widetilde\lambda_j(e^{i\theta} y)\,d\theta .
\end{align*}
Then we see \(q_j\) and \(\ell_j\) are \(S^1\)-invariant.  
It is easy to check the desired properties are valid.
\end{proof}

For later reference, set
\begin{equation}
 \Upsilon:=C_0\widehat{\omega}, \qquad
 R_j:=\ell_j\Upsilon+\varepsilon_j\widehat\omega,
 \qquad S_j:=\Theta_{q_j}+R_j\geq0.
 \label{eq:shifted-smooth-curvature}
\end{equation}

We need a simple lemma.

\begin{lemma}
\label{lem:vertical-lelong-strata}
For every $a>0$ there is a proper analytic subset $Z_a\subsetneq X$ such that
\begin{equation}
 E_a(T)\cap Y_1=Z_a\times A_{J_1}.
 \label{eq:vertical-lelong-set}
\end{equation}
Moreover, if $K\Subset A_{J_1}$ and $U\subset X$ is an open
neighborhood of $Z_{a/2}$, then, for sufficiently large $j$,
\begin{equation}
 \ell_j<a
 \quad\text{on }(X\setminus U)\times K.
 \label{eq:uniform-loss-away-from-stratum}
\end{equation}
\end{lemma}

\begin{proof}
Siu's analyticity theorem~\cite{Siu74} gives that $E_a(T)$ is analytic.  Fix $x\in X$.  The intersection of
$E_a(T)$ with $\{x\}\times A_{J_1}$ is an analytic
subset of the annulus and is rotation-invariant.  If it is nonempty, it must be the
whole annulus. 
Fixing any $\zeta_0\in A_{J_1}$ and setting
$Z_a:=E_a(T)\cap (X\times\{\zeta_0\})$
proves \eqref{eq:vertical-lelong-set}.

On the compact set $(X\setminus U)\times K$ one has
$\nu(T,\cdot)<a/2$. Since $\ell_j\searrow\nu(T,\cdot)$, 
\eqref{eq:uniform-loss-away-from-stratum} follows immediately.
\end{proof}

\vspace{0.4cm}
The $S^1$-invariant approximation sequence on $X\times A_{J_1}$ corresponds to an $\operatorname{Im}\tau$-invariant approximation sequence on $X\times S_1$, where $S_1=J_1\times i\mathbb{R}$.
We now prove some convergence properties of this sequence $q_j$. 

\begin{proposition}
\label{prop:smooth-demailly-package}
For intervals $J_0\Subset J_1$, $q_j$ constructed in
Lemma~\ref{lem:relative-demailly-smoothing} satisfy
\begin{align}
 q_j&\longrightarrow Q
 &&\text{strongly in }L^1(X\times J_0),
 \label{eq:qj-L1}\\
 e^{-q_j}&\longrightarrow e^{-Q}
 &&\text{strongly in }L^1(X\times J_0),
 \label{eq:density-L1}\\
 e^{-q_j/2}&\longrightarrow e^{-Q/2}
 &&\text{strongly in }L^2(X\times J_0),
 \label{eq:half-density-L2}\\
 (q_j)_\tau e^{-q_j/2}&\longrightarrow
 Q_\tau e^{-Q/2}
 &&\text{strongly in }L^2(X\times J_0).
 \label{eq:smooth-kinetic-convergence}
\end{align}
We use the notation $(q_j)_\tau:=\dot q_j/2$.

Set $M_j(s):=\int_Xe^{-q_{j,s}}\Omega$ and
$F_j(s):=-\log M_j(s)$.
Then
\begin{equation}
 0\leq F_j\searrow0
 \quad\text{uniformly on } J_0.
 \label{eq:Fj-uniform}
\end{equation}
For sufficiently large $j$, $1/2\leq M_j\leq1$ and for any $q>1$,
\begin{equation}
 \dot F_j\longrightarrow0
 \quad\text{strongly in }L^{q}(J_0).
 \label{eq:Fj-speed-L2}
\end{equation}
Consequently, if $r_j:=(q_j)_\tau-(F_j)_\tau$, then
\begin{align}
 \int_Xr_{j,s}e^{-q_{j,s}}\Omega&=0, \forall s\in J_0
 \label{eq:centered-smooth-identity}\\
 r_je^{-q_j/2}&\longrightarrow Q_\tau e^{-Q/2}
 \quad\text{strongly in }L^2(X\times J_0).
 \label{eq:centered-kinetic-convergence}
\end{align}
\end{proposition}

\begin{proof}
Because $q_j\searrow Q$, one has
$0\leq q_j-Q\leq q_1-Q$.  Dominated convergence
therefore proves \eqref{eq:qj-L1}.  The same monotonicity gives
$0\leq e^{-q_j}\uparrow e^{-Q}$, and the latter is integrable by
Lemma~\ref{lem:kinetic}; monotone convergence proves
\eqref{eq:density-L1}.  
Since $(\sqrt a-\sqrt b)^2\leq|a-b|$ for $a,b\geq0$,
\eqref{eq:half-density-L2} follows immediately.

\vspace{0.4cm}
We now prove the kinetic convergence \eqref{eq:smooth-kinetic-convergence}.  The uniform lower bound
\eqref{eq:uniform-semipositivity-loss}
gives a constant $\kappa$ such that
$\frac{d^2}{ds^2}q_j(x,s)\geq-2\kappa$ on $X\times J_1$.  
Put $g_{j,x}(s):=q_j(x,s)+\kappa s^2$.
Then every $g_{j,x}$ is convex.  Moreover,
$g_{j,x}\searrow g_x$, where $g_x(s):=Q(x,s)+\kappa s^2$, pointwise.  Let $d>0$ be so small
that $J_0\pm d\subset J_1$.  Since $q_j\leq q_1$ is uniformly bounded from above on $X\times J_1$, the same argument as in Lemma~\ref{lem:kinetic} implies that there is a constant $B$ depending only on $\kappa,d,\sup_{X\times J_1}q_1$ such that 
\begin{equation}
 |\dot q_j(x,s)|
 \leq d^{-1}(B-q_j(x,s))
 \leq d^{-1}(B-Q(x,s))
 \label{eq:smooth-secant-bound}
\end{equation}
for all $(x,s)\in X\times J_0$; 

$g_x(s)$ is $s$-differentiable for almost every $(x,s)$. It is easy to see $\dot g_j(s)\to\dot g(s)$.  Indeed, for
$h>0$ satisfying $s\pm h\in J_1$ it holds that
\[
 \frac{g_j(s)-g_j(s-h)}h
 \leq \dot g_j(s)
 \leq\frac{g_j(s+h)-g_j(s)}h.
\]
Letting $j\to\infty$ and then $h\to0$, we obtain $\dot g_j(s)\to\dot g(s)$ as desired. Therefore $\dot q_j(x,s)\to\dot Q(x,s)$ for almost every $(x,s)$.

Let $p>1$ be supplied by \eqref{eq:kinetic-strong-openness}.
We have $|\dot q_j|^2e^{-q_j}\leq d^{-2}(B-Q)^2e^{-Q}\leq C e^{-pQ}$ on $X\times J_0$ for some uniform constant $C$.
The dominated convergence theorem implies
\begin{align*}
    \lim_{j\rightarrow\infty}\int_{X\times J_0}\Big|\dot q_je^{-\frac{q_j}{2}}-\dot Q e^{-\frac{Q}{2}}\Big|^2\Omega ds=0.
\end{align*}
This is \eqref{eq:smooth-kinetic-convergence}.

\vspace{0.3cm}
For every $s$, monotone convergence gives
$M_j(s)\nearrow1$, because $q_{j,s}\searrow Q_s$ and
$\int_Xe^{-Q_s}\Omega=1$.  By definition we get $F_j(s)\searrow0$.  Each $F_j$ is
smooth, and Dini's theorem gives the locally uniform convergence
\eqref{eq:Fj-uniform}.  

\vspace{0.3cm}
We also verify the speed statement, 
Put $z_j:=\dot q_je^{-q_j/2}$,
$h_j:=e^{-q_j/2}$, $z:=\dot Qe^{-Q/2}$, and $h:=e^{-Q/2}$.
It follows that
\begin{equation}
 \dot M_j=-\int_Xz_jh_j\,\Omega,
 \qquad
 \dot F_j=M_j^{-1}\int_Xz_jh_j\,\Omega.
 \label{eq:normalizer-derivative}
\end{equation}
By \eqref{eq:centered-limit-speed},
$\int_Xzh\,\Omega=0$ for almost every $s$.  The strong $L^2$
convergences \eqref{eq:half-density-L2} and
\eqref{eq:smooth-kinetic-convergence}, followed by Cauchy--Schwarz on
$X\times J_0$, therefore show that
\[
 \int_{J_0}\left|\int_Xz_jh_j\,\Omega\right|ds=\int_{J_0}\Big|\int_X(z_jh_j-zh)\Omega\Big|ds\longrightarrow0.
\]
This proves $\dot F_j\to0$ in $L^1(J_0)$.

By Jensen's
inequality for the probability measure
$M_j^{-1}e^{-q_{j,s}}\Omega$, we get for every $a>1$,
\[
 |\dot F_j(s)|^{a}
 =\Big|\frac{1}{M_j(s)}\int_X \dot q_{j,s} e^{-q_{j,s}}\Omega \Big|^{a}
 \leq \frac{1}{M_j(s)}\int_X|\dot q_{j,s}|^{a}e^{-q_{j,s}}\,\Omega.
\]
The right-hand side is uniformly integrable in $s$: combine
\eqref{eq:smooth-secant-bound} with
$(B-y)^{a}e^{-y}\leq C e^{-py}$ and $q_j\geq Q$.
Hence the sequence $\dot F_j$ is bounded in $L^{a}(J_0)$ and converges to zero
in $L^1(J_0)$.  Since $a$ is arbitrary, interpolation inequality gives
\eqref{eq:Fj-speed-L2}.

\vspace{0.3cm}
Finally, \eqref{eq:normalizer-derivative}, with
$\partial_\tau=\frac12\partial_s$, gives
\[
 (F_j)_\tau M_j
 =\int_X(q_j)_\tau e^{-q_j}\Omega,
\]
which is exactly \eqref{eq:centered-smooth-identity}.  Moreover,
\[
 \|(F_j)_\tau e^{-q_j/2}\|_{L^2(X\times J_0)}^2
 =\int_{J_0}|(F_j)_\tau|^2M_j\,ds\longrightarrow0.
\]
Combining this with \eqref{eq:smooth-kinetic-convergence} proves
\eqref{eq:centered-kinetic-convergence}.
\end{proof}

\vspace{0.4cm}
\section{\texorpdfstring{The $D^\prime$-equations and uniform $L^2$ estimates}{The D' equations and uniform L2 estimates}}
\label{sec:pass3}

\subsection{\texorpdfstring{Solve $D^\prime$-equations with smooth weights $q_j$}{Solve D' equations with smooth weights qj}}
In this section the weights $q_j$ are those constructed in Lemma~\ref{lem:relative-demailly-smoothing}.  We recall a useful property for readers' convenience.
\begin{equation}
 \Theta_{q_j}=p_X^*\theta+\ddc q_j\geq-R_j,
 \qquad
 R_j=\ell_j\Upsilon+\varepsilon_j\widehat\omega,
 \qquad
 0\leq\ell_j\leq c_* .
 \label{eq:middle-curvature-lower-bound}
\end{equation}
In particular, the fiber curvatures $\theta+\ddc_X q_{j,s}$ have a uniform lower bound independent of $j$ and $s$.  

Let
\[
 \mathcal R_0:=\operatorname{Ran}\left(
 \bar\partial:\operatorname{Dom}(\bar{\partial})\subset L^2_{n,0}(X,L,h_\Omega)
 \longrightarrow L^2_{n,1}(X,L,h_\Omega)\right).
\]
This is a closed subspace by classical Hodge theory on the compact manifold $X$.

\vspace{0.5cm}
The following estimate essentially follows from
\cite[Lemma~6.3]{Ber15b}.

\begin{lemma}
\label{lem:negative-exact-inverse}
Let $\chi\in C^\infty(X)$ satisfy
\begin{equation}
 \theta+\ddc\chi\geq-C_0\omega,
 \qquad \chi\leq C_1,
 \qquad \int_Xe^{-\chi}\Omega\leq C_2.
 \label{eq:negative-exact-hypotheses}
\end{equation}
If $r\in L^2(X,e^{-\chi}\Omega)$ and
\begin{equation}
 \int_Xr e^{-\chi}\Omega=0,
 \label{eq:negative-exact-mean-zero}
\end{equation}
then there is a unique $L$-valued $(n,1)$-form $\beta$ such that
\begin{equation}
 \beta\in\mathcal R_0\cap\operatorname{Dom}\bar\partial_\chi^*,
 \qquad
 \bar\partial_\chi^*\beta=-ru.
 \label{eq:negative-exact-adjoint}
\end{equation}
Moreover, we have the estimate
\begin{equation}
 \|\beta\|_\chi\leq C\|r\|_\chi.
 \label{eq:negative-exact-estimate}
\end{equation}
The constant depends only on $C_0,C_1,C_2$ and $(X,\omega)$. Let $v$ be the $L$-valued $(n-1,0)$-form such that $i\omega\wedge v=\beta$. Then
\begin{equation}
v\in \operatorname{Dom}D'_{\chi},\qquad
 D'_{\chi,X}v=ru,
 \label{eq:negative-exact-chern}
\end{equation}
and
\begin{equation}
    \|v\|_\chi\leq C\|r\|_\chi.
\end{equation}
\end{lemma}

\begin{proof}
Let
\[
 T_\chi=\bar\partial:
 L^2_{n,0}(X,L,e^{-\chi})\longrightarrow
 L^2_{n,1}(X,L,e^{-\chi}).
\]
Since $K_X\otimes L\simeq\mathcal O_X$
and $X$ is compact and connected, $\ker T_\chi$ is one-dimensional and generated by $u$.
We also note that $\operatorname{Ran}(T_{\chi})=\mathcal{R}_0$ since $\chi$ is smooth.

By \cite[Lemma~6.3]{Ber15b}, under assumptions in
\eqref{eq:negative-exact-hypotheses}, 
every
$f\in\operatorname{Ran}T_\chi$ has a unique
$S_\chi f\in\operatorname{Dom}T_\chi\cap(\ker T_\chi)^\perp$ satisfying
\begin{equation}
 T_\chi S_\chi f=f,
 \qquad
 \|S_\chi f\|_\chi\leq C\|f\|_\chi,
 \label{eq:negative-right-inverse}
\end{equation}
with $C$ depending only on $C_0,C_1,C_2$ and $(X,\omega)$.    

Put $y=-ru$.  Condition \eqref{eq:negative-exact-mean-zero} says exactly
that $y\perp\ker T_\chi$: 
\[
 (y,u)_\chi
 =-\int_X r e^{-\chi}\Omega=0.
\]
Since $\operatorname{Ran}T_\chi$ is closed, we have
\[
 \operatorname{Ran}T_\chi^*=(\ker T_\chi)^\perp,
 \qquad
 (\ker T_\chi^*)^\perp=\operatorname{Ran}T_\chi.
\]
It therefore gives a unique
$\beta\in (\ker T^*_{\chi})^{\perp} \cap \operatorname{Dom}T^*_\chi=\operatorname{Ran}T_\chi\cap\operatorname{Dom}T_\chi^*$
such that
\[
 T_\chi^*\beta=y=-ru,
 \qquad
 \|\beta\|_\chi\leq C\|r\|_\chi.
\]
Indeed, take the $L^2$-minimal solution $a=S_\chi\beta$, so that
$\beta=T_\chi a$ and $a\perp\ker T_\chi$.  Then
\eqref{eq:negative-right-inverse} and the adjoint identity give
\[
 \|\beta\|_\chi^2
 =(T_\chi a,\beta)_\chi
 =(a,T_\chi^*\beta)_\chi
 \leq C\|\beta\|_\chi\|T_\chi^*\beta\|_\chi.
\]

\vspace{0.3cm}
 Let $L_\omega=\omega\wedge$ and $\Lambda_\omega=L_\omega^*$.  We have the K\"ahler identity
\[
 \bar\partial_\chi^*
 =-i[\Lambda_\omega,D'_{\chi}].
\]
For smooth $L$-valued $(n-1,0)$-form $v$, one has $\Lambda_\omega L_\omega v=v$.  Consequently, $\bar\partial_\chi^*(i\omega\wedge v)
=-D'_{\chi}v$.
The same identity holds for $\beta\in \operatorname{Dom}\bar\partial_\chi^*$ and $v$ defined by $i\omega\wedge v=\beta$. 
Indeed, there is a sequence of smooth forms $\beta_j$ such that $\|\beta_j-\beta\|_\chi,\|\bar\partial_\chi^*\beta_j-\bar\partial_\chi^*\beta\|_\chi\rightarrow0$. Define $v_j$ by $i\omega\wedge v_j=\beta_j$.
The Lefschetz map $v\mapsto i\omega\wedge v$ is an isometry for $(n-1,0)$-forms, so we have $\|v_j-v\|_\chi\rightarrow0$ and $\|D'_\chi v_j-(-\bar\partial_\chi^*\beta)\|_\chi\rightarrow0$.
Therefore, $v\in\operatorname{Dom}D'_\chi$ and $D'_\chi v=-\bar\partial_\chi^*\beta$.
This proves
\eqref{eq:negative-exact-chern}.  Uniqueness follows from the fact that
$\operatorname{Ran}T_\chi\cap\ker T_\chi^*=\{0\}$.
\end{proof}

\vspace{0.4cm}
Recall from Proposition~\ref{prop:smooth-demailly-package} that
$M_j(s)=\int_Xe^{-q_{j,s}}\Omega$, $F_j=-\log M_j$, and
$r_j=(q_j)_\tau-(F_j)_\tau$, with
\begin{equation}
 \int_Xr_{j,s}e^{-q_{j,s}}\Omega=0,\qquad \forall s\in J_0.
 \label{eq:middle-centered-source}
\end{equation}
This means exactly $r_{j,s}u\in (\operatorname{ker}T_{q_{j,s}})^{\perp}$.
We can apply the preceding lemma and obtain the following:

\begin{proposition}
\label{prop:smooth-chern-family}
Let $S_0=J_0\times i\mathbb{R}$.
For every $j$ there is an $\operatorname{Im}\tau$-invariant $v_j\in \mathcal{C}^{\infty}(X\times S_0,p_X^*(\Lambda^{n-1,0}T^*X\otimes L))$ such that for every $s\in J_0$
\begin{equation}
 i\omega\wedge v_{j,s}\in\mathcal R_0,
 \qquad
 -\bar\partial_{q_{j,s}}^*(i\omega\wedge v_{j,s})=D'_{q_{j,s},X}v_{j,s}=r_{j,s}u.
 \label{eq:smooth-chern-equation}
\end{equation}
Writing $v_j=-\iota_{V_j}u$ defines a smooth vertical $(1,0)$ vector
field $V_j$ on $X\times S_0$.  For every
$K\Subset J_0$,
\begin{equation}
 \sup_j\int_{K\times X}|v_j|^2_{\omega,h_{\Omega}}e^{-q_j}dV_{\omega}\,ds<\infty.
 \label{eq:smooth-v-L2}
\end{equation}
\end{proposition}

\begin{proof}
Since $q_j,r_j$ are $\operatorname{Im}\tau$-invariant, we first construct $v_j$ on $X\times J_0$ and trivially extend it to $X\times S_0$.
A direct computation yields that
\begin{equation}
 \int_X|r_j|^2e^{-q_j}\Omega
 =\int_X|(q_j)_\tau|^2e^{-q_j}\Omega
  -M_j|(F_j)_\tau|^2
 \leq\int_X|(q_j)_\tau|^2e^{-q_j}\Omega.
 \label{eq:smooth-source-variance}
\end{equation}
\eqref{eq:uniform-semipositivity-loss} gives the uniform lower bound for the curvature
\[
 \theta+\ddc_Xq_{j,s}\geq-C_0\omega,\qquad \forall s\in J_0.
\]
Moreover, $q_j\leq q_1\leq C_1$ gives a common upper bound, and $\int_X e^{-q_{j,s}}\Omega\leq \int_Xe^{-Q_s}\Omega=1$.  Lemma~\ref{lem:negative-exact-inverse}, applied to each fiber $ X\times \{s\} $, gives the unique solution of
\eqref{eq:smooth-chern-equation} and a uniform constant $C=C(C_0,C_1,X,\omega)$ so that
\[
\|v_{j,s}\|_{q_{j,s}}\leq C\|r_{j,s}u\|_{q_{j,s}}.
\]
It remains to prove the smoothness of $v_j$ with respect to $s$.
If so, integrating in $s$ and using
\eqref{eq:smooth-source-variance} together with
\eqref{eq:smooth-kinetic-convergence} proves
\eqref{eq:smooth-v-L2}.

For the trivial family $p_X:X\times J_0\rightarrow J_0$, the operators $\square_{j,s}:=\bar\partial_{q_{j,s}}^*\bar\partial_X: \mathcal{C}^{\infty}(X,K_{X}\otimes L)\rightarrow\mathcal{C}^{\infty}(X,K_{X}\otimes L)$
form a smooth family of strongly elliptic operators that are formally self-adjoint with respect to the smooth metric $(p_X^*h_\Omega) e^{-q_{j}}$ on $p_X^*(K_X\otimes L)$.  
For the definition of a smooth family of strongly elliptic, formally self-adjoint operators, readers are referred to \cite[\S~7.1, Definition~7.4]{Kod86}.
Moreover,
$\ker\square_{j,s}=H^0(X,K_X\otimes L)=\mathbb C u$ for every
$s$.  Thus $\dim\ker\square_{j,s}$ is constant on $J_0$.
Let $G_{j,s}:\mathcal{C}^{\infty}(X,K_{X}\otimes L)\rightarrow \mathcal{C}^{\infty}(X,K_{X}\otimes L)$ be the Green operator of $\square_{j,s}$.
The smooth-dependence theorem \cite[\S~7.1, Theorem~7.6]{Kod86} (see also \cite[Theorem~5]{KS60}) for Green operators of a smooth family of strongly elliptic, formally self-adjoint operators says that
for a
smooth section $r_{j,s}u\in \mathcal{C}^{\infty}(X\times J_0,p_X^*(K_{X}\otimes L))$,
$g_{j,s}u:=-G_{j,s}(r_{j,s}u)$ is also a smooth section on $X\times J_0$.
Since $r_{j,s}u\perp\ker\square_{j,s}$,
$\square_{j,s}G_{j,s}(r_{j,s}u)=r_{j,s}u$ for every $s$.
This means $\bar\partial_{q_{j,s}}^* \bar\partial_X(g_{j,s}u)=-r_{j,s}u$.
Since $\bar\partial_X(g_{j,s}u)\in \operatorname{Ran}(\bar\partial_X)\cap \operatorname{Dom}\bar\partial_{q_{j,s}}^*$ and solves the equation \eqref{eq:smooth-chern-equation}, we get $i\omega\wedge v_{j,s}=\bar\partial_X(g_{j,s}u)$ for every $s$.
Consequently, $i\omega\wedge v_j$ and hence
$v_j$ are smooth on $X\times J_0$. The proof is complete.
\end{proof}

\vspace{0.4cm}
\subsection{Uniform \texorpdfstring{$L^2$}{L2} estimates}
\label{sec:pass4}

Since $i\omega\wedge v_j\in\mathcal R_0$, it is
$\bar\partial_X$-closed.  Therefore
$0=\bar\partial_X(i\omega\wedge v_j)
=i\omega\wedge\bar\partial_Xv_j$, so $\bar\partial_Xv_j$ is a
primitive $L$-valued $(n-1,1)$-form.

For any primitive $L$-valued $(n-1,1)$-form $\eta$, define its
Hodge--Riemann norm by
\begin{equation}
 |\eta|_{\mathrm{HR}}^2\Omega
 :=
 -c_n[\eta,\eta]_{h_\Omega}
 =
 |\eta|_{\omega,h_\Omega}^2\,dV_\omega.
 \label{eq:HR-norm}
\end{equation}
The last equality is the pointwise primitive Hodge--Riemann identity.
Put $\xi_j:=\frac{\partial}{\partial\tau}-V_j$, where $V_j$ is the vertical $(1,0)$-vector field defined by $-\iota_{V_j}u=v_j$.
Define real functions on the base by
\begin{align}
 \mathfrak A_j(s)
 &:=4\int_X\Theta_{q_j}
      [\xi_j,\overline{\xi_j}]e^{-q_j}\Omega,
 \label{eq:smooth-A}\\
 \mathfrak B_j(s)
 &:=4\int_X|\bar\partial_Xv_j|_{\mathrm{HR}}^2e^{-q_j}\Omega.
 \label{eq:smooth-B}
\end{align}
The factor $4$ is included because an $\mathrm{Im}\tau$-invariant scalar function satisfies
$F_{\tau\bar\tau}=F''/4$.

As in \cite[Theorem~3.1]{Ber15b}, using the curvature formula in \cite{Ber11}, we can prove the following.

\begin{lemma}
\label{lem:smooth-signed-rank-one}
For every $j$,
\begin{equation}
 M_jF_j''=\mathfrak A_j+\mathfrak B_j,
 \qquad
 \mathfrak B_j\geq0.
 \label{eq:smooth-rank-one-real}
\end{equation}
\end{lemma}

\begin{proof}
Let $p:X\times S_0\rightarrow S_0$ be the projection. 
Consider the rank-one
direct-image bundle
\[
 \mathcal E
 :=p_*\bigl(K_{X\times S_0/S_0}\otimes p_X^*L\bigr)
 \simeq S_0\times\mathbb C u .
\]
The metric \(h_\Omega e^{-q_j}\) induces a metric $\|\cdot\|_{j,\tau}$ on
\(\mathcal E\), for which, with $s=\operatorname{Re}\tau$,
$\|u\|_{j,\tau}^2:=\int_X|u|^2_{\omega,h_{\Omega}e^{-q_{j,s}}}dV_{\omega}=\int_Xc_n[u,u]_{h_\Omega}e^{-q_{j,s}}
=\int_Xe^{-q_{j,s}}\Omega=M_j(s)$.

For a smooth \(L\)-valued \((n,0)\)-form \(\eta\), write
\(\eta^{\perp_{j,s}}\) for its component orthogonal to
\(\ker\bar\partial_X=\mathbb C u\) with respect to the metric $h_{\Omega}e^{-q_{j,s}}$.  
By \eqref{eq:middle-centered-source} and \eqref{eq:smooth-chern-equation}, we have
\begin{equation}
 D'_{q_{j,s},X}v_{j,s}
 =
 r_{j,s} u
 =
 \bigl((q_{j,s})_\tau u\bigr)^{\perp_{j,s}}.
 \label{eq:smooth-berndtsson-equation}
\end{equation}


Set \(\widehat u_j:=u-d\tau\wedge v_j\).
By equation \eqref{eq:smooth-berndtsson-equation} and the primitivity of $\bar\partial_Xv_j$, \cite[Theorem~3.1]{Ber15b} (see also \cite[\S~2]{Ber11})  
gives
\begin{equation}
 \begin{split}
 \left\langle
   \Theta^{\mathcal E,j}u,u
 \right\rangle_{j,\tau}
 =
 p_*\left(
   c_n\,\Theta_{q_j}\wedge
   [\widehat u_j,\widehat u_j]_{h_\Omega}e^{-q_j}
 \right)                                                     
+
 \left(
   \int_X
   |\bar\partial_Xv_j|_{\mathrm{HR}}^2e^{-q_j}\Omega
 \right)i\,d\tau\wedge d\bar\tau \qquad \text{on }S_0.
 \end{split}
 \label{eq:smooth-berndtsson-formula}
\end{equation}


\vspace{0.4cm}
Recall that \(u\) is a holomorphic frame of \(\mathcal E\), with
$\|u\|^2_{j,\tau}=M_j(\tau)=e^{-F_j(\tau)}$.
Thus the Chern curvature is
$\Theta^{\mathcal E,j}=i\partial\bar\partial F_j$.
Therefore $\langle\Theta^{\mathcal E,j}u,u\rangle_{j,\tau}
=M_j\,i\partial\bar\partial F_j=M_j(F_j)_{\tau\bar\tau}i\,d\tau\wedge d\bar\tau$.

For the first term on the right, recall that
\(\xi_j=\partial/\partial\tau-V_j\) and
\(v_j=-\iota_{V_j}u\).  Since $\iota_{\xi_j}u=v_j$ and
\(\iota_{V_j}v_j=-\iota_{V_j}\iota_{V_j}u=0\), one has
$\iota_{\xi_j}(d\tau\wedge v_j)=v_j$, and therefore
\(\iota_{\xi_j}\widehat u_j=0\).

Choose a local \((1,0)\)-cotangent frame
\[
 d\tau,\ \zeta_j^1,\ldots,\zeta_j^n
\]
dual to a tangent frame whose first vector is \(\xi_j\), which means $d\tau(\xi_j)=1$ and
\(\zeta_j^\alpha(\xi_j)=0\).  
Since
\(\iota_{\xi_j}\widehat u_j=0\), locally
\[
 \widehat u_j
 =f_j\,\zeta_j^1\wedge\cdots\wedge\zeta_j^n\otimes e_L
\]
for some smooth function $f_j$.
Writing
\[
 \Theta_{q_j}
 =
 i\,\Theta_{0\bar0}^{\,j}\,d\tau\wedge d\bar\tau
 +\text{terms containing some }\zeta_j^\alpha
 \text{ or }\bar\zeta_j^\beta,
\]
we have
\[
 \Theta_{0\bar0}^{\,j}
 =\Theta_{q_j}[\xi_j,\overline{\xi_j}].
\]
Therefore all the remaining terms vanish after wedging with
\([\widehat u_j,\widehat u_j]_{h_\Omega}\): $\Theta_{q_j}\wedge[\widehat{u}_j,\widehat{u}_j]_{h_\Omega}=(i\Theta_{0\bar0}^{\,j}d\tau\wedge d\bar\tau)\wedge [\widehat{u}_j,\widehat{u}_j]_{h_\Omega}$.
Since $\widehat u_j=u-d\tau\wedge v_j$, 
we can infer that $(i\Theta_{0\bar0}^{\,j}d\tau\wedge d\bar\tau)\wedge [\widehat{u}_j,\widehat{u}_j]_{h_\Omega}=(i\Theta_{0\bar0}^{\,j}d\tau\wedge d\bar\tau)\wedge [u,u]_{h_{\Omega}}$.
It follows pointwise that
\[
 c_n\,\Theta_{q_j}\wedge
 [\widehat u_j,\widehat u_j]_{h_\Omega}e^{-q_j}
 =
 \Theta_{q_j}[\xi_j,\overline{\xi_j}]
 e^{-q_j}\Omega\wedge i\,d\tau\wedge d\bar\tau.
\]
Therefore,
\[
 p_*\left(
 c_n\,\Theta_{q_j}\wedge
 [\widehat u_j,\widehat u_j]_{h_\Omega}e^{-q_j}
 \right)
 =
 \left(
 \int_X\Theta_{q_j}[\xi_j,\overline{\xi_j}]
 e^{-q_j}\Omega
 \right)i\,d\tau\wedge d\bar\tau.
\]

Substituting into \eqref{eq:smooth-berndtsson-formula}, we obtain
\begin{equation}
 \begin{split}
 M_j(F_j)_{\tau\bar\tau}
 &=
 \int_X
 \Theta_{q_j}[\xi_j,\overline{\xi_j}]
 e^{-q_j}\Omega                                      \\
 &\quad+
 \int_X
 |\bar\partial_Xv_j|_{\mathrm{HR}}^2e^{-q_j}\Omega .
 \end{split}
 \label{eq:smooth-rank-one-complex}
\end{equation}

Finally, all the quantities are invariant under translation in
\(t=\operatorname{Im}\tau\), so
\((F_j)_{\tau\bar\tau}=F_j''/4\).  Multiplying
\eqref{eq:smooth-rank-one-complex} by \(4\) and using the definitions
of \(\mathfrak A_j\) and \(\mathfrak B_j\), we conclude that
$M_jF_j''=\mathfrak A_j+\mathfrak B_j$.
\(\mathfrak B_j\geq0\) is obvious.
\end{proof}

\vspace{0.5cm}
Define
$\mathfrak D_j(s):=
4\int_XR_j[\xi_j,\overline{\xi_j}]e^{-q_j}\Omega$.
Since $\Theta_{q_j}\geq-R_j$,
\begin{equation}
 (\mathfrak A_j)^-\leq\mathfrak D_j,
 \label{eq:smooth-A-negative-bound}
\end{equation}
where $(\mathfrak A_j)^-:=\max\{-\mathfrak A_j,0\}$.
On $X\times S_0$, $R_j\leq C(\ell_j+\varepsilon_j)(\omega+id\tau\wedge d\bar\tau)$ implies
\[
 R_j[\xi_j,\overline{\xi_j}]
 \leq C(\ell_j+\varepsilon_j)(1+|V_j|_\omega^2).
\]
Since $0\leq\ell_j\leq c_*$, $|V_j|_{\omega}^2\Omega=|v_j|^2_{\omega,h_{\Omega}}dV_{\omega}$,
\eqref{eq:smooth-v-L2} consequently gives the rough estimate
\begin{equation}
 \sup_j\int_K\mathfrak D_j(s)\,ds<\infty
 \qquad(K\Subset J_0).
 \label{eq:smooth-D-rough-bound}
\end{equation}
This rough control is
enough to obtain the required local $L^2$ bound.

\begin{lemma}
\label{lem:signed-curvature-bootstrap}
For every compact interval $K\Subset J_0$,
\begin{equation}
 \sup_j\int_K\left(
 \|v_j\|_{j,s}^2+\|\bar\partial_Xv_j\|_{j,s}^2
 \right)ds<\infty.
 \label{eq:smooth-graph-bound}
\end{equation}
Here $\|\cdot\|_{j,s}$ denotes the $L^2$-norm with respect to $h_\Omega e^{-q_{j,s}}$. 
\end{lemma}

\begin{proof}
Choose a compact interval $K_1$ with $K\Subset K_1\Subset J_0$.
Recall that $M_jF_j''=\mathfrak A_j+\mathfrak B_j$,
$\mathfrak B_j\geq0$, and
$(\mathfrak A_j)^-\leq\mathfrak D_j$.
For $j$ sufficiently large, Proposition~\ref{prop:smooth-demailly-package}
gives $M_j\geq1/2$ on $K_1$.  
Moreover
$M_j\leq1$ by \eqref{eq:Fj-uniform}.  On the other hand,
\eqref{eq:smooth-D-rough-bound} gives
$\sup_j\int_{K_1}\mathfrak D_j\,ds<\infty$.  Hence for $j$ large we have
\[
 (F_j'')^-
 \leq M_j^{-1}(\mathfrak A_j)^-
 \leq 2\mathfrak D_j,
\]
and therefore
\begin{equation}
 \sup_j\int_{K_1}(F_j'')^-\,ds<\infty.
 \label{eq:rough-negative-second-derivative}
\end{equation}

Choose $0\leq\eta\in C_c^\infty(K_1)$ with $\eta=1$ on $K$.  Then
\begin{align*}
 \int_K(F_j'')^+\,ds
 &\leq\int_{K_1}\eta(F_j'')^+\,ds \\
 &=\int_{K_1}F_j\eta''\,ds
   +\int_{K_1}\eta(F_j'')^-\,ds.
\end{align*}
The first term is uniformly bounded (and in fact tends to zero), because
$F_j\to0$ uniformly on $K_1$; the second is bounded by
\eqref{eq:rough-negative-second-derivative}.  Thus
\begin{equation}
 \sup_j\int_K|F_j''|\,ds<\infty.
 \label{eq:smooth-bootstrap-output}
\end{equation}

To estimate $\mathfrak B_j$, rewrite the identity as
$\mathfrak B_j=M_jF_j''-\mathfrak A_j$.  We have
\[
 \mathfrak B_j
 \leq M_j(F_j'')^++(\mathfrak A_j)^-
 \leq (F_j'')^++\mathfrak D_j
\]
for every $j$.  Integration over $K$ gives
\begin{equation}
 \sup_j\int_K\mathfrak B_j\,ds<\infty.
 \label{eq:smooth-B-rough-bound}
\end{equation}
By the definition of $\mathfrak B_j$ and \eqref{eq:HR-norm}, this is
exactly a uniform bound for
\[
 \int_{K\times X}|\bar\partial_Xv_j|_{\mathrm{HR}}^2
 e^{-q_j}\Omega\,ds.
\]
Combining it with the already established estimate $\sup_j\int_{K\times X}|v_j|^2_{\omega,h_{\Omega}}e^{-q_{j,s}}dV_{\omega}\,ds<+\infty$ in
\eqref{eq:smooth-v-L2} proves \eqref{eq:smooth-graph-bound}. 
\end{proof}
\vspace{0.4cm}
\section{\texorpdfstring{$L^2$ non-concentration near analytic sets and removal of curvature loss}{L2 non-concentration near analytic sets and removal of curvature loss}}
\label{sec:weighted-no-concentration}

\subsection{\texorpdfstring{$L^2$ non-concentration of $v_j$ near analytic sets}{L2 non-concentration of vj near analytic sets}}
This step is inspired by Berndtsson's non-concentration
argument near a klt divisor \cite[Lemma~6.5]{Ber15b}.  

\begin{lemma}
\label{lem:weighted-graph-no-concentration}
Let $K\Subset J_0$ and let $Z\subset X$ be a proper analytic subset.
Then
\begin{equation}
 \lim_{\rho\to0^+}\sup_j
 \Big(\int_K\!\int_{U_\rho(Z)}|v_j|_{\omega,h_\Omega}^2
 e^{-q_j}dV_\omega\,ds\Big)=0,
 \qquad
 U_\rho(Z):=\{x\in X:\operatorname{dist}_\omega(x,Z)<\rho\}.
 \label{eq:weighted-nc-conclusion}
\end{equation}
\end{lemma}
We remark that Lemma~\ref{lem:weighted-graph-no-concentration} actually holds for every closed subset $Z\subset X$ with zero Lebesgue measure.
\begin{proof}
The basic idea is to decompose $v_j=a_j+h_j$ locally, where $\bar\partial_X a_j=\bar\partial_X v_j$ and $h_j$ is holomorphic.
H\"ormander's $L^2$-estimates control the $L^2$-norm of $a_j$, and the sub mean-value inequality for holomorphic functions controls the $L^2$-norm of $h_j$.  

Choose a relatively compact
coordinate cover $\{U_{\alpha}\}_{\alpha=1,\ldots,N_0}$ of $X$, on which $L$ is
trivial. There exists some $r_*>0$ so that every point of $X$ is the center of a coordinate
ball $B_{2r_*}=B(z_0,2r_*)\Subset U_{\alpha}$ for some $\alpha$.  

For each chart $U_{\alpha}$ with coordinate $(z_1,\cdots,z_n)$, we fix a non-vanishing holomorphic frame $e$ of $L$ and
write
\[
 |e|_{h_\Omega}^2=e^{-\psi},
 \qquad
 v_{j,s}=\sum_{|I|=n-1}v_{j,s,I}\,dz^I\otimes e.
\]
Thus the metric
$h_\Omega e^{-q_{j,s}}$ has local weight $\psi+q_{j,s}$.  

Recall that \eqref{eq:uniform-semipositivity-loss} gives
$\theta+\ddc_Xq_{j,s}\geq-C_0\omega$, with $C_0$ independent of $j$ and
$s\in J_0$.  
There is a constant $C_1$, depending only on $C_0$ and $(X,\omega)$ such that
\[
 i\partial\bar\partial(\psi+q_{j,s})
 \geq-C_1i\partial\bar\partial|z|^2
 \quad\text{on }B_{2r}.
\]
Choose $A:=C_1r_*^2+1$ and set
\[
 \Phi_{j,s,r}
 :=\psi+q_{j,s}+A r^{-2}|z-z_0|^2.
\]
Then for $r\leq r_*$,
\begin{equation}
 i\partial\bar\partial\Phi_{j,s,r}
 \geq r^{-2}i\partial\bar\partial|z|^2
 \quad\text{on }B_{2r}.
 \label{eq:nc-strict-local-weight}
\end{equation}

Put $f_{j,s}:=\bar\partial_Xv_{j,s}$.  In the above local chart,
the coefficient forms $f_{j,s,I}:=\bar\partial v_{j,s,I}$ are
$\bar\partial$-closed $(0,1)$-forms.  H\"ormander's estimate \cite{Hor65} (see also \cite{Hor90}), with the strictly plurisubharmonic weight
$\Phi_{j,s,r}$, gives
solutions $a_{j,s,I}$ satisfying
$\bar\partial a_{j,s,I}=f_{j,s,I}$ and
\[
 \int_{B_{2r}}|a_{j,s,I}|^2e^{-\Phi_{j,s,r}}d\lambda
 \leq r^2
 \int_{B_{2r}}|f_{j,s,I}|^2e^{-\Phi_{j,s,r}}d\lambda.
\]
Note that $0\leq Ar^{-2}|z-z_0|^2\leq4A$ on $B_{2r}$, 
Define $a_{j,s}:=\sum_Ia_{j,s,I}dz^I\otimes e$ on $B_{2r}$; it satisfies
$\bar\partial_X a_{j,s}=\bar\partial_Xv_{j,s}$ and
\begin{equation}
 \int_{B_{2r}}|a_{j,s}|_{\omega,h_\Omega}^2
 e^{-q_{j,s}}dV_\omega
 \leq C_{\mathrm H}r^2
 \int_{B_{2r}}|\bar\partial_Xv_{j,s}|_{\omega,h_\Omega}^2
 e^{-q_{j,s}}dV_\omega.
 \label{eq:nc-local-hormander}
\end{equation}
The constant $C_{\mathrm H}$ depends only on $C_0,A$ and
$(X,\omega,h_\Omega)$.

\vspace{0.4cm}
Set $h_{j,s}:=v_{j,s}-a_{j,s}$.  It is holomorphic on $B_{2r}$.  Since
$q_j\leq q_1$ and $q_1$ is smooth on a neighborhood of $J_0\times X$,
there is a constant $C_{J_0}$ such that $q_{j}\leq C_{J_0}$ for all $j$.  Hence we have a simple observation:
\[
 \int_{B_{2r}}|\gamma|^2\,dV_\omega
 \leq e^{C_{J_0}}\int_{B_{2r}}|\gamma|^2e^{-q_{j,s}}\,dV_\omega.
\]
The mean-value inequality for every holomorphic
coefficient of $h_{j,s}$, together with
\eqref{eq:nc-local-hormander} yields
\begin{equation}\label{eq:nc-holomorphic-sup}
\begin{split}
\sup_{B_r}|h_{j,s}|_{\omega,h_\Omega}^2
 &\leq \frac{C}{r^{2n}}\int_{B_{2r}}|h_{j,s}|_{\omega,h_\Omega}^2dV_{\omega}\\
 &\leq \frac{C}{r^{2n}}\int_{B_{2r}}2\bigl(|v_{j,s}|_{\omega,h_\Omega}^2+|a_{j,s}|_{\omega,h_\Omega}^2\bigl)dV_{\omega}\\
 &\leq \frac{C'}{r^{2n}}\int_{B_{2r}}\bigl(|v_{j,s}|_{\omega,h_\Omega}^2+|\bar\partial_X v_{j,s}|_{\omega,h_\Omega}^2\bigl)e^{-q_{j,s}}dV_{\omega}\\
 &= \frac{C'}{r^{2n}} E_{j,s}(B_{2r}).
\end{split}
\end{equation}
where
\begin{equation}
 E_{j,s}(B):=\int_B
 \left(|v_{j,s}|_{\omega,h_\Omega}^2
       +|\bar\partial_Xv_{j,s}|_{\omega,h_\Omega}^2\right)
 e^{-q_{j,s}}dV_\omega.
\end{equation}
Therefore, for every measurable $U\subset B_r$, using
$|v_{j,s}|_{\omega,h_\Omega}^2
 \leq2|a_{j,s}|_{\omega,h_\Omega}^2
     +2|h_{j,s}|_{\omega,h_\Omega}^2$
and
$q_j\geq Q$, we obtain
\begin{align}
 \int_U|v_{j,s}|_{\omega,h_\Omega}^2
 e^{-q_{j,s}}dV_\omega
 \leq{}&Cr^2\int_{B_{2r}}
 |\bar\partial_Xv_{j,s}|_{\omega,h_\Omega}^2
 e^{-q_{j,s}}dV_\omega \notag\\
 &+\frac{C}{r^{2n}}E_{j,s}(B_{2r})\int_Ue^{-Q_s}\Omega.
 \label{eq:nc-local-final}
\end{align}

\vspace{0.5cm}
For every $r\leq r_*$, there are finitely many balls $\{B_{\beta,r}\}$ of radius $r$ covering
$Z$, their doubled balls $B_{\beta,2r}$ have overlap bounded by a
constant $N_X$ depending only on $X$. 
Their union contains $U_{\rho(r)}(Z)$ for some $\rho(r)>0$. 
For $\rho\leq \rho(r)$, apply \eqref{eq:nc-local-final} to
$U_{\rho}\cap B_{\beta,r}$ and sum over $\alpha$:
\begin{align*}
 &\int_{U_{\rho}}|v_{j,s}|_{\omega,h_\Omega}^2e^{-q_{j,s}}dV_{\omega}\\
 &\leq \sum_{\beta}  \int_{U_{\rho}\cap B_{\beta,r}}|v_{j,s}|_{\omega,h_\Omega}^2e^{-q_{j,s}}dV_{\omega}\\
 &\leq\sum_\beta Cr^2\int_{B_{\beta,2r}}
 |\bar\partial_Xv_{j,s}|_{\omega,h_\Omega}^2
 e^{-q_{j,s}}dV_\omega
 +\sum_\beta\frac{C}{r^{2n}}E_{j,s}(B_{\beta,2r})\int_{U_{\rho}\cap B_{\beta,r}}e^{-Q_s}\Omega\\
 &\leq N_XCr^2E_{j,s}(X)
 +N_X\frac{C}{r^{2n}}E_{j,s}(X)
 \int_{U_{\rho}}e^{-Q_s}\Omega.
\end{align*}
Lemma~\ref{lem:signed-curvature-bootstrap} implies $C_K:=\sup_j\int_KE_{j,s}(X)ds<+\infty$.
Integrating in $s$ therefore gives
\begin{equation}
 \sup_j\int_K\!\int_{U_{\rho}}|v_j|_{\omega,h_\Omega}^2
 e^{-q_j}dV_\omega\,ds
 \leq C'C_Kr^2+\frac{C'C_K}{r^{2n}}\sup_{s\in K}\int_{U_{\rho}}e^{-Q_s}\Omega.
 \label{eq:nc-global-estimate}
\end{equation}
Here the second term is controlled in the following way:
\[
 \int_K E_{j,s}(X)\left(\int_Ue^{-Q_s}\Omega\right)ds
 \leq
 \left(\sup_{s\in K}\int_Ue^{-Q_s}\Omega\right)
 \int_KE_{j,s}(X)\,ds.
\]

By Lemma~\ref{lem:kinetic}, the map $s\mapsto e^{-Q_s}$ is continuous
from the compact interval $K$ into $L^1(X,\Omega)$.  
Since $Z$ has zero $\Omega$-measure and the closedness of $Z$ implies that $1_{U_\rho}$ decreases pointwise to $1_Z$, it follows that
\begin{equation}
 \sup_{s\in K}\int_{U_\rho}e^{-Q_s}\Omega\longrightarrow0
 \quad\text{as }\rho\to0^+.
 \label{eq:nc-uniform-integrability}
\end{equation}
For a fixed $r$, let $\rho\to0^+$; by \eqref{eq:nc-uniform-integrability},
\[
 \limsup_{\rho\to0^+}\sup_j
 \int_K\!\int_{U_\rho(Z)}|v_j|_{\omega,h_\Omega}^2
 e^{-q_j}dV_\omega\,ds
 \leq C_Kr^2.
\]
Now let $r\to0^+$; the proof is complete.  
\end{proof}
\vspace{0.4cm}

\subsection{Removal of the curvature loss and vanishing of the equality defects}
\label{sec:lelong-loss}

The preceding no-concentration result removes
the loss of positivity of regularization sequence.

\begin{proposition}
\label{prop:lelong-loss-vanishing}
For every $K\Subset J_0$,
\begin{equation}
 \int_{K\times X}(\ell_j+\varepsilon_j)
 \bigl(1+|V_j|_\omega^2\bigr)e^{-q_j}\Omega\,ds
 \longrightarrow0.
 \label{eq:lelong-loss-vanishing}
\end{equation}
Consequently,
\begin{equation}
 \int_K\mathfrak D_j(s)ds=4\int_{K\times X}R_j[\xi_j,\overline{\xi_j}]e^{-q_j}\Omega\,ds\longrightarrow0.
 \label{eq:D-vanishing}
\end{equation}
\end{proposition}

\begin{proof}
The identity
\eqref{eq:contraction-identity} and
Lemma~\ref{lem:weighted-graph-no-concentration} give
\begin{equation}
 \lim_{\rho\downarrow0}\sup_j
 \int_{K\times U_\rho(Z)}|V_j|_\omega^2e^{-q_j}\Omega\,ds=0
 \label{eq:V-no-concentration}
\end{equation}
for every proper analytic subset $Z\subset X$.

Fix $a>0$ and use the vertical analytic set $Z_{a/2}$ from
Lemma~\ref{lem:vertical-lelong-strata}.  If $U$ is any neighborhood of
$Z_{a/2}$, then \eqref{eq:uniform-loss-away-from-stratum} gives, for all
large $j$,
\[
 \ell_j<a
 \quad\text{on }(X\setminus U)\times K.
\]
Splitting the integral into $U$ and $X\setminus U$ yields
\begin{align*}
 \int_{K\times X}\ell_j|V_j|_\omega^2e^{-q_j}\Omega\,ds
 \leq{}&a\int_{K\times X}|V_j|_\omega^2e^{-q_j}\Omega\,ds\\
 &+\|\ell_1\|_{L^\infty}
   \int_{K\times U}|V_j|_\omega^2e^{-q_j}\Omega\,ds.
\end{align*}
The first integral on the right is uniformly bounded by
\eqref{eq:smooth-v-L2}.  Take the upper limit as $j\to\infty$, then shrink
$U$ to $Z_{a/2}$ and use \eqref{eq:V-no-concentration}; this gives
\[
 \limsup_{j\to\infty}
 \int_{K\times X}\ell_j|V_j|_\omega^2e^{-q_j}\Omega\,ds
 \leq Ca.
\]
Since $a>0$ is arbitrary, we obtain
\begin{equation}
 \lim_{j\rightarrow\infty}\int_{K\times X}\ell_j|V_j|_\omega^2e^{-q_j}\Omega\,ds=0.
 \label{eq:ell-V-vanishing}
\end{equation}
The scalar part is simpler.
The same decomposition gives
\begin{align*}
 \int_{K\times X}\ell_je^{-q_j}\Omega\,ds
 \leq{}&a\int_{K\times X}e^{-q_j}\Omega\,ds\\
 &+\|\ell_1\|_{L^\infty}
   \int_K\int_Ue^{-Q_s}\Omega\,ds,
\end{align*}
because $e^{-q_j}\leq e^{-Q}$.  Uniform integrability
\eqref{eq:nc-uniform-integrability}, followed by $a\downarrow0$, gives
\begin{equation}
 \int_{K\times X}\ell_je^{-q_j}\Omega\,ds\longrightarrow0.
 \label{eq:ell-scalar-vanishing}
\end{equation}
Finally,
\[
 \varepsilon_j\int_{K\times X}
 \bigl(1+|V_j|_\omega^2\bigr)e^{-q_j}\Omega\,ds\longrightarrow0
\]
is clear thanks to \eqref{eq:smooth-v-L2} and \eqref{eq:density-L1}.  This proves
\eqref{eq:lelong-loss-vanishing}; the definition of $\mathfrak D_j$ immediately
gives \eqref{eq:D-vanishing}.
\end{proof}

\vspace{0.3cm}
Now we can prove the following proposition.
\begin{proposition}
\label{prop:smooth-defect-vanishing}
For every $K\Subset J_0$,
\begin{equation}
 \int_K(\mathfrak A_j)^+\,ds\longrightarrow0,
 \qquad
 \int_K\mathfrak B_j\,ds\longrightarrow0,
 \qquad
 \int_K|F_j''|\,ds\longrightarrow0.
 \label{eq:smooth-defects-vanish}
\end{equation}
In particular,
\begin{equation}
 \int_{K\times X}|\bar\partial_Xv_j|^2e^{-q_j}\Omega\,ds
 \longrightarrow0.
 \label{eq:dbar-v-defect-vanish}
\end{equation}
Moreover, with $S_j=\Theta_{q_j}+R_j\geq0$ as in
\eqref{eq:shifted-smooth-curvature},
\begin{equation}
 \int_K\int_XS_j[\xi_j,\overline{\xi_j}]
 e^{-q_j}\Omega\,ds\longrightarrow0.
 \label{eq:shifted-curvature-defect}
\end{equation}
\end{proposition}

\begin{proof}
Fix \(K\Subset J_0\), and choose a compact interval
\(K_1\Subset J_0\) such that $K\Subset \operatorname{int}K_1$.

By \eqref{eq:smooth-A-negative-bound} and
\eqref{eq:D-vanishing}, applied on \(K_1\), we have
\begin{equation}
 \int_{K_1}(\mathfrak A_j)^-\,ds
 \leq
 \int_{K_1}\mathfrak D_j\,ds
 \longrightarrow 0.
 \label{eq:A-negative-vanish}
\end{equation}

We next show that
\begin{equation}
 M_jF_j''\longrightarrow 0
 \qquad\text{in }
 \mathcal D'\bigl(\operatorname{int}K_1\bigr).
 \label{eq:weighted-second-derivative-distributional}
\end{equation}
First, the locally uniform convergence
\(F_j\to0\) from \eqref{eq:Fj-uniform} directly implies
\[
 F_j''\longrightarrow0
 \qquad\text{in }
 \mathcal D'\bigl(\operatorname{int}K_1\bigr).
\]

In the proof of Lemma~\ref{lem:signed-curvature-bootstrap}, 
we get \eqref{eq:smooth-bootstrap-output}:
\begin{equation}
 \sup_j\int_{K_1}|F_j''|\,ds<\infty.
 \label{eq:local-F-second-total-variation}
\end{equation}
Since \(M_j=e^{-F_j}\), the locally uniform convergence \(F_j\to0\)
also gives \(M_j\to1\) uniformly on \(K_1\).  Therefore, for every
\(\varphi\in C_c^\infty(\operatorname{int}K_1)\),
\[
 \begin{split}
 \left|
 \left\langle (M_j-1)F_j'',\varphi\right\rangle
 \right|
 &\leq
 \|\varphi\|_{L^\infty(K_1)}
 \|M_j-1\|_{L^\infty(K_1)}
 \int_{K_1}|F_j''|\,ds\\
 &\longrightarrow0
 \end{split}
\]
by \eqref{eq:local-F-second-total-variation}.  Combining this with
\(F_j''\to0\) in distributions proves
\eqref{eq:weighted-second-derivative-distributional}.

We now use the identity
\eqref{eq:smooth-rank-one-real}.  Writing
\(\mathfrak A_j=(\mathfrak A_j)^+-(\mathfrak A_j)^-\), it becomes
\begin{equation}
 (\mathfrak A_j)^++\mathfrak B_j
 =
 M_jF_j''+(\mathfrak A_j)^-.
 \label{eq:positive-defect-identity}
\end{equation}
Choose
\(\eta\in C_c^\infty(\operatorname{int}K_1)\) such that
\(0\leq\eta\leq1\) and \(\eta\equiv1\) on \(K\).  Since both
\((\mathfrak A_j)^+\) and \(\mathfrak B_j\) are nonnegative,
\eqref{eq:positive-defect-identity} gives
\[
 \begin{split}
 0
 &\leq
 \int_K\bigl((\mathfrak A_j)^++\mathfrak B_j\bigr)\,ds\\
 &\leq
 \int_{K_1}
 \eta\bigl((\mathfrak A_j)^++\mathfrak B_j\bigr)\,ds\\
 &=
 \left\langle M_jF_j'',\eta\right\rangle
 +
 \int_{K_1}\eta(\mathfrak A_j)^-\,ds.
 \end{split}
\]
The first term tends to \(0\) by
\eqref{eq:weighted-second-derivative-distributional}, and the second
one tends to \(0\) by \eqref{eq:A-negative-vanish}.  Consequently,
\[
 \int_K(\mathfrak A_j)^+\,ds\longrightarrow0,
 \qquad
 \int_K\mathfrak B_j\,ds\longrightarrow0.
\]

\vspace{0.4cm}
It remains to prove the \(L^1\)-vanishing of \(F_j''\).  Since
\(M_j\to1\) uniformly on \(K_1\), for all sufficiently large \(j\) one
has \(M_j\geq\frac12\) there.  Thus, using
\(M_jF_j''=\mathfrak A_j+\mathfrak B_j\),
\[
 \begin{split}
 \int_K|F_j''|\,ds
 &\leq
 2\int_K|\mathfrak A_j+\mathfrak B_j|\,ds\\
 &\leq
 2\int_K
 \bigl(
   (\mathfrak A_j)^+
   +(\mathfrak A_j)^-
   +\mathfrak B_j
 \bigr)\,ds
 \longrightarrow0.
 \end{split}
\]
This proves \eqref{eq:smooth-defects-vanish}.

\vspace{0.3cm}
By the definition of \(\mathfrak B_j\), the convergence
\(\int_K\mathfrak B_j\,ds\to0\) is precisely the asserted \eqref{eq:dbar-v-defect-vanish}.

Finally, recall that
\(S_j=\Theta_{q_j}+R_j\geq0\).  From the definitions of
\(\mathfrak A_j\) and \(\mathfrak D_j\),
\[
 \begin{split}
 0
 &\leq
 4\int_{K\times X}
 S_j[\xi_j,\overline{\xi_j}]
 e^{-q_j}\Omega\,ds\\
 &=
 \int_K\bigl(\mathfrak A_j+\mathfrak D_j\bigr)\,ds\\
 &\leq
 \int_K(\mathfrak A_j)^+\,ds
 +
 \int_K\mathfrak D_j\,ds
 \longrightarrow0,
 \end{split}
\]
where the last limit follows from the first part of the proof and
\eqref{eq:D-vanishing}.  
The proof is complete.
\end{proof}

\vspace{0.4cm}

\section{Passage to the limit on \texorpdfstring{$X$}{X}}
\label{sec:smooth-limit}

\begin{proposition}
\label{prop:smooth-fixed-exact-limit}
After passing to a subsequence, there exists a vertical \(L\)-valued
\((n-1,0)\)-form
\[
 v\in
 L^2_{\mathrm{loc}}
 \bigl(X\times J_0,
       p_X^*(\Lambda^{n-1,0}T_X^*\otimes L),h_{\Omega},\omega\bigr)
\]
such that, for every \(K\Subset J_0\),
\begin{align}
 v_j&\rightharpoonup v
 &&\text{weakly in }L^2(X\times K),
 \label{eq:vj-unweighted-weak}\\
 v_je^{-q_j/2}&\rightharpoonup ve^{-Q/2}
 &&\text{weakly in }L^2(X\times K),
 \label{eq:vj-half-density-weak}\\
 \bar\partial_Xv&=0
 &&\text{in distributions on }X\times K,
 \label{eq:limit-fiber-holomorphic}\\
 D'_{X}(ve^{-Q})&=Q_\tau u e^{-Q}
 &&\text{in distributions on }X\times K.
 \label{eq:limit-conservative-equation}
\end{align}
Moreover, for almost every $s$, the slice $v_s$ is holomorphic and
\begin{equation}
 i\omega\wedge v_s\in\mathcal R_0.
 \label{eq:limit-fixed-exact-gauge}
\end{equation}
For such \(s\), the identity
\begin{equation}
 v_s=-\iota_{V_s}u
 \label{eq:limit-vector-field}
\end{equation}
defines a holomorphic vector field \(V_s\) on \(X\).
\end{proposition}
Remark. 
For brevity, we write $L^2(X \times K)$ for $L^2(X\times K,p^*_X(\Lambda^{n-1,0}T^*X\otimes L),h_{\Omega},\omega)$.
Equations \eqref{eq:limit-fiber-holomorphic} and \eqref{eq:limit-conservative-equation} mean that for every $\Psi\in C^{\infty}_c(X\times K, p_X^*(\Lambda^{n-1,1}T_X^*\otimes L))$ and $\Phi\in C^{\infty}_c(X\times K, p_X^*(\Lambda^{n,0}T_X^*\otimes L))$, we have 
\begin{equation}
\begin{split}
     &\int_K\int_X\langle v,\bar\partial^*_X\Psi\rangle_{\omega,h_{\Omega}} dV_{\omega}ds=0,\\
    &\int_K\int_X\langle ve^{-Q}, (D_{X}')^*\Phi\rangle_{\omega,h_{\Omega}}dV_{\omega}ds
    =\int_K\int_X\langle Q_{\tau}ue^{-Q},\Phi \rangle_{\omega,h_{\Omega}} dV_{\omega}ds
\end{split}
   \label{eq:weak covariant derivative}
\end{equation}
where $\bar\partial_X^{*}$ and $(D'_X)^*$ are taken with respect to $(\omega,h_{\Omega})$.
\vspace{0.4cm}
\begin{proof}
Fix \(K\Subset J_0\).
There is some $C_K>0$ so that $q_j\leq C_K$ on \(X\times K\).  Hence
\(e^{-q_j}\geq e^{-C_K}\), and \eqref{eq:smooth-v-L2} implies
\[
 \sup_j\int_{K\times X}|v_j|^2\,dV_{\omega}\,ds<\infty.
\]
Thus, after passing to a subsequence,
\(v_j\rightharpoonup v\) weakly in \(L^2(K\times X)\).

The same estimate \eqref{eq:smooth-v-L2} shows that
\(v_je^{-q_j/2}\) is bounded in \(L^2(K\times X)\).  Passing to a
further subsequence, write
\[
 v_je^{-q_j/2}\rightharpoonup W
 \qquad\text{weakly in }L^2(K\times X).
\]
We prove that \(W=ve^{-Q/2}\). 
First, from $v,e^{-Q/2}\in L^2(K\times X) $ we know $ve^{-Q/2}\in L^1(K\times X) $. 
Take \(\Psi\in \mathcal{C}^{\infty}_c(K\times X,p_X^*(\Lambda^{n-1,0}T^*X\otimes L))\), then
\[
 \int_{K}\int_{X}
 \langle
   v_j,
   \bigl(e^{-q_j/2}-e^{-Q/2}\bigr)\Psi
 \rangle_{\omega,h_{\Omega}}
 dV_\omega\,ds
 \longrightarrow0,
\]
because \(v_j\) is uniformly bounded in the unweighted \(L^2\)-space
and, by \eqref{eq:half-density-L2},
\[
 e^{-q_j/2}-e^{-Q/2}\longrightarrow0
 \qquad\text{strongly in }L^2(K\times X).
\]
On the other hand, \(e^{-Q/2}\Psi\in L^2(K\times X)\), so the weak
convergence of \(v_j\) gives
\[
 \int_{K\times X}
 \langle v_j,e^{-Q/2}\Psi\rangle dV_\omega\,ds
 \longrightarrow
 \int_{K\times X}
 \langle v,e^{-Q/2}\Psi\rangle dV_\omega\,ds.
\]
It follows that
\begin{align*}
    \int_K\int_X\langle v_je^{-q_j/2},\Psi \rangle dV_{\omega}ds \longrightarrow \int_K\int_X\langle ve^{-Q/2} ,\Psi \rangle dV_{\omega}ds,
\end{align*}
In particular, we obtain that $ve^{-Q/2}=W$ in the sense of distributions. Since $ve^{-Q/2}\in L^1$ and $W\in L^2$, we therefore infer that $ve^{-Q/2}\in L^2(K\times X)$ and $v_je^{-q_j/2}\rightharpoonup ve^{-Q/2}$ weakly, proving
\eqref{eq:vj-half-density-weak}.

\vspace{0.3cm}
We next prove \eqref{eq:limit-fiber-holomorphic}.  Since
\(q_j\leq C_K\), \eqref{eq:dbar-v-defect-vanish} gives
\[
 \begin{split}
 \int_{K\times X}|\bar\partial_Xv_j|^2dV_{\omega}\,ds
 &\leq
 e^{C_K}
 \int_{K\times X}
 |\bar\partial_Xv_j|^2e^{-q_j}dV_{\omega}\,ds\\
 &\longrightarrow0.
 \end{split}
\]
Combining this strong convergence with
\(v_j\rightharpoonup v\) proves
\(\bar\partial_Xv=0\) in distributions.

\vspace{0.3cm}
We now verify that for almost every $s$, 
$i\omega\wedge v_s\in \mathcal{R}_0$.  
Put
\[\mathcal H:=L^2(X,\Lambda^{n,1}T^*X\otimes L,h_\Omega,\omega), \qquad \mathcal{R}_0:=\operatorname{Ran}\bigl(\bar\partial:\operatorname{Dom}\bar\partial\subset L^2(X,\Lambda^{n,0}T^*X\otimes L)\rightarrow \mathcal H\bigr).\]
The subspace
\(\mathcal R_0\subset\mathcal H\) is closed; let
\(P_0^\perp:\mathcal H\to\mathcal R_0^\perp\) be the orthogonal
projection. 

We say that a map $W:K\rightarrow \mathcal H$ is strongly measurable if it is the a.e. limit of a sequence of simple maps; that is, if there is a sequence $W_j:K\rightarrow \mathcal{H}$ such that $W_j=1_{E_{j,1}}h_{j,1}+\cdots +1_{E_{j,N_j}}h_{j,N_j}$ for measurable sets $E_{j,i}\subset K$ and $h_{j,i}\in \mathcal H$, with $\lim_{j\rightarrow\infty}W_j(s)=W(s)$ for a.e. $s\in K$.
Then 
\[
L^2(K,\mathcal{H}):=\Big\{W \ \text{strongly measurable}:\|W\|_{L^2}:=\Big(\int_K|W(s)|_{\mathcal{H}}^2ds\Big)^{\frac{1}{2}}<\infty\Big\}.
\]
This is a Hilbert space with inner product $(W_1,W_2)=\int_K\langle W_1(s),W_2(s)\rangle_{\mathcal H}ds$, $W_1,W_2\in L^2(K,\mathcal H)$ (see, e.g., \cite[Chapter~1]{HNVW16}).
We need the following fact:
\begin{align*}
L^2(K\times X,p^*_X(\Lambda^{n,1}T^*X\otimes L))\rightarrow L^2(K,\mathcal H)\\
w\mapsto W:K\rightarrow \mathcal H,s\mapsto w_s
\end{align*}
is an isometry between Hilbert spaces.
Indeed, simple sections (respectively, maps) are dense in both spaces. Moreover, for a simple section $w=1_{E_1}(s)s_1(x)+\cdots+1_{E_N}(s)s_N(x)$, where $E_i\subset K$ is measurable and $s_i\in L^2(X,\Lambda^{n,1}T^*X\otimes L)$, it is easy to see that
\[
\|w\|^2=\int_K \sum_i 1_{E_i}\int_X|s_i|^2_{\omega,h_{\Omega}}dV_{\omega}ds
=\sum_i\Big(\int_{E_i}ds\Big)|s_i|^2_{\mathcal{H}}
=\|W\|^2.
\]
Now, for each $W\in L^2(K,\mathcal H)$, $W$ is an a.e. limit of simple maps $W_j=1_{E_{j,1}}(s)s_{j,1}(x)+\cdots+1_{E_{j,N_j}}(s)s_{j,N_j}(x)$, and it is clear that $P_0^{\perp}W_j=1_{E_{j,1}}(s)P_0^{\perp}s_{j,1}(x)+\cdots+1_{E_{j,N_j}}(s)P_0^{\perp}s_{j,N_j}(x)$ are still simple. Therefore $P_0^{\perp}W:K\to\mathcal{H}$, $s\mapsto P_0^{\perp}W_s$, is strongly measurable and $\|P_0^{\perp}W\|_{L^2}\leq \|W\|_{L^2}$. 
In this way, the projection $P_0^\perp$ induces a bounded linear operator on $L^2(K,\mathcal H)$, acting on each $W$ fiberwise. By the isometry above, $P_0^\perp$ can also be viewed as a bounded linear operator on $L^2(K\times X)$.

Since the Lefschetz map $a\mapsto i\omega\wedge a$ from $L^2(\Lambda^{n-1,0}T^*X\otimes L)$ to $L^2(\Lambda^{n,1}T^*X\otimes L)$ is an isometry, we know
\(i\omega\wedge v_j\rightharpoonup i\omega\wedge v\) weakly in
\(L^2(K\times X)\).  By \eqref{eq:smooth-chern-equation}, $i\omega\wedge v_{j,s}\in \mathcal{R}_0$ for every $j,s$, hence
\(P_0^\perp(i\omega\wedge v_j)=0\). By the weak continuity of $P_0^{\perp}$ on $L^2(K\times X)$, we infer that 
\(P_0^\perp(i\omega\wedge v)=0\). Since $\mathcal R_0$ is closed and $\mathcal R_0=\ker P_0^{\perp}$,
this means exactly
$i\omega\wedge v_s\in \mathcal{R}_0$ for almost every \(s\in K\) and  \eqref{eq:limit-fixed-exact-gauge} is proved.

\vspace{0.3cm}
It remains to prove \eqref{eq:limit-conservative-equation}.  For each
smooth \(q_j\), the equation \(D'_{q_j,X}v_j=r_ju\) is
equivalent to
\begin{equation}
 D'_{X}(v_je^{-q_j})
 =r_ju e^{-q_j}.
 \label{eq:smooth-conservative-equation}
\end{equation}
Indeed,
\[
 D'_{X}(v_je^{-q_j})
 =
 e^{-q_j}
 \bigl(
   D'_{X}v_j-\partial_Xq_j\wedge v_j
 \bigr)
 =
 e^{-q_j}D'_{q_j,X}v_j.
\]

Both sides of \eqref{eq:smooth-conservative-equation} can be factored
into half-densities:
\[
 v_je^{-q_j}
 =
 \bigl(v_je^{-q_j/2}\bigr)e^{-q_j/2},
 \qquad
 r_je^{-q_j}
 =
 \bigl(r_je^{-q_j/2}\bigr)e^{-q_j/2}.
\]
For every smooth compactly supported test form \(\Phi\), the weak convergence
\eqref{eq:vj-half-density-weak} and the strong convergence
\eqref{eq:half-density-L2} give
\[
 \int_{K}\int_{X}
 \langle v_je^{-q_j},\Phi\rangle dV_{\omega}\,ds
 \longrightarrow
 \int_{K}\int_{X}
 \langle ve^{-Q},\Phi\rangle dV_{\omega}\,ds.
\]
Thus \(v_je^{-q_j}\to ve^{-Q}\) in distributions.  Likewise,
\eqref{eq:centered-kinetic-convergence} and
\eqref{eq:half-density-L2} imply
\(r_je^{-q_j}\to Q_\tau e^{-Q}\) strongly in \(L^1(K\times X)\).
Testing \eqref{eq:smooth-conservative-equation} against a smooth
compactly supported form $\Psi$ gives
\[
\int_K\int_X\langle D'_X(v_je^{-q_j}),\Psi\rangle dV_{\omega}ds
=
\int_K\int_X \langle v_je^{-q_j},(D^{\prime}_X)^*\Psi\rangle dV_{\omega}ds
\]
and passing to the limit proves
\eqref{eq:limit-conservative-equation}.

\vspace{0.3cm}
Finally, Fubini's theorem applied to
\eqref{eq:limit-fiber-holomorphic} shows that
\(\bar\partial_Xv_s=0\) on \(X\) for almost every
\(s\).  
Indeed, take test form $\Psi=\eta(s)w(x)$ where $\eta\in \mathcal{C}_c^{\infty}(K)$ and $w\in \mathcal{C}^{\infty}(X,\Lambda^{n-1,1}T^*X\otimes L)$. $\bar\partial_Xv=0$ reads as
\[
\int_K\eta(s)\int_X\langle v_s,\bar\partial^*_Xw\rangle dV_{\omega}ds=0.
\]
Since $\eta$ is arbitrary, for almost every $s$, $\int_X\langle v_s,\bar\partial^*_Xw\rangle dV_{\omega}=0$.
Taking a countable dense subset $\{w_j\}_{j\geq1}\subset C^{\infty}(X,\Lambda^{n-1,1}T^*X\otimes L)$ yields that, for almost every $s$, $\int_X\langle v_s,\bar\partial_X^*w\rangle dV_{\omega}=0$ for every $w\in C^{\infty}(X,\Lambda^{n-1,1}T^*X\otimes L)$.
This implies $\bar\partial_X v_s=0$, as desired.
Since contraction with the nowhere-vanishing holomorphic identity
section \(u\) gives a fiberwise linear isomorphism
\[
 T_X^{1,0}
 \longrightarrow
 \Lambda^{n-1,0}T_X^*\otimes L,
 \qquad
 V\longmapsto-\iota_Vu,
\]
equation \eqref{eq:limit-vector-field} defines a holomorphic vector
field \(V_s\).

Applying the preceding argument to a compact exhaustion of \(J_0\) and
using a diagonal subsequence gives a single \(v\) with all the asserted properties.
\end{proof}

We next prove that the limit $v$ is independent of both the choice of the subsequence $v_j$ and, more importantly, the regularization sequence $q_j$. 
The reason is that for an exhaustion $J_0\Subset J_1\Subset\cdots \Subset I$, Demailly's regularization theorem may produce different $q_j$ on each $J_k$. 
Thus, $v$ obtained on $J_0\times X$ in Proposition~\ref{prop:smooth-fixed-exact-limit} may not be patched with the one obtained on $J_1\times X$.
To exclude this concern, we give a characterization of $v$, which only involves $Q$.

\begin{lemma}
\label{lem:smooth-limit-uniqueness}
There exists a unique
\[
 v\in L^2_{\mathrm{loc}}(X\times J_0),
 \qquad
 ve^{-Q/2}\in L^2_{\mathrm{loc}}(X\times J_0),
\]
satisfying the following three conditions:
\begin{align}
 i\omega\wedge v_s&\in\mathcal R_0
 &&\text{for almost every }s,                                   
 \label{eq:limit-uniqueness-gauge}\\
 \bar\partial_Xv&=0
 &&\text{in distributions},                                     
 \label{eq:limit-uniqueness-holomorphic}\\
 D'_{X}(ve^{-Q})&=Q_\tau u e^{-Q}
 &&\text{in distributions}.                                     
 \label{eq:limit-uniqueness-conservative}
\end{align}
\end{lemma}
\begin{proof}
Let \(v\) and \(\widetilde v\) satisfy the three conditions in the
statement, and put \(p:=v-\widetilde v\).  Then
\begin{equation}
i\omega\wedge p_s\in\mathcal R_0
 \ \text{for almost every }s,\qquad
 \bar\partial_Xp=0,\qquad
 D'_{X}(pe^{-Q})=0.
 \label{eq:homogeneous-limit-data}
\end{equation}
The same arguments as in Proposition~\ref{prop:smooth-fixed-exact-limit} yield that for almost every $s$, $p_s$ is holomorphic and $D^{\prime}(p_se^{-Q_s})=0$ in the sense that $\int_X\langle p_se^{-Q_s},(D^{\prime})^*\Psi\rangle_{\omega,h_{\Omega}} dV_{\omega}=0$ holds for every test form $\Psi$ on $X$.

Fix such a \(s\) and set
\(\beta_s:=i\omega\wedge p_s\).
Since \(\beta_s\in\mathcal R_0=\operatorname{Ran}\bar\partial\), there is a smooth $L$-valued $(n,0)$-form $W_s$ satisfying
\begin{equation}
 \bar\partial W_s=i\omega\wedge p_s,
 \label{eq:fiberwise-smooth-primitive}
\end{equation}
where its smoothness follows from the smoothness of $p_s$ and the regularity of the Green operator.

Let
\(\rho=e^{-\chi}>0\) be smooth, and let \(\bar\partial_\chi^*\) denote the adjoint
with respect to the metric \(h_\Omega e^{-\chi}\).  
For a smooth $L$-valued $(n-1,0)$-form $a$, the K\"ahler identity $\bar\partial_{\chi}^*=-i[\Lambda_{\omega},D^{\prime}_{\chi}]$ gives
\begin{align}
 \bar\partial_\chi^*(i\omega\wedge a)
 &=
 -D'_{\chi}a.
 \label{eq:weighted-adjoint-expanded}
\end{align}
Moreover, 
\begin{align}
 D'(a\rho)
 =
 \rho D'_{\chi}a.
 \label{eq:weighted-product-expanded}
\end{align}
Combining \eqref{eq:weighted-adjoint-expanded} and
\eqref{eq:weighted-product-expanded}, for every smooth \(L\)-valued
\((n,0)\)-form \(W\) we obtain
\begin{equation}
\begin{aligned}
 \int_X
 \langle
   i\omega\wedge a,\bar\partial W
 \rangle_{\omega,h_\Omega}
 \rho\,dV_\omega& =
 \int_X\langle i\omega\wedge a,\bar\partial W\rangle_{\omega,h_{\Omega}\rho}dV_{\omega} =
 \int_X\langle -\rho^{-1}D^{\prime}(a\rho),W\rangle_{\omega,h_{\Omega}\rho} dV_{\omega} \\
 &=
 -\int_X\langle
   D'(a\rho),W\rangle_{\omega,h_{\Omega}} dV_{\omega}.
\end{aligned}
\label{eq:weighted-integration-by-parts}
\end{equation}

Formula \eqref{eq:weighted-integration-by-parts} continues to hold for
\(\rho=e^{-Q_s}\).  Indeed, the fact that
\(e^{-q_{j,s}}\rightarrow e^{-Q_s}\) in \(L^1(X)\) ensures the convergence of the left-hand side. For the right-hand side, the same convergence gives
\begin{align*}
 \int_X\langle D'(ae^{-q_{j,s}}),W\rangle dV_{\omega}
 &=\int_X\langle ae^{-q_{j,s}},(D')^*W\rangle dV_{\omega}\\
 &\longrightarrow\int_X\langle ae^{-Q_s},(D')^*W\rangle dV_{\omega}\\
 &:=\int_X\langle D'(ae^{-Q_s}),W\rangle dV_{\omega}.
\end{align*}

Apply \eqref{eq:weighted-integration-by-parts} with
\(a=p_s\), \(W=W_s\), and \(\rho=e^{-Q_s}\).
Using \eqref{eq:fiberwise-smooth-primitive}, the fact that
\(a\mapsto i\omega\wedge a\) is an isometry, and
$D^{\prime}(p_se^{-Q_s})=0$, we obtain
\begin{align}
 \int_X|p_s|_{\omega,h_\Omega}^2e^{-Q_s}\,dV_\omega
 &=
 \int_X|i\omega\wedge p_s|_{\omega,h_\Omega}^2
        e^{-Q_s}\,dV_\omega \notag\\
 &=
 \int_X
 \left\langle
   i\omega\wedge p_s,\bar\partial W_s
 \right\rangle_{\omega,h_\Omega}
 e^{-Q_s}\,dV_\omega \notag\\
 &=
 -\left\langle
   D'_{X}(p_se^{-Q_s}),W_s
 \right\rangle_X
 =0.
 \label{eq:homogeneous-limit-norm-zero}
\end{align}
Since \(e^{-Q_s}>0\) almost everywhere, this implies \(p_s=0\).
Therefore \(p=0\) almost everywhere, and hence
\(v=\widetilde v\).
\end{proof}
\vspace{0.5cm}
\section{Exact test identities and parameter independence}
\label{sec:smooth-exact-tests}
Recall that $S=\{\tau=s+it\in\mathbb{C}:0<s<1\}$ with measure \(dA=i\,d\tau\wedge d\bar\tau\) and $J_0\Subset (0,1)$ is an open interval.
$u$ is the canonical section of $K_X\otimes L=\mathcal{O}_X$.
We retain the notation \(\rho_j=e^{-q_j}\), \(\rho=e^{-Q}\),
\(\Theta_j=\Theta_{q_j}\), as well as \(\xi_j,R_j\), and
\(S_j=\Theta_j+R_j\geq0\) from the preceding sections.

Let \(U(\tau,x):=w(\tau,x)u(x)\in \mathcal{C}^{\infty}(S\times X,p_X^*(K_X\otimes L))\), where $w$ is a smooth function on $S\times X$ satisfying $w(\tau+2\pi i,x)=w(\tau,x)$ for every $(\tau,x)\in S\times X$ and $\operatorname{supp}(w)\subset K\times i\mathbb{R}\times X$, with $K\Subset J_0$ an open interval.  

Put \(B:=\bar\partial_XU\).  Define
\(D_{\tau,q_j}B:=\partial_\tau B-(q_j)_\tau B\) and
\(D_{\tau,Q}B:=\partial_\tau B-Q_\tau B\), and set
\begin{align}
 J_j(U)
 &:=
 \int_{S_K\times X}
 \left\langle
   i\omega\wedge v_j,D_{\tau,q_j}B
 \right\rangle_{\omega,h_\Omega}
 \rho_j\,dV_\omega\,dA,
 \label{eq:smooth-exact-pairing}\\
 J(U)
 &:=
 \int_{S_K\times X}
 \left\langle
   i\omega\wedge v,D_{\tau,Q}B
 \right\rangle_{\omega,h_\Omega}
 \rho\,dV_\omega\,dA
 \label{eq:limit-exact-pairing}
\end{align}
where $S_K:=K\times i(0,2\pi)$.
The integrals are well defined because $ve^{-Q/2}\in L^2(K\times X)$ and $Q_{\tau}e^{-Q/2}\in L^2(K\times X)$.

\vspace{0.3cm}
The next proposition passes the exact orthogonality from the smooth
approximants to the limiting solution.

\begin{proposition}
\label{prop:smooth-exact-tests}
For every such \(U\),
\begin{equation}
 J_j(U)\longrightarrow0,
 \qquad
 J(U)=0.
 \label{eq:limiting-exact-test}
\end{equation}
\end{proposition}

\begin{proof}
We may assume that $w$ has compact support in $S_K\times X$.
For general $w$, we write $w=w_1+w_2$ so that $w_i(\tau+2\pi i,x)=w_i(\tau,x)$, $w_1$ is supported in $K\times \bigcup_{k\in\mathbb{Z}}i(\frac{\pi}{3}+2k\pi,\frac{5\pi}{3}+2k\pi)$, $w_2$ is supported in $K\times\bigcup_{k\in\mathbb{Z}}i(-\frac{2\pi}{3}+2k\pi,\frac{2\pi}{3}+2k\pi)$.
Let $S_K':=K\times i(-\frac{4\pi}{3},\frac{2\pi}{3})$. 
Since $v_j,q_j$ are $t$-invariant and $B_1=\bar\partial_X(w_1u),B_2=\bar\partial_X(w_2u)$ are periodic in the $t$-variable, it follows that
\begin{align*}
 \int_{S_K\times X}
 \left\langle
   i\omega\wedge v_j,D_{\tau,q_j}B
 \right\rangle_{\omega,h_\Omega}
 \rho_j\,dV_\omega\,dA
 =
  &\int_{S_K\times X}
 \left\langle
   i\omega\wedge v_j,D_{\tau,q_j}B_1
 \right\rangle_{\omega,h_\Omega}
 \rho_j\,dV_\omega\,dA\\
 +
  &\int_{S_K'\times X}
 \left\langle
   i\omega\wedge v_j,D_{\tau,q_j}B_2
 \right\rangle_{\omega,h_\Omega}
 \rho_j\,dV_\omega\,dA.
\end{align*}

\emph{Step 1.} $J_j(U)\rightarrow 0$.
Since the Hermitian pairing $\langle\cdot,\cdot\rangle_{\omega,h_{\Omega}}$ is independent of $\tau$, we compute that
\[
 \begin{split}
 \partial_{\bar\tau}
 \Bigl(
   \langle i\omega\wedge v_j,B\rangle\rho_j
 \Bigr)
 &=\partial_{\bar\tau}\langle i\omega\wedge v_j,B\rangle \rho_j 
 +\langle i\omega\wedge v_j,B\rangle \partial_{\bar\tau}\rho_j\\
 &=\langle i\omega\wedge \partial_{\bar\tau}v_j,B\rangle \rho_j
 +\langle i\omega\wedge v_j,\partial_{\tau}B\rangle\rho_j
 +\langle i\omega\wedge v_j,B\rangle \partial_{\bar\tau}\rho_j\\
 &=\langle i\omega\wedge \partial_{\bar\tau}v_j,B\rangle \rho_j
 +\langle i\omega\wedge v_j,\partial_{\tau}B-\partial_{\tau}q_jB\rangle\rho_j\\
 &=
 \langle
   i\omega\wedge\partial_{\bar\tau}v_j,B
 \rangle\rho_j
 +
 \langle
   i\omega\wedge v_j,D_{\tau,q_j}B
 \rangle\rho_j.
 \end{split}
\]
An integration by parts gives $\int_{S_K}\partial_{\bar\tau}(\langle i\omega\wedge v_j,B\rangle \rho_j)id\tau\wedge d\bar\tau=0$. Therefore
\[
 J_j(U)
 =
 -\int_{S_K\times X}
 \left\langle
   i\omega\wedge\partial_{\bar\tau}v_j,
   \bar\partial_XU
 \right\rangle
 \rho_j\,dV_\omega\,dA.
\]
For the metric \(h_\Omega e^{-q_j}\), the K\"ahler identity
gives
\[
 \bar\partial_{q_j,X}^*
 \bigl(i\omega\wedge\partial_{\bar\tau}v_j\bigr)
 =
 -D'_{q_j,X}(\partial_{\bar\tau}v_j).
\]
Therefore we get
\[
\begin{split}
J_j(U)
&=
-\int_{S_K\times X}
\left\langle
i\omega\wedge\partial_{\bar\tau}v_j,\bar\partial_XU
\right\rangle_{\omega,h_\Omega\rho_j}
dV_\omega\,dA\\
&=
-\int_{S_K\times X}
\left\langle
\bar\partial_{q_j,X}^*
\bigl(i\omega\wedge\partial_{\bar\tau}v_j\bigr),U
\right\rangle_{\omega,h_\Omega\rho_j}
dV_\omega\,dA\\
&=
\int_{S_K\times X}
\left\langle
D'_{q_j,X}(\partial_{\bar\tau}v_j),U
\right\rangle_{\omega,h_\Omega}
\rho_j\,dV_\omega\,dA.
\end{split}
\]
Recall the preceding equation
\(D'_{q_j,X}v_j
=\bigl((q_j)_\tau-(F_j)_\tau\bigr)u\).  Since
\(D'_{q_j,X}=D'_{X}-\partial_Xq_j\wedge\),
we have
\[
\begin{split}
    \partial_{\bar\tau}
 \bigl(D'_{q_j,X}v_j\bigr)
 =
 D'_{q_j,X}(\partial_{\bar\tau}v_j)
 -
 \partial_X(q_j)_{\bar\tau}\wedge v_j.
\end{split}
\]
Because \(v_j=-\iota_{V_j}u\), every \((1,0)\)-form \(\alpha\) satisfies
\(\alpha\wedge v_j=-\alpha(V_j)u\).  The above computations give us
\begin{equation}
 \begin{split}
 D'_{q_j,X}(\partial_{\bar\tau}v_j)
 &=\partial_{\bar\tau}\bigl(D'_{q_j,X}v_j\bigr)+\partial_X(q_j)_{\bar\tau}\wedge v_j\\
 &=\partial_{\bar\tau}\bigl((q_j)_\tau-(F_j)_\tau \bigr)u+\partial_X(q_j)_{\bar\tau}\wedge v_j\\
 &=
 \Bigl(
   (q_j)_{\tau\bar\tau}
   -\bigl(\partial_X(q_j)_{\bar\tau}\bigr)(V_j)
   -(F_j)_{\tau\bar\tau}
 \Bigr)u\\
 &=\bigl((\theta+i\partial\bar\partial q_j)[\partial_\tau-V_j,\overline{\partial_\tau}]-(F_j)_{\tau\bar\tau}\bigr)u\\
 &=
 \Bigl(
   \Theta_j[\xi_j,\overline{\partial_\tau}]
   -(F_j)_{\tau\bar\tau}
 \Bigr)u.
 \end{split}
 \label{eq:differentiated-smooth-chern}
\end{equation}

Therefore, by $|u|^2_{\omega,h_{\Omega}}dV_{\omega}=\Omega$ and
\eqref{eq:differentiated-smooth-chern},
\begin{equation}
 J_j(U)
 =
 \int_{S_K\times X}
 \Bigl(
   \Theta_j[\xi_j,\overline{\partial_\tau}]
   -(F_j)_{\tau\bar\tau}
 \Bigr)
 \overline w\,\rho_j\Omega\,dA.
 \label{eq:smooth-exact-identity}
\end{equation}

Write the mixed-curvature term $\Theta_j[\xi_j,\overline{\partial_{\tau}}]=S_j[\xi_j,\overline{\partial_{\tau}}]-R_j[\xi_j,\overline{\partial_{\tau}}]$. Pointwise positivity of
\(S_j\) gives
\[
 \left|
 S_j[\xi_j,\overline{\partial_\tau}]
 \right|^2
 \leq
 S_j[\xi_j,\overline{\xi_j}]\,
 S_j[\partial_\tau,\overline{\partial_\tau}].
\]
Hence
\begin{align}
 &\left|
 \int_{S_K\times X}
 S_j[\xi_j,\overline{\partial_\tau}]
 \overline w\,\rho_j\Omega\,dA
 \right|^2 \notag\\
 &\qquad\leq
 \left(
  \int_{S_K\times X}
  S_j[\xi_j,\overline{\xi_j}]
  \rho_j\Omega\,dA
 \right)
 \left(
  \int_{S_K\times X}
  |w|^2S_j[\partial_\tau,\overline{\partial_\tau}]
  \rho_j\Omega\,dA
 \right).
 \label{eq:shifted-mixed-CS}
\end{align}
The first factor tends to zero by
\eqref{eq:shifted-curvature-defect}.  We claim that the second factor
is uniformly bounded.

Since the pullback of \(\theta\) has no base--base component,
\(S_j[\partial_\tau,\overline{\partial_\tau}]
=(q_j)_{\tau\bar\tau}
+R_j[\partial_\tau,\overline{\partial_\tau}]\).
For every fixed $x$, we apply the Bochner--Kodaira--Nakano identity on the K\"ahler manifold $(S,i\,d\tau\wedge d\bar\tau)$ and the trivial line bundle on $S$ equipped with the smooth metric $\rho_j(\cdot,x)$ to obtain
\begin{equation}
 \int_{S_K}
 |D_{\tau,q_j}w|^2\rho_jdA
 =
 \int_{S_K}
 \left(
   |\partial_{\bar\tau}w|^2
   +(q_j)_{\tau\bar\tau}|w|^2
 \right)\rho_jdA,
 \label{eq:base-weighted-identity}
\end{equation}

Consequently,
\[
 \int |w|^2(q_j)_{\tau\bar\tau}\rho_j\Omega\,dA
 \leq\int|D_{\tau,q_j}w|^2\rho_j\Omega\,dA.
\]
Since
\[
 |D_{\tau,q_j}w|^2
 \leq2|\partial_\tau w|^2+2|w|^2|(q_j)_\tau|^2,
\]
the density and kinetic estimates from
Proposition~\ref{prop:smooth-demailly-package} give a uniform bound for
this integral.  Moreover,
\(R_j[\partial_\tau,\overline{\partial_\tau}]
 \leq C(\ell_j+\varepsilon_j)\), and
\(0\leq\ell_j\leq c_*\).  Thus
\[
 \sup_j
 \int_{S_{K}\times X}
 |w|^2S_j[\partial_\tau,\overline{\partial_\tau}]
 \rho_j\Omega\,dA<\infty.
\]
It follows from \eqref{eq:shifted-mixed-CS} that
\begin{equation}
 \int_{S_K\times X}
 S_j[\xi_j,\overline{\partial_\tau}]
 \overline w\,\rho_j\Omega\,dA
 \longrightarrow0.
 \label{eq:shifted-mixed-vanishing}
\end{equation}

We now treat \(R_j[\xi_j,\overline{\partial_{\tau}}]\). Since $R_j\geq0$, pointwise
Cauchy--Schwarz gives
\[
 \left|
 R_j[\xi_j,\overline{\partial_\tau}]
 \right|^2
 \leq
 R_j[\xi_j,\overline{\xi_j}]\,
 R_j[\partial_\tau,\overline{\partial_\tau}].
\]
It was proved in
\eqref{eq:D-vanishing} that
\(\int_{S_K\times X}R_j[\xi_j,\overline{\xi_j}]
\rho_j\Omega\,dA\to0\), whereas the definition of $R_j$ gives
\[
\sup_j\int_{S_K\times X}
R_j[\partial_\tau,\overline{\partial_\tau}]
\rho_j\Omega\,dA
<+\infty.
\]  
Therefore
\(\int_{S_K\times X}R_j[\xi_j,\overline{\partial_\tau}]
\overline w\,\rho_j\Omega\,dA\to0\).
Together with \eqref{eq:shifted-mixed-vanishing}, this proves
\begin{equation}
 \int_{S_K\times X}
 \Theta_j[\xi_j,\overline{\partial_\tau}]
 \overline w\,\rho_j\Omega\,dA
 \longrightarrow0.
 \label{eq:unshifted-mixed-vanishing}
\end{equation}

Finally, \(F_j\) is $\operatorname{Im}\tau$-invariant, so
\((F_j)_{\tau\bar\tau}=F_j''/4\).  Since \(w\) is bounded and the
masses \(M_j=\int_X\rho_j\Omega\) are uniformly bounded,
\[
 \left|
 \int_{S_K\times X}
 (F_j)_{\tau\bar\tau}\overline w\,
 \rho_j\Omega\,dA
 \right|
 \leq
 C_U\int_{J_0}|F_j''(s)|\,ds
 \longrightarrow0
\]
by \eqref{eq:smooth-defects-vanish}.  Equation
\eqref{eq:smooth-exact-identity} now gives \(J_j(U)\to0\).

\vspace{0.3cm}
\emph{Step 2.} Verify $J_j(U)\to J(U)$ and $J(U)=0$.  
\begin{align*}
    J_j(U)=\int_{S_K\times X}\langle i\omega\wedge v_j,D_{\tau,q_j}B\rangle \rho_j\Omega dA
    =\int_{[0,2\pi]}dt\int_{K\times X}\langle i\omega\wedge v_j\sqrt{q_j},\sqrt{q_j}D_{\tau,q_j}B\rangle\Omega ds
\end{align*}
Fix $t\in [0,2\pi]$. The convergences obtained in
Propositions~\ref{prop:smooth-demailly-package} give: 
\begin{equation}
 \begin{split}
 \sqrt{\rho_j}D_{\tau,q_j}B\,
 &=
 \sqrt{\rho_j}\partial_\tau B\,
 -
 \sqrt{\rho_j}B(q_j)_\tau\\
 &\longrightarrow
 \sqrt\rho\partial_\tau B\,
 -
 \sqrt\rho BQ_\tau
 =
 \sqrt\rho D_{\tau,Q}B\,
 \end{split}
 \label{eq:exact-test-strong-factor}
\end{equation}
strongly in \(L^2(K\times X, p_X^*(\Lambda^{n,1}T^*X\otimes L))\).
Proposition~\ref{prop:smooth-fixed-exact-limit} gives
\(i\omega\wedge v_j\sqrt{\rho_j}\rightharpoonup
i\omega\wedge v\sqrt\rho\) weakly in \(L^2(K\times X,p_X^*(\Lambda^{n,1}T^*X\otimes L))\).
Hence for every $t\in [0,2\pi]$, 
\begin{align*}
\int_{K\times X}\langle i\omega\wedge v_j\sqrt{q_j}, \sqrt{q_j} D_{\tau,q_j}B\rangle \Omega ds
\to
\int_{K\times X}\langle i\omega\wedge v\sqrt{v},\sqrt{q}D_{\tau,Q}B \rangle \Omega ds.
\end{align*}
Dominated convergence theorem then implies
\begin{align*}
    J_j(U)=\int_{[0,2\pi]}dt\int_{K\times X}\langle i\omega\wedge v_j\sqrt{q_j},\sqrt{q_j}D_{\tau,q_j}B\rangle\Omega ds
    \to
    \int_{[0,2\pi]}dt\int_{K\times X}\langle i\omega\wedge v\sqrt{q},\sqrt{q}D_{\tau,Q}B \rangle \Omega ds=J(U). 
\end{align*}
Since \(J_j(U)\to0\) by \emph{Step 1.}, we obtain \(J(U)=0\).
\end{proof}

\vspace{0.4cm}
Proposition~\ref{prop:smooth-exact-tests} now gives a weak ODE in the finite-dimensional
space of holomorphic vector fields.

\begin{proposition}
\label{prop:smooth-autonomy}
There exists a single vector field
\(V\in H^0(X,T_X^{1,0})\) such that
\[
 v_s=-\iota_Vu
 \qquad\text{for almost every }s\in I.
\]
\end{proposition}

\begin{proof}
Let
\[
 \mathcal H_{\mathrm{ex}}
 :=
 \left\{
  q\in H^0
  \bigl(X,\Lambda^{n-1,0}T_X^*\otimes L\bigr):
  i\omega\wedge q\in\mathcal R_0
 \right\}.
\]
Choose a basis
\(q_\alpha=-\iota_{Z_\alpha}u\), with
\(Z_\alpha\in H^0(X,T_X^{1,0})\) and $\alpha=1,\ldots,N$.
Since \(i\omega\wedge q_\alpha\in\mathcal R_0\), it has a smooth
\(L\)-valued \((n,0)\)-form \(W_\alpha=w_\alpha u\) satisfying
\(\bar\partial W_\alpha=i\omega\wedge q_\alpha\).

By Proposition~\ref{prop:smooth-fixed-exact-limit}, \(v_s\) belongs to
\(\mathcal H_{\mathrm{ex}}\) for almost every \(s\).
We may therefore
write \(v_s=\sum_{\alpha=1}^N c_\alpha(s)q_\alpha\).
It is easy to see $c_{\alpha}\in L^2_{\mathrm{loc}}(I)$. Equip \(\mathcal H_{\mathrm{ex}}\) with the fixed Hermitian inner product
\[
 (q,r)_0
 :=
 \int_X
 \langle q,r\rangle_{\omega,h_\Omega}\,dV_\omega.
\]
After replacing the basis if necessary, we may assume that $(q_\alpha,q_\beta)_0=\delta_{\alpha\beta}$. Then $c_{\alpha}(s)=\int_X\langle v_s,q_{\alpha}\rangle_{\omega,h_{\Omega}}dV_{\omega}$ is measurable thanks to $v\in L^2_{\mathrm{loc}}(I\times X)$. For $K\Subset I$,
$\int_K |c_{\alpha}|^2(s)\,ds\leq \|q_{\alpha}\|_0^2\int_{K\times X}|v|^2_{\omega,h_{\Omega}}dV_{\omega}\,ds<+\infty$ by Proposition~\ref{prop:smooth-fixed-exact-limit} hence $c_{\alpha}\in L^2_{\mathrm{loc}}(I)$.


\vspace{0.4cm}
Define a Hermitian matrix at every $s\in I$ by
\begin{equation}
 G_{\alpha\bar\beta}(s)
 :=
 \int_X
 \langle Z_\alpha,Z_\beta\rangle_\omega
 e^{-Q_s}\Omega.
 \label{eq:autonomy-gram-matrix}
\end{equation}
Let $K\Subset J_0$ be an open interval
and let \(\eta=(\eta_1,\ldots,\eta_N)\in C_c^\infty(K,\mathbb C^N)\).
For every \(\beta\), apply Proposition~\ref{prop:smooth-exact-tests} to
\(U_\beta=\overline{\eta_\beta}\,W_\beta\).  Then
\(\bar\partial_XU_\beta
=\overline{\eta_\beta}\,i\omega\wedge q_\beta\) and
\(D_{\tau,Q}(\bar\partial_XU_\beta)
=(\partial_\tau\overline{\eta_\beta}
-Q_\tau\overline{\eta_\beta})i\omega\wedge q_\beta\).
Using \(v=\sum_{\alpha=1}^Nc_\alpha q_\alpha\) and summing the identities \(J(U_\beta)=0\) over \(\beta=1,\cdots,N\), we obtain
\begin{align}
    0
&=
\sum_{\beta=1}^N
\int_{S_K\times X}
\left\langle
 i\omega\wedge v,
 D_{\tau,Q}(\bar\partial_XU_\beta)
\right\rangle
e^{-Q}\,dV_\omega\,dA
\notag\\
&=
\int_{S_K\times X}
\sum_{\alpha,\beta}
c_\alpha
\left\langle
 i\omega\wedge q_\alpha,
 \bigl(
   \partial_\tau\overline{\eta_\beta}
   -Q_\tau\overline{\eta_\beta}
 \bigr)i\omega\wedge q_\beta
\right\rangle
e^{-Q}\,dV_\omega\,dA.
\label{eq:autonomy-test-expanded-1}
\end{align}
The identity
\[
 \left\langle
  i\omega\wedge q_\alpha,
  i\omega\wedge q_\beta
 \right\rangle_{\omega,h_\Omega}
 dV_\omega
 =
 \langle Z_\alpha,Z_\beta\rangle_\omega\Omega
\]
therefore transforms \eqref{eq:autonomy-test-expanded-1} into
\begin{align}
0
&=
\int_{S_K}
\sum_{\alpha,\beta}c_\alpha
\left[
 \bigl(\partial_{\bar\tau}\eta_\beta\bigr)
 \int_X
 \langle Z_\alpha,Z_\beta\rangle_\omega
 e^{-Q}\Omega
\right.
\notag\\
&\hspace{4.7cm}\left.
 -
 \eta_\beta
 \int_X
 Q_{\bar\tau}
 \langle Z_\alpha,Z_\beta\rangle_\omega
 e^{-Q}\Omega
\right]dA
\notag\\
&=
\int_{S_K}
\sum_{\alpha,\beta}c_\alpha
\left[
 G_{\alpha\bar\beta}
 \partial_{\bar\tau}\eta_\beta
 -
 \eta_\beta
 \int_X
 Q_{\bar\tau}
 \langle Z_\alpha,Z_\beta\rangle_\omega
 e^{-Q}\Omega
\right]dA.
\label{eq:autonomy-test-expanded-2}
\end{align}
We compute the weak derivative $\dot G_{\alpha\bar\beta}$: let $\chi(s)$ be a test function on $I$,
\begin{align}
 \bigl( \dot G_{\alpha\bar\beta},\chi\bigr)
 &=
 -\int_IG_{\alpha\bar\beta}\dot\chi ds\\
 &=
 -\int_X\langle Z_{\alpha},Z_{\beta}\rangle_{\omega}\Omega\int_I\dot\chi e^{-Q}ds\\
 &=
 -\int_X\langle Z_{\alpha},Z_{\beta}\rangle_{\omega}\Omega\int_I\chi \dot Q e^{-Q}ds\\
 &=
 \bigl(-\int_X
 \dot Q
 \langle Z_\alpha,Z_\beta\rangle_\omega
 e^{-Q}\Omega,\chi\bigr).
 \label{eq:autonomy-Gram-derivative}
\end{align}
Here the third equality uses the fact that, for fixed $x\in X$, the function $s\mapsto e^{-Q(s,x)}$ is locally Lipschitz, hence its weak derivative coincides with its pointwise derivative.
So far we have proved that $\dot G_{\alpha\bar\beta}=-\int_X\dot Q\langle Z_\alpha,Z_\beta\rangle_\omega
 e^{-Q}\Omega$ weakly.
Moreover, the Cauchy--Schwarz inequality gives
\[
\begin{split}
    \int_K\Big|\int_X\dot Q\langle Z_\alpha,Z_\beta\rangle_\omega
 e^{-Q}\Omega\Big|^2ds
 &\leq \int_K\bigl(\int_X|\dot Q|^2e^{-Q}\Omega \bigr)\bigl(\int_X|\langle Z_{\alpha}, Z_{\beta}\rangle_{\omega}|^2e^{-Q}\Omega\bigr)\,ds\\
 &\leq C\int_{K\times X}|\dot Q|^2e^{-Q}\Omega\,ds<+\infty
\end{split}
 \]
 where the finiteness of the last integral follows from \eqref{eq:canonical-kinetic-bound}. Therefore we get
\(G\in W^{1,2}_{\mathrm{loc}}(I)\).

\vspace{0.3cm}
It is well-known that $W^{1,1}(K)$ coincides with the space of absolutely continuous functions, so $G_{\alpha\bar\beta}$ is absolutely continuous.
Consequently, the expression in brackets in
\eqref{eq:autonomy-test-expanded-2} is
\(G_{\alpha\bar\beta}\partial_{\bar\tau}\eta_\beta
+(\partial_{\bar\tau}G_{\alpha\bar\beta})\eta_\beta
=\partial_{\bar\tau}(G_{\alpha\bar\beta}\eta_\beta)\).
It follows that
\begin{equation}
 \int_{S_K}
 \sum_{\alpha,\beta}
 c_\alpha\,
 \partial_{\bar\tau}
 \bigl(G_{\alpha\bar\beta}\eta_\beta\bigr)
 \,dA=0.
 \label{eq:autonomy-complex-gram-ode}
\end{equation}
All the coefficients are invariant in $\operatorname{Im}\tau$-direction.  Hence
\(\partial_{\bar\tau}(G_{\alpha\bar\beta}\eta_{\beta})=\frac12\partial_s(G_{\alpha\bar\beta}\eta_{\beta})\) and \eqref{eq:autonomy-complex-gram-ode} becomes
\begin{equation}
 \int_K
 \sum_{\alpha,\beta}
 c_\alpha(s)
 \frac{d}{ds}
 \bigl(
  G_{\alpha\bar\beta}(s)\eta_\beta(s)
 \bigr)\,ds=0.
 \label{eq:gram-ode}
\end{equation}

Since $K$ is one-dimensional, \(W^{1,2}(K)\) is an algebra. 
For $v=(v_1,\cdots,v_N)\in\mathbb{C}^N$ we have $v G v^{\dagger}=\int_X|\sum_{\alpha}v_{\alpha}Z_{\alpha}|^2_{\omega}e^{-Q_s}\Omega$.
Then it is clear that for some $\lambda>0$, the eigenvalues of every $G(s)$, $s\in K$, are greater than $\lambda$. Therefore $\operatorname{det}G\geq\lambda^N$ on $K$.
The inverse matrix $G^{-1}$ is then well-defined and every entry $(G^{-1})_{\alpha\bar\beta}\in W^{1,2}(K)$. 
Since every element in $W^{1,2}(K)$ with compact support can be approximated by functions in $C_c^{\infty}(K)$, \eqref{eq:gram-ode} still holds for every \(\eta\in W^{1,2}(K,\mathbb C^N)\) with compact support.  
For every $\gamma=(\gamma_1,\cdots,\gamma_N)\in C_{c}^{\infty}(K,\mathbb{C}^N)$, $\eta:=G^{-1}\gamma\in W^{1,2}(K,\mathbb{C}^N)$ has compact support.
Therefore \eqref{eq:gram-ode} implies
\[
 \int_K\sum_{\alpha=1}^Nc_\alpha(s)\frac{d}{ds}\gamma_\alpha(s)\,ds=0
 \qquad
 \text{for every }\gamma\in C_c^\infty(K,\mathbb C^N).
\]
Thus \(c_\alpha'=0\) in distributions for every \(\alpha\).  The
coefficients \(c_\alpha\) are therefore constant almost everywhere, and
\(V:=\sum_{\alpha=1}^Nc_\alpha Z_\alpha\) satisfies
\(v_s=-\iota_Vu\) for almost every \(s\).
\end{proof}

\vspace{0.3cm}

\section{Solution to Berndtsson's problem}
\label{sec:full-current-nullity}


We are ready to prove Theorem~\ref{thm:main}.
\subsection{Flow of holomorphic vector field and Cartan calculus of currents}
\label{sec:transport-and-currents}

Let $X$ be a compact complex manifold and $V$ be a holomorphic vector field, we set
\[
 \operatorname{Re}V:=\frac{V+\overline V}{2},
 \qquad
 \operatorname{Im}V:=\frac{V-\overline V}{2i}.
\]
Both real vector fields are complete since $X$ is compact.  If
$\Lambda:\mathbb C\to\operatorname{Aut}^0(X)$ is the holomorphic
flow generated by $V/2$, then $s\mapsto\Lambda_s$ is the real flow
of $\operatorname{Re}V$, while $t\mapsto\Lambda_{it}$ is the real flow
of $-\operatorname{Im}V$.  This convention accounts for the factor
$1/2$ in \eqref{eq:main-full-nullity}.

\vspace{0.5cm}
Let \(S\) be a current of degree \(k\) on a complex manifold $M$ of
complex dimension $m$, and let \(Z\) be a smooth complex vector field.
The contraction \(\iota_ZS\) is defined coefficientwise as follows.
In local real coordinates \(x^1,\ldots,x^{2m}\), write
\[
 S=\sum_{|I|=k}S_I\,dx^I,
 \qquad
 Z=\sum_{j=1}^{2m} Z^j\frac{\partial}{\partial x^j},
\]
where the \(S_I\) are distributions.  Then
\[
 \iota_ZS
 :=
 \sum_{|I|=k}\sum_{\ell=1}^k
 (-1)^{\ell-1}
 Z^{i_\ell}S_I\,
 dx^{i_1}\wedge\cdots
 \wedge\widehat{dx^{i_\ell}}\wedge\cdots
 \wedge dx^{i_k}.
\]
It is straightforward to see the following.
\begin{lemma}
\label{lem:current-calculus}
We have
\begin{equation}
 \langle\iota_ZS,\eta\rangle
 =
 (-1)^{k+1}\langle S,\iota_Z\eta\rangle,
 \label{eq:current-contraction-definition}
\end{equation}
for every compactly supported test form \(\eta\) of degree
\(2m-k+1\).
\end{lemma}
We introduce the Lie derivative of a current with respect to a smooth vector field.
\begin{lemma}\label{lem:Lie-derivative-and-flow}
Let \(Y\) be a smooth vector field. The Lie derivative of the current $S$ with respect to $Y$ is defined by
\[
 \langle\mathcal L_YS,\eta\rangle
 :=
 -\langle S,\mathcal L_Y\eta\rangle.
\]  
Then Cartan's formula holds:
\begin{equation}
 \mathcal L_YS=d\iota_YS+\iota_YdS.
 \label{eq:Cartan}
\end{equation}
If $Y$ is a complete real vector field, its flow \(F_s:M\rightarrow M\) is defined for every $s\in\mathbb{R}$. 
Then
\begin{equation}
 \frac d{ds}F_s^*S
 =
 F_s^*(\mathcal L_YS).
 \label{eq:current-Cartan-flow}
\end{equation}
The identity means that, for
every compactly supported test form \(\eta\), the function
\(s\mapsto\langle F_s^*S,\eta\rangle\) is differentiable and
\[
 \frac d{ds}\langle F_s^*S,\eta\rangle
 =
 \langle F_s^*(\mathcal L_YS),\eta\rangle.
\]
In particular, $\mathcal{L}_YS=0$ implies that $F^*_sS=S$ for every $s$.
\end{lemma}

\begin{proof}

Using the definitions of \(dS\) and \(\iota_YS\), one obtains
\[
 \langle d\iota_YS+\iota_YdS,\eta\rangle
 =
 -\langle S,(d\iota_Y+\iota_Yd)\eta\rangle=-\langle S,\mathcal L_Y\eta\rangle.
\]
This proves
\[
 \mathcal L_YS=d\iota_YS+\iota_YdS.
\]
Since each \(F_s\) is a diffeomorphism, the pullback of an
arbitrary current is well defined by
\[
 \langle F_s^*S,\eta\rangle
 :=
 \langle S,(F_{-s})^*\eta\rangle.
\]
Fix a compactly supported smooth test form \(\eta\).  The map
\[
 s\longmapsto(F_{-s})^*\eta
\]
is smooth. More precisely, for fixed $s_0$, if $|s-s_0|<\varepsilon$ is sufficiently small, by a partition of unity, we may assume the support of $(F_{-s})^*\eta$ is contained in a coordinate ball $B$ for every $s$.
Moreover, $(F_{-s})^*\eta=\sum_{|I|=2m-k} \eta_I(s,x)dx^I$ for some smooth functions $\eta_I$ on $(s_0-\varepsilon,s_0+\varepsilon)\times B$.
It follows that
\[
 s\longmapsto
 \langle F_s^*S,\eta\rangle
 =
 \langle S,(F_{-s})^*\eta\rangle
\]
is differentiable and $\frac{d}{ds}\langle S,(F_{-s})^*\eta\rangle=\langle S,\frac{d}{ds} (F_{-s})^*\eta\rangle$. The usual properties of the flow give
\[
 \frac d{ds}(F_{-s})^*\eta
 =
 -(F_{-s})^*(\mathcal L_Y\eta).
\]
Therefore
\begin{align*}
 \frac d{ds}\langle F_s^*S,\eta\rangle
 &=
 -\langle S,(F_{-s})^*(\mathcal L_Y\eta)\rangle.
\end{align*}
Since \(F_s\) is the flow of \(Y\), it preserves its generating
vector field.  Consequently \(\mathcal L_Y\) commutes with pullback
by \(F_s\), and hence
\[
 (F_{-s})^*(\mathcal L_Y\eta)
 =
 \mathcal L_Y\bigl((F_{-s})^*\eta\bigr).
\]
It follows that
\begin{align*}
 \frac d{ds}\langle F_s^*S,\eta\rangle
 &=
 -\langle S,
   \mathcal L_Y((F_{-s})^*\eta)\rangle\\
 &=
 \langle\mathcal L_YS,(F_{-s})^*\eta\rangle\\
 &=
 \langle F_s^*(\mathcal L_YS),\eta\rangle.
\end{align*}
Since \(\eta\) is arbitrary, the proof is complete.
\end{proof}

\subsection{Nullity of the contraction current}
In this section we prove equation \eqref{eq:main-full-nullity}.
\begin{lemma}
\label{lem:weighted-parallel-nullity}
Let \((M,\omega)\) be a K\"ahler manifold of dimension \(m\), let \((E,h)\) be a
smooth Hermitian holomorphic line bundle, and let \(q\) be a quasi-psh function
satisfying \(e^{-q}\in L^1_{\mathrm{loc}}(M)\).  Suppose that \(\alpha\) is a
nowhere-vanishing holomorphic \(E\)-valued \((m-1,0)\)-form and that
\(\xi\) is a nowhere-vanishing holomorphic vector field satisfying
\(\iota_\xi\alpha=0\).  If \(D'_h(\alpha e^{-q})=0\) in distributions,
then
\begin{equation}
 \iota_\xi(\Theta_h+\ddc q)=0
 \label{eq:weighted-parallel-nullity}
\end{equation}
as a \((0,1)\)-current on $M$.
\end{lemma}

\begin{proof}
The assertion is local.  Fix a point and choose a holomorphic flow box
\((w,z^1,\ldots,z^{m-1})\), with
\(\xi=\partial/\partial w\).
If \(e\) is a local holomorphic frame of \(E\), then
\(\iota_\xi\alpha=0\) implies
\[
 \alpha
 =
 a(w,z)\,
 dz^1\wedge\cdots\wedge dz^{m-1}\otimes e.
\]
The coefficient \(a\) is a non-vanishing holomorphic function.  
Replacing \(e\)
by the holomorphic frame \(ae\), we may therefore assume that
\(\alpha=dz^1\wedge\cdots\wedge dz^{m-1}\otimes e\).  Write
\(|e|_h^2=e^{-\kappa}\). Then \(\Theta_h=\ddc\kappa\).

Since \(e^{-q}\in L^1_{\mathrm{loc}}\) and \(\kappa\) is smooth,
\(e^{-(q+\kappa)}\in L^1_{\mathrm{loc}}\).  
Let $q_j\searrow q$ be the standard local convolution of $q$.
In the chosen frame, the only possibly
nonzero coefficient of \(D'_h(\alpha e^{-q_j})\) is its
\(dw\wedge dz^1\wedge\cdots\wedge dz^{m-1}\) coefficient.  Explicitly,
\[
 \begin{split}
 D'_h(\alpha e^{-q_j})
 &=
 \bigl(
  \partial_w(e^{-q_j})-\kappa_w e^{-q_j}
 \bigr)
 dw\wedge dz^1\wedge\cdots\wedge dz^{m-1}\otimes e\\
 &=
 e^\kappa(\partial_w e^{-(q_j+\kappa)})\,
 dw\wedge dz^1\wedge\cdots\wedge dz^{m-1}\otimes e.
 \end{split}
\]
Our assumption $D'_h(\alpha e^{-q})=0$ means $\langle \alpha e^{-q},(D'_{h})^*\beta\rangle_{\omega,h}=0$ for every smooth compactly supported $E$-valued $(m,0)$-form $\beta$.
Hence $\langle D'_h(\alpha e^{-q_j}),\beta\rangle_{\omega,h}=\langle \alpha e^{-q_j},(D'_h)^*\beta\rangle_{\omega,h}\to 0$.
Since $\beta$ is arbitrary, we obtain $\partial_{w}e^{-(q_j+\kappa)}\to 0$ in the sense of distributions. Therefore
\begin{equation}
 \partial_we^{-(q+\kappa)}=0
 \label{eq:real-weight-independent-w}
\end{equation}
in distributions.

Taking the complex conjugate of
\eqref{eq:real-weight-independent-w} therefore gives
\(\partial_{\bar w}e^{-(q+\kappa)}=0\).  Thus real weak derivatives of \(e^{-(q+\kappa)}\) in
the \(w\)-direction vanish.  After shrinking the coordinate ball, it follows that
\(e^{-(q+\kappa)(w,z)}=F_0(z)\) for some $L^1_{\mathrm{loc}}$ function $F_0$ almost everywhere.
Equivalently, $q+\kappa=-\log F_0$ a.e. It is clear that we can choose a quasi-psh function $g$ so that $g(z)=-\log F_0(z)$ almost everywhere. Then two quasi-psh functions $q+\kappa$ and $g$ coincide almost everywhere, hence they actually coincide everywhere.
In particular 
\(q+\kappa\) is independent of \(w\),  which gives
\(\partial_w(q+\kappa)=0\) in distributions.  Since
\(\Theta_h+i\partial\bar\partial q=\ddc(q+\kappa)\), we conclude that
\[
 \iota_\xi(\Theta_h+i\partial\bar\partial q)
 =
 i\bar\partial\bigl(\partial_w(q+\kappa)\bigr)
 =0.
\]
\end{proof}

We now apply the preceding lemma to obtain equation \eqref{eq:main-full-nullity}.

\begin{proposition}
\label{prop:full-total-space-nullity}
Let \(V\) be the holomorphic vector field on $X$ obtained in
Proposition~\ref{prop:smooth-autonomy}, $v$ be the $p_X^*L$-valued $(n-1)$-form on $X\times S$ obtained in Proposition~\ref{prop:smooth-fixed-exact-limit}.
Then
\(v=-\iota_Vu\) on \(X\times S\).  Put
\[
 \xi:=\frac{\partial}{\partial\tau}-V,
 \qquad
 \widehat u:=u-d\tau\wedge v.
\]
Then
\begin{equation}
 D'_{h_{\Omega},X\times S}(\widehat u e^{-Q})=0
 \label{eq:total-weighted-parallelity}
\end{equation}
in the sense of distributions on \(X\times S\), and
\begin{equation}
 \iota_{\partial/\partial\tau-V}
 \bigl(p_X^*\theta+\ddc_{x,\tau}Q\bigr)=0
 \label{eq:full-total-space-nullity}
\end{equation}
in the sense of currents on $X\times S$.
\end{proposition}

\begin{proof}
Put \(\rho=e^{-Q}\), and regard \(V\), \(v\), and \(\rho\) as invariant
objects on \(X\times S\).  Since \(v=-\iota_Vu\), one has
\(\widehat u=u+d\tau\wedge\iota_Vu
=\iota_\xi(d\tau\wedge u)\).
It follows that \(\iota_\xi\widehat u=0\).  It is easy to see that \(\widehat u\in H^0(X\times S,\Lambda^{n,0}T^*(X\times S)\otimes p_X^*L)\) is
holomorphic.  Its restriction to every vertical fiber is $u$;
hence it is nowhere vanishing.  The vector field \(\xi\) is nowhere
vanishing because \(d\tau(\xi)=1\).

The chain rule \eqref{eq:canonical-chain-rule} and the 
divergence-form equation \eqref{eq:limit-conservative-equation} give
\begin{equation}
 \partial_\tau\rho=-Q_\tau\rho,
 \qquad
 D'_{h_{\Omega},X}(v\rho)=Q_\tau u\rho
 \label{eq:total-parallel-inputs}
\end{equation}
in distributions on \(X\times S\). 
The second equality means that for every $\Phi\in C_c^{\infty}(X\times S,p_X^*(\Lambda^{n,0}T^*X\otimes L))$,
\begin{align}
    \int_{X\times S}\langle v\rho,(D'_{h_{\Omega,X}})^*\Phi\rangle dV_{\omega}dA 
    =
    \int_{X\times S}\langle Q_{\tau}u\rho,\Phi\rangle dV_{\omega}dA.
    \label{eq:weak covariant derivative on S}
\end{align}
In fact this equality was obtained on \(X\times I\), see \eqref{eq:weak covariant derivative}. Since $\rho$ and $Q$ are $\operatorname{Im}\tau$-invariant, \eqref{eq:weak covariant derivative on S} is reduced to \eqref{eq:weak covariant derivative} by Fubini's theorem.

Let $v_j$ be the approximation sequence used before, recall that they satisfy (see Proposition \ref{prop:smooth-chern-family})
\begin{align*}
    D'_{X}(v_j\rho_j)=r_ju\rho_j=\bigl((q_j)_\tau-(F_j)_\tau\bigl)u\rho_j \qquad \text{on } X\times S.
\end{align*}
A direct computation shows that
\begin{align*}
    D'_{X\times S}\bigl((u-d\tau\wedge v_j)\rho_j\bigl)
    &=d\tau\wedge
 \bigl(
  (\partial_\tau\rho_j)u
  +D'_{X}(v_j\rho_j)
 \bigr)\notag\\
    &=d\tau \wedge \bigl( -(q_j)_\tau u\rho_j+r_ju\rho_j\bigl)\\
    &=d\tau\wedge\bigl(-(F_j)_{\tau}u\rho_j\bigl).
\end{align*}
The right-hand side strongly converges to $0$ in $L^1_{\mathrm{loc}}(X\times S)$. 
In fact, $\int_X\rho_j(x,\tau)\Omega(x)\leq\int_X\rho(x,\tau)\Omega(x)=1$ for every $\tau$ and $(F_j)_\tau\to0$ in $L^1_{\rm{loc}}(I)$ by \eqref{eq:Fj-speed-L2}.

On the other hand, take $\Psi\in C_c^{\infty}(X\times S,K_{X\times S}\otimes p_X^*L)$, we verify that
\begin{align}
\begin{split}
    \int_{X\times S}\langle D'_{X\times S}\bigl(u-d\tau\wedge v_j\bigl)\rho_j,\Psi\rangle_{\hat{\omega},h_{\Omega}} dV_{\hat{\omega}}
&:=
    \int_{X\times S}\langle u\rho_j-d\tau\wedge v_j\rho_j,(D'_{X\times S})^*\Psi\rangle_{\hat{\omega},h_{\Omega}}dV_{\hat{\omega}}\\
 &   \longrightarrow
    \int_{X\times S}\langle u\rho-d\tau\wedge v\rho,(D_{X\times S}')^*\Psi\rangle_{\hat{\omega},h_{\Omega}}dV_{\hat\omega},
\end{split}
    \label{eq:product space convergence}
\end{align}
where $\hat\omega=\omega+id\tau\wedge d\bar\tau$ is the K\"ahler form on $X\times S$.
The convergence $\int\langle u\rho_j,(D'_{X\times S})^*\Psi\rangle dV_{\hat\omega}\to\int\langle u\rho,(D'_{X\times S})^*\Psi\rangle dV_{\hat\omega}$ is clear.
To deal with the second term, we decompose $p_X^*L$-valued $(n,0)$-form $(D'_{X\times S})^*\Psi$ into $d\tau\wedge \varphi+\psi$, so that $\varphi\in C_c^{\infty}(X\times S,p_X^*(\Lambda^{n-1,0}T^*X\otimes L))$ and $\psi$ only has vertical terms $dx$.
It is easy to see 
\begin{align*}
\langle d\tau\wedge v_j\rho_j,(D'_{X\times S})^*\Psi\rangle_{\hat\omega,h_{\Omega}}
=\langle d\tau\wedge v_j\rho_j,d\tau\wedge\varphi\rangle_{\hat\omega,h_{\Omega}}
=\langle v_j\rho_j,\varphi\rangle_{\omega,h_{\Omega}}
\end{align*}
 and similarly $\langle d\tau\wedge v\rho,(D'_{X\times S})^*\Psi\rangle_{\hat\omega,h_{\Omega}}=\langle v\rho,\varphi\rangle_{\omega,h_{\Omega}}$ thanks to $\hat{\omega}=\omega+id\tau\wedge d\bar\tau$.
Therefore \eqref{eq:product space convergence} follows from the weak convergence $v_j\sqrt{\rho_j}\rightharpoonup v\sqrt{\rho}$ and strong convergence $\sqrt{\rho_j}\to\sqrt{\rho}$ in $L^2_{\mathrm{loc}}(X\times I)$ in Proposition~\ref{prop:smooth-fixed-exact-limit} and Proposition~\ref{prop:smooth-demailly-package}.
Consequently we prove
\begin{align*}
    \left\langle D_{X\times S}'\bigl((u-d\tau\wedge v)\rho\bigl),\Psi\right\rangle
    :=
    \left\langle (u-d\tau\wedge v)\rho,(D'_{X\times S})^*\Psi\right\rangle
    =0.
\end{align*}
This proves \eqref{eq:total-weighted-parallelity}.

\vspace{0.3cm}
Apply Lemma~\ref{lem:weighted-parallel-nullity} to
\(E=p_X^*L\), \(h=p_X^*h_\Omega\), \(q=Q\), and
\(\alpha=\widehat u\).
The curvature of \(p_X^*h_\Omega\) is \(p_X^*\theta\), and the preceding
paragraphs verify that \(\widehat u\) and \(\xi\) are nowhere vanishing, hence satisfy all the
hypotheses of the lemma.  We therefore obtain
\[
 \iota_{\partial/\partial\tau-V}
 \bigl(p_X^*\theta+\ddc_{x,\tau}Q\bigr)=0
\]
as a $(0,1)$-current on $X\times S$.
\end{proof}

\subsection{Transport of the normalized measures and currents}
\label{sec:smooth-transport}

In this section we prove the transport equations \eqref{eq:main-measure-transport} and \eqref{eq:main-current-transport}.
The measure transport \eqref{eq:main-measure-transport} essentially follows from the $D^\prime$-equation
$D'_X(ve^{-Q_s})=(\partial_\tau Q)_s e^{-Q_s}u$, whereas the current transport \eqref{eq:main-current-transport} directly follows from
the nullity equation \eqref{eq:main-full-nullity}.

\begin{proposition}
\label{prop:smooth-transport}
Put \(W:=-(V+\bar V)\) and let \(\{\Phi_t\}_{t\in\mathbb{R}}\) be the flow of \(W\).  Then, for every
\(s,s_0\in I\),
\begin{align}
 \Phi_{s-s_0}^*(e^{-Q_s}\Omega)
 &=e^{-Q_{s_0}}\Omega,
 \label{eq:smooth-measure-transport}\\
 \Phi_{s-s_0}^*(\theta+\ddc_XQ_s)
 &=\theta+\ddc_XQ_{s_0}.
 \label{eq:smooth-current-transport}
\end{align}
\end{proposition}

\begin{proof}
\emph{Part 1. Proof of the transport formula of measures \eqref{eq:smooth-measure-transport}.}

We first derive an elementary
equivalence: for $\rho,a\in L^1(X)$,
\begin{equation}
 D'_{X}(-\iota_Vu\,\rho)=a\,u
 \quad\Longleftrightarrow\quad
 -\partial_X\bigl(\iota_V(\rho\Omega)\bigr)=a\,\Omega.
 \label{eq:chern-divergence-identity}
\end{equation}
Every smooth $L$-valued $(n,0)$-form $\Psi$ on $X$ has the form $\Psi=wu$, $w\in C^{\infty}(X)$.
The left-hand side means for every $w\in C^{\infty}(X)$ we have 
\begin{equation}
    \int_X\langle-\iota_Vu\rho,(D'_X)^*(wu)\rangle dV_{\omega}=\int_X\langle au,wu\rangle dV_{\omega}=\int_Xa\bar{w}\Omega.
    \label{eq:char-left}
\end{equation}
The right-hand side means for every $w\in C^{\infty}(X)$,
\begin{equation}
    \int_X\partial_X\bar w\wedge \iota_V(\rho\Omega)=\int_X a\bar w\Omega.
    \label{eq:char-right}
\end{equation}

Indeed, we only need to prove that for every $w\in C^{\infty}(X)$, the following equality of volume forms holds:
\begin{equation}
 \langle -\iota_Vu,(D'_X)^*(wu)\rangle_{\omega,h_{\Omega}} dV_{\omega}
=\partial_X\bar w\wedge\iota_V\Omega. 
\label{eq:smooth-equ}
\end{equation}
If \eqref{eq:smooth-equ} is proved, multiplying both sides by $\rho$ and integrating proves \eqref{eq:char-left}.
To prove \eqref{eq:smooth-equ}, take $g\in C^{\infty}(X)$. We compute
\begin{align*}
    \int_Xg\langle -\iota_Vu,(D'_X)^*(wu)\rangle_{\omega,h_{\Omega}} dV_{\omega}
    &=\int_X\langle D'_X(-g\iota_{V}u),wu\rangle_{\omega,h_{\Omega}} dV_{\omega}\\
    &=\int_Xc_n[D'_X(-g\iota_Vu),wu]\\
    &=\int_X-\bar w\partial_X(\iota_V(g\Omega))\\
    &=\int_X\partial_X\bar w\wedge \iota_V(g\Omega).
\end{align*}
The third equality uses $c_n[D'_X(-g\iota_Vu),u]=-\partial_X(\iota_V(g\Omega))$ on $X$, which can be easily verified in a local chart. Thus \eqref{eq:smooth-equ} is proved.

\vspace{0.4cm}

Recall that $v\in L^2_{\mathrm{loc}}(X\times S,p_X^*(\Lambda^{n-1,0}T^*X\otimes L))$ obtained in Proposition~\ref{prop:smooth-fixed-exact-limit} satisfies $v=-\iota_Vu$ for a single $V\in H^0(X,T_X^{1,0})$ obtained in Proposition~\ref{prop:smooth-autonomy}.
Moreover, for almost every $s$, $D_X'(-\iota_Vu e^{-Q_s})=D_X'(v_se^{-Q_s})=(\partial_{\tau}Q)_se^{-Q_s}u$ by Proposition~\ref{prop:smooth-fixed-exact-limit}.
Applying \eqref{eq:chern-divergence-identity}, we get 
\[-\partial_X\bigl(\iota_V(e^{-Q_s}\Omega)\bigr)=(\partial_{\tau}Q)_se^{-Q_s}\Omega, \qquad \text{for almost every s}.\]
Taking the conjugate and adding, we obtain
\begin{equation}
d_X\bigl(\iota_{V+\bar V}(e^{-Q_s}\Omega)\bigr)
=-\dot Q_s\,e^{-Q_s}\Omega, \qquad \dot Q_s=(\partial_s Q)_s.
\label{eq:Lie-derivative}
\end{equation}



Let \(\Phi_t\) be the flow of \(W:=-(V+\bar V)\).  It is defined for all \(t\in\mathbb{R}\)
because \(X\) is compact.  Given \(s_0\in I\) and
\(\varphi\in C^\infty(X)\), set
\(\varphi_s:=\varphi\circ\Phi_{-(s-s_0)}\).  Then
\(\partial_s\varphi_s+W(\varphi_s)=0\). 
We need to prove $s\mapsto \int_X\varphi_se^{-Q_s}\Omega$ is constant.
First, we prove that it is locally absolutely continuous:
\[
\begin{split}
\int_X(\varphi_{r+t}e^{-Q_{r+t}}-\varphi_re^{-Q_r})\Omega
&=\int_X\int_r^{r+t}\partial_s(\varphi_se^{-Q_s})\,ds\,\Omega\\
&=\int_X\int_r^{r+t}\bigl(\partial_s\varphi_se^{-Q_s}-\varphi_s\dot Q_se^{-Q_s}\bigr)\,ds\,\Omega\\
&=\int_r^{r+t}\int_X\bigl(\partial_s\varphi_se^{-Q_s}-\varphi_s\dot Q_se^{-Q_s}\bigr)\Omega\,ds\\
&=\int_r^{r+t}\int_X\bigl(-W(\varphi_s)e^{-Q_s}-\varphi_s\dot Q_se^{-Q_s}\bigr)\Omega\,ds.
\end{split}
\]
Hence 
\[
\Big|\int_X(\varphi_{r+t}e^{-Q_{r+t}}-\varphi_re^{-Q_r})\Omega\Big|
\leq \bigl(\|W(\varphi)\|_{C^0}+\|\varphi\|_{C^0}\bigr)
\int_r^{r+t}\int_X(e^{-Q_s}+|\dot Q_s|e^{-Q_s})\Omega\,ds.
\]
By Lemma~\ref{lem:kinetic}, $\int_K\int_X(e^{-Q_s}+|\dot Q_{s}|e^{-Q_s})\Omega\,ds<\infty$ for every compact interval $K\Subset I$, and absolute continuity follows. Moreover, for almost every $s$ we have 
\begin{equation}
    \frac{d}{ds}\int_X\varphi_se^{-Q_s}\Omega
    =\int_X\bigl(-W(\varphi_s)e^{-Q_s}-\varphi_s\dot Q_se^{-Q_s} \bigr)\Omega.
\end{equation}
Since $\int_X\varphi_sd_X\bigl(\iota_W(e^{-Q_s}\Omega)\bigr)=-\int_Xd_X\varphi_s\wedge \iota_W(e^{-Q_s}\Omega)=-\int_XW(\varphi_s)e^{-Q_s}\Omega$, this together with equality
\eqref{eq:Lie-derivative} gives
\[\frac{d}{ds}\int_X\varphi_s\, e^{-Q_s}\Omega=0.\]
A absolutely continuous function with a.e. zero derivative is constant, hence
\(\Phi_{s-s_0}^*(e^{-Q_s}\Omega)=e^{-Q_{s_0}}\Omega\).
This proves \eqref{eq:smooth-measure-transport}.

\vspace{0.3cm}
\emph{Part 2. Proof of the current transport equation \eqref{eq:smooth-current-transport}.}

Put
\(T:=p_X^*\theta+\ddc_{x,\tau}Q\) and
\(\xi:=\partial/\partial\tau-V\).
By \eqref{eq:full-total-space-nullity},
\(\iota_\xi T=0\).  Since \(T\) is real, also
\(\iota_{\bar\xi}T=0\).  Thus, for
\(\mathcal W:=\xi+\bar\xi=\partial/\partial s+W\),
one has \(\iota_{\mathcal W}T=0\).  As \(dT=0\), Cartan's formula for the current $T$ gives
\(\mathcal L_{\mathcal W}T=d(\iota_{\mathcal{W}}T)+\iota_{\mathcal{W}}(dT)=0\) on $X\times S$.  The flow of \(\mathcal W\) is
\(\Psi_r(x,\tau)=(\Phi_r(x),\tau+r)\).
For a given $r\in\mathbb{R}$, we set
\begin{align*}
    I_r:=\{s\in(0,1):s+r\in(0,1)\}, \qquad 
    S_r:=\{\tau\in S:\operatorname{Re}\tau\in I_r\}.
\end{align*}
Then \(\Psi_r^*T=T\) on $X\times S_r$ by Lemma~\ref{lem:Lie-derivative-and-flow}.  
Taking \(r=s-s_0\) yields
\[
 \Phi_{s-s_0}^*(\theta+\ddc_XQ_s)
 =
 \theta+\ddc_XQ_{s_0}.
\]
In fact, take $(x,s_0)\in X\times\{s_0\}$, and let $\psi_0$ and $\psi$ be local potentials of $\theta$ near $x$ and $\Phi_{s-s_0}(x)$, respectively.
$\Psi_{s-s_0}^*T=T$ means that $\Psi_{s-s_0}^*(\psi+Q)-(\psi_0+Q)$ is pluriharmonic in a neighborhood of $(x,s_0)$, where $\Psi_{s-s_0}(x,s_0)=(\Phi_{s-s_0}(x),s)$.
Restricting to $X\times\{s_0\}$, we get $\ddc_X\bigl[\Phi_{s-s_0}^*(\psi+Q_s)-(\psi_0+Q_{s_0})\bigr]=0$.
This proves \eqref{eq:smooth-current-transport}.
\end{proof}

\subsection{Conclusion of the proof of Theorem~\ref{thm:main}}\label{sec:theorem}

\begin{proof}[Proof of Theorem~\ref{thm:main}]
Proposition~\ref{prop:smooth-fixed-exact-limit}, 
Proposition~\ref{prop:smooth-exact-tests} 
and Proposition~\ref{prop:smooth-autonomy} produce a single
$V\in H^0(X,T_X^{1,0})$. Let $v=-\iota_Vu$, equation $D^{\prime}_X(ve^{-Q_s})=(\partial_\tau Q_s) e^{-Q_s}u$ holds for a.e. $s\in I$.

Proposition~\ref{prop:full-total-space-nullity} proves
\[
 \iota_{\partial/\partial\tau-V}
 \bigl(p_X^*\theta+\ddc_{x,\tau}Q\bigr)=0.
\]
Put $W=-(V+\bar V)$ and $\mathcal V=-2V$. This nullity equation is exactly \eqref{eq:main-full-nullity}.  
By above $D'$-equation, Proposition~\ref{prop:smooth-transport} gives
\[
 \Phi_{s-s_0}^*(e^{-Q_s}\Omega)=e^{-Q_{s_0}}\Omega,
 \qquad
 \Phi_{s-s_0}^*(\theta+\ddc_X Q_s)=\theta+\ddc_X Q_{s_0},
\]
where $\Phi_t$ is the flow of $W$.
Since
$\operatorname{Re}\mathcal V=-(V+\bar V)=W$, so the flow $G_t$ required in the
statement of Theorem~\ref{thm:main} is nothing but $\Phi_t$, 
Finally,
$e^{-Q_s}=e^{F(s)-\phi_s}$ and
$\ddc_XQ_s=\ddc_X\phi_s$.
The two transport equations are exactly
\eqref{eq:main-measure-transport} and
\eqref{eq:main-current-transport}.
\end{proof}

\vspace{0.4cm}
\section{Application I of Theorem~\ref{thm:main}: proof of Theorem~\ref{thm:main-big-BM}}
\label{sec:twisted-data}

\subsection{Preliminaries of pluripotential theory in big classes and finite-energy geometry}

In this section, we assume $(X,\omega)$ is a compact K\"ahler manifold, $\theta$ is a smooth closed real
$(1,1)$-form so that the class $[\theta]\in H^{1,1}(X,\mathbb{R})$ is big.  The ample locus
$\operatorname{Amp}([\theta])$ consists of the points $x\in X$ for
which some K\"ahler current in $[\theta]$ with analytic singularities
is smooth near $x$.
Set
\begin{equation}
 V_\theta
 :=\sup\Big\{v\in\operatorname{PSH}(X,\theta):v\leq0\Big\}.
 \label{eq:BM-minimal-envelope}
\end{equation}
Thus $\sup_XV_\theta=0$.  We say $u\in\operatorname{PSH}(X,\theta)$ has minimal singularities if there is some $C>0$ so that
\begin{align*}
    u\leq V_{\theta}+C,\qquad u\geq V_\theta-C.
\end{align*}
For $u\in\operatorname{PSH}(X,\theta)$ we set
$u_k:=\max\{u,V_\theta-k\}$. Recall from \cite{BEGZ10} that the non-pluripolar Monge--Amp\`ere
measure is defined as
\[
 \left\langle(\theta+\ddc u)^n\right\rangle
 :=\lim_{k\to\infty}{\bf1}_{\{u>V_\theta-k\}}
       (\theta+\ddc u_k)^n.
\]
We set
\[\mathsf V:=\operatorname{Vol}([\theta])
 =\int_X\left\langle(\theta+\ddc V_\theta)^n\right\rangle\]
and
\[\mathcal E(X,\theta)=\{u\in\operatorname{PSH}(X,\theta):\int_X\langle(\theta+dd^cu)^n\rangle=\mathsf{V}\}.\]  
For potential $u\in\operatorname{PSH}(X,\theta)$ with minimal
singularities we set
\begin{equation}
 E_\theta(u)=\frac1{(n+1)\mathsf V}
 \sum_{j=0}^n\int_X(u-V_\theta)
 \left\langle(\theta+\ddc u)^j\wedge
 (\theta+\ddc V_\theta)^{n-j}\right\rangle .
 \label{eq:BM-energy-definition}
\end{equation}
This functional is extended to $\mathcal{E}(X,\theta)$ by setting
\begin{equation}
 E_\theta(u):=\lim_{k\to\infty}E_\theta(u_k),
 \qquad
 \mathcal E^1(X,\theta)
 :=\{u\in\mathcal E(X,\theta):E_\theta(u)>-\infty\}.
 \label{eq:BM-finite-energy-definition}
\end{equation}
It satisfies
\begin{equation}
 E_\theta(u+c)=E_\theta(u)+c.
 \label{eq:BM-energy-translation}
\end{equation}
For $u,v\in\mathcal E^1(X,\theta)$, let
\[
 P_\theta(u,v)
 :=\left(\sup\{w\in\operatorname{PSH}(X,\theta):
                      w\leq\min(u,v)\}\right)^*
\]
be the rooftop envelope.
With the help of these envelopes one can define a complete metric on $\mathcal{E}^1(X,\theta)$ as follows.
For $u,v\in\mathcal{E}^1(X,\theta)$ we have $P_\theta(u,v)\in\mathcal{E}^1(X,\theta)$ (see \cite[Theorem 3.2]{DDL18b}) and let
\[
 d_1(u,v)
 :=E_\theta(u)+E_\theta(v)-2E_\theta(P_\theta(u,v)).
\]

\vspace{0.5cm}
Fix a smooth positive volume form $\Omega$, and a
quasi-psh function $\psi$ such that
\begin{equation}
\begin{split}
     \eta:&=\operatorname{Ric}(\Omega)-\theta,\\
     \eta_\psi:&=\operatorname{Ric}(\Omega)-\theta+\ddc\psi\geq0,\\
 \qquad e^{-V_{\theta}-\psi}&\in L^1(X,\Omega).
\end{split}
 \label{eq:BM-twisted-data}
\end{equation}
Note that by \cite[Proposition 2.5]{DZ24}, $e^{-\varphi-\psi}\in L^1(X,\Omega)$ holds for every $\varphi\in\mathcal{E}(X,\theta)$.

\begin{definition}
\label{def:BM-normalized-solution}
A function
$\varphi\in\mathcal E^1(X,\theta)$ is called a normalized finite-energy solution of the
$\eta_{\psi}$-twisted Monge--Am\`ere equation if it satisfies
\begin{equation}
 \left\langle(\theta+\ddc\varphi)^n\right\rangle
 =e^{-\varphi-\psi}\Omega.
 \label{eq:BM-normalized-equation}
\end{equation}
We write $T_\varphi:=\theta+\ddc\varphi$ and denote the set of such
resulting currents by $\mathcal S(\theta,\psi)$.  
\end{definition}

By openness theorem \cite{Ber15a,GZ15}, the right-hand side of \eqref{eq:BM-normalized-equation} belongs to $L^p(X)$ for some $p>1$.
Then \cite[Theorem~B]{BEGZ10} implies that $\varphi$ has minimal singularities.
For such a solution, the Ricci current of its non-pluripolar measure satisfies
\begin{equation}
 \operatorname{Ric}(T_\varphi)
 :=\operatorname{Ric}( T_\varphi^n\rangle)
 =\operatorname{Ric}(\Omega)+\ddc(\varphi+\psi)
 = T_\varphi+\eta.
 \label{eq:BM-Ricci-current-equation}
\end{equation}
As in \cite{DZ24}, because of the strong openness theorem \cite{GZ15}, for $\lambda>0$ we define the $\lambda$-Ding functional:
\begin{equation}
 \mathcal D_{\psi}^{\lambda}(u)
 :=-\frac1\lambda\log\left(
       \frac1{\mathsf V}\int_Xe^{-\lambda u-\psi}\Omega
      \right)-E_\theta(u), \qquad u\in \mathcal{E}^1(X,\theta).
 \label{eq:BM-lambda-Ding}
\end{equation}
It is invariant under addition of constants by
\eqref{eq:BM-energy-translation}. 
We write
\[
 d\mu_\psi:=e^{-\psi}\Omega,
 \qquad
 \mathcal L_{\mu_\psi}^{\lambda}(u)
 :=-\log\left(\frac1{\mathsf V}\int_Xe^{-\lambda u}\,d\mu_\psi\right).
\]
If $\lambda=1$ we may omit the superscript $\lambda$.

We first recall the following property, saying that the $1$-Ding functional is constant along the weak geodesic connecting two normalized solutions:

\begin{lemma}
\label{lem:BM-Ding-flat}
Let $\varphi_0,\varphi_1$ be two normalized solutions and let
$(\varphi_s)_{0\leq s\leq1}$ be their finite-energy weak geodesic.
Then every $\varphi_s$ is a normalized solution of \eqref{eq:BM-normalized-equation}, and
\begin{equation}
 E_\theta(\varphi_s)=\textup{constant},\qquad
 \mathcal D_{\psi}(\varphi_s)=\textup{constant},\qquad
 \int_Xe^{-\varphi_s-\psi}\Omega=\mathsf V.
 \label{eq:BM-flat-identities}
\end{equation}
Moreover, there is a constant $C>0$ such that, for all $s,t\in[0,1]$,
\[
 |\varphi_t-\varphi_s|\leq C|s-t|.
\]
\end{lemma}

\begin{proof}
Since \(\varphi_0\) and \(\varphi_1\) have minimal singularities, their
difference is bounded.  Hence
\cite[Lemma~3.1]{DDL18b} gives
\begin{equation}
 |\varphi_t-\varphi_s|
 \leq C|t-s|,
 \qquad s,t\in[0,1],
 \label{eq:BM-time-Lipschitz}
\end{equation}
for some \(C>0\).  In particular, every \(\varphi_s\) has minimal
singularities.

By \cite[Proposition~5.7]{DZ24} and \cite[Theorem~0.1]{BP08}, 
the endpoint
solutions $\varphi_0,\varphi_1$ are global minimizers of \(\mathcal D_{\psi}\), and
\(s\mapsto\mathcal D_{\psi}(\varphi_s)\) is continuous and convex.  If we set \(m:=\inf_{\mathcal E^1(X,\theta)}\mathcal D_{\psi}\), then we clearly have
\begin{equation}
 \mathcal D_{\psi}(\varphi_s)\equiv m,
 \qquad 0\leq s\leq1.
 \label{eq:BM-Ding-constant}
\end{equation}

The Monge--Amp\`ere energy $E_\theta$ is affine along the weak geodesic by
\cite[Theorem~3.12]{DDL18b}, which implies $s\mapsto \mathcal L_{\mu_\psi}(\varphi_s)$ is also affine. The endpoint
equations give
\[
 \int_Xe^{-\varphi_i}\,d\mu_\psi=\mathsf V,
 \qquad i=0,1,
\]
and by definition we get
\[
 \mathcal L_{\mu_\psi}(\varphi_0)
 =\mathcal L_{\mu_\psi}(\varphi_1)=0.
\]
Since $1$-Ding functional take values \(m\) at both endpoints, we have
\[
 E_\theta(\varphi_0)
 =E_\theta(\varphi_1)=-m.
\]
Affinity therefore yields
\[
 E_\theta(\varphi_s)=-m,
 \qquad
 \mathcal L_{\mu_\psi}(\varphi_s)=0,
 \qquad 0\leq s\leq1.
\]
Equivalently,
\begin{equation}
 \int_Xe^{-\varphi_s-\psi}\Omega=\mathsf V,
 \qquad 0\leq s\leq1.
 \label{eq:BM-geodesic-mass}
\end{equation}
This proves \eqref{eq:BM-flat-identities}.

Each $\varphi_s$ is a global minimizer of the $1$-Ding functional
$\mathcal D_{\mu_\psi}$, hence $\varphi_s$ is a solution of
\eqref{eq:BM-normalized-equation} by \cite[Proposition~5.2]{DZ24}.
The proof is complete.
\end{proof}

The following
elementary observation is useful.
\begin{lemma}
\label{lem:BM-positive-splitting}
Let $S_1,S_2$ be positive $(1,1)$-currents and let $\xi$ be a smooth
$(1,0)$-vector field.  If $\iota_\xi(S_1+S_2)=0$, then
$\iota_\xi S_1=\iota_\xi S_2=0$.
\end{lemma}

\begin{proof}
The assertion is local.  On a coordinate ball \(U\subset X\), write
\[
 S_j=i\sum_{k,\ell=1}^n
       \mu^{(j)}_{k\bar\ell}\,
       dz^k\wedge d\bar z^\ell ,
\]
where \(\bigl(\mu^{(j)}_{k\bar\ell}\bigr)\) is a Hermitian matrix of
complex Radon measures.  Positivity of \(S_j\) means that, for every
smooth \((1,0)\)-vector field
\(V=\sum_k V^k\partial/\partial z^k\), the measure
\[
 \mu_j[V,V]
 :=\sum_{k,\ell}V^k\overline{V^\ell}\,
       \mu^{(j)}_{k\bar\ell}
\]
is nonnegative.

Write
\[
 \xi=\sum_k\xi^k\frac{\partial}{\partial z^k}.
\]
Contracting the hypothesis once more with \(\bar\xi\) gives
\[
 0=-i\,\iota_{\bar\xi}\iota_\xi(S_1+S_2)
   =\mu_1[\xi,\xi]+\mu_2[\xi,\xi].
\]
Since both measures on the
right-hand side are nonnegative, it follows that
\[
 \mu_1[\xi,\xi]=\mu_2[\xi,\xi]=0.
\]

Fix \(j\in\{1,2\}\) and a smooth \((1,0)\)-vector field \(V\) on \(U\).
Set
\[
 \nu_j[\xi,V]
 :=\sum_{k,\ell}\xi^k\overline{V^\ell}\,
       \mu^{(j)}_{k\bar\ell}.
\]
For every \(\chi\in C_c^\infty(U)\) with \(\chi\geq0\) and every
\(t\in\mathbb C\), positivity gives
\begin{align*}
 0
 &\leq \bigl\langle\mu_j[\xi+tV,\xi+tV],\chi\bigr\rangle\\
 &=\langle\mu_j[\xi,\xi],\chi \rangle+2\operatorname{Re}\!\left(
      \bar t\,\bigl\langle\nu_j[\xi,V],\chi\bigr\rangle
    \right)
   +|t|^2\bigl\langle\mu_j[V],\chi\bigr\rangle\\
 &=2\operatorname{Re}\!\left(
      \bar t\,\bigl\langle\nu_j[\xi,V],\chi\bigr\rangle
    \right)
   +|t|^2\bigl\langle\mu_j[V],\chi\bigr\rangle,
\end{align*}
Since this inequality holds for every \(t\in\mathbb C\), its linear
term must vanish.  Hence
\[
 \bigl\langle\nu_j[\xi,V],\chi\bigr\rangle=0.
\]
Therefore
\(\nu_j[\xi,V]=0\) as a complex Radon measure.

Taking
\(V=\partial/\partial z^\ell\) shows that
\[
 \sum_k\xi^k\mu^{(j)}_{k\bar\ell}=0
 \qquad (\ell=1,\ldots,n).
\]
These are precisely the coefficients of \(\iota_\xi S_j\).
Thus \(\iota_\xi S_j=0\).
\end{proof}

\vspace{0.4cm}
\subsection{Proof of Theorem~\ref{thm:main-big-BM} and solution to Problem~\ref{ques:DR-strictly convex}}
\label{sec:big-BM}

Our arguments follow Berndtsson's variational proof of the
Bando--Mabuchi theorem and its twisted extension
\cite[Sections~5--7]{Ber15b}, with \cite[Theorem B]{Ber15b} replaced
by Theorem~\ref{thm:main}.

\begin{proof}[Proof of Theorem~\ref{thm:main-big-BM}]
Let $(\varphi_s)_{0\leq s\leq1}$ be the finite-energy weak geodesic
joining the two solutions.  On the open strip
$S=\{0<\operatorname{Re}\tau<1\}$ set
\begin{align*}
 \mathcal T&:=p_X^*\theta+\ddc_{x,\tau}\varphi,
 &q(x,\tau)&:=\varphi_{\operatorname{Re}\tau}(x)+\psi(x).
\end{align*}
Then $\mathcal T\geq0$, and
\begin{equation}
 \mathcal R:=p_X^*\operatorname{Ric}(\Omega)+\ddc_{x,\tau}q
 =\mathcal T+p_X^*\eta_{\psi}\geq0.
 \label{eq:BM-total-current-splitting}
\end{equation}
Lemma~\ref{lem:BM-Ding-flat} shows that every $\varphi_s$ is a
normalized solution with
\begin{equation}
 -\log\int_Xe^{-q_s}\Omega=-\log\mathsf V, \qquad\forall s\in(0,1)
 \label{eq:BM-anticanonical-mass}
\end{equation} 
Thus Theorem~\ref{thm:main}, applied to the
$\operatorname{Ric}(\Omega)$-psh subgeodesic $q$, gives
$\mathcal V\in H^0(X,T_X^{1,0})$ such that, with
\[
 \xi:=\frac{\partial}{\partial\tau}+\frac12\mathcal V,
\]
one has $\iota_\xi\mathcal R=0$ on $X\times S$.  Positive-current splitting in
\eqref{eq:BM-total-current-splitting} yields
\begin{equation}
 \iota_\xi\mathcal T=0,
 \qquad \iota_{\mathcal V}\eta_\psi=0
 \label{eq:BM-split-nullities}
\end{equation}
as currents on $X\times S$ and $X$ respectively.

Let $\Lambda:\mathbb C\to\operatorname{Aut}^0(X)$ be the holomorphic flow generated
by $\mathcal V/2$. 
If $F_s$, $G_t$ are the real flows generated by $\operatorname{Re}\mathcal V$, $-\operatorname{Im}\mathcal V$, respectively, then $\Lambda_{s+it}=G_t\circ F_s=F_s\circ G_t$.
Since $\eta$ is real and closed, the second identity
in \eqref{eq:BM-split-nullities} implies $\iota_{\operatorname{Re}\mathcal{V}}\eta_\psi=0$, $\iota_{\operatorname{Im}\mathcal{V}}\eta_\psi=0$. Cartan's formula for currents \eqref{eq:Cartan} shows that $\mathcal{L}_{\operatorname{Re}\mathcal{V}}(\eta_\psi)=d(\iota_{\operatorname{Re}\mathcal{V}}\eta_\psi)+\iota_{\operatorname{Re}\mathcal{V}}(d\eta_\psi)=0$ and similarly $\mathcal{L}_{\operatorname{Im}\mathcal{V}}\eta_\psi=0$.
Therefore by Lemma~\ref{lem:Lie-derivative-and-flow}, $\eta_\psi$ is invariant under $G_t,F_s$. As a consequence,
\begin{equation}
 \Lambda_{z}^*(\eta_\psi)=\eta_\psi\qquad \forall z\in\mathbb C.
 \label{eq:BM-full-twist-invariance}
\end{equation}

\vspace{0.3cm}
The first identity in \eqref{eq:BM-split-nullities} and its conjugate
annihilate $\mathcal T$ along the real vector fields
\[
 \xi+\bar\xi=\frac{\partial}{\partial s}
                 +\operatorname{Re}\mathcal V,
 \qquad
 i(\xi-\bar\xi)=\frac{\partial}{\partial t}
                 -\operatorname{Im}\mathcal V.
\]
The flow generated by $\xi+\bar\xi$ is
\[
\Phi_r:(x,\tau)\longmapsto(\Lambda_r(x),\tau+r), \qquad r\in\mathbb{R},
\]
The flow generated by $i(\xi-\bar\xi)$ is
\[ 
\Psi_s:(x,\tau)\longmapsto(\Lambda_{is}(x),\tau+is), \qquad s\in\mathbb{R}.
\]
Because $\mathcal T$ is real and closed, the first identity
in \eqref{eq:BM-split-nullities} and Cartan's formula \eqref{eq:Cartan} show that $\mathcal{L}_{\xi+\bar{\xi}}\mathcal{T}=0$, $\mathcal{L}_{i(\xi-\bar{\xi})}\mathcal{T}=0$ on $X\times S$. The first equality implies that, for \(r\in\mathbb R\) we set
\[
 I_r:=\{s\in(0,1):s+r\in(0,1)\},
 \qquad
 S_r:=\{\tau\in S:\operatorname{Re}\tau\in I_r\}.
\]
then
\begin{equation}
 \Phi_r^*(\mathcal T)=\mathcal T
 \qquad\text{on }X\times S_r.
 \label{eq:BM-product-flow-invariance}
\end{equation}
Take \(r=s-s_0\). Note that both sides are closed positive currents of bidegree $(1,1)$ in a neighborhood of $X\times\{s_0\}$. So we can restrict both sides (indeed, restrict potentials of both sides) on $X\times \{s_0\}$ and obtain
\begin{equation}
 \Lambda_{s-s_0}^*(T_s)=T_{s_0}
 \qquad (s,s_0\in(0,1)).
 \label{eq:BM-current-transport}
\end{equation}

\vspace{0.3cm}
A completely similar argument applied to the imaginary product flow
\[
 \Psi_s:(x,\tau)\longmapsto(\Lambda_{is}(x),\tau+is)
\]
gives $\Psi_s^*(\mathcal{T})=\mathcal{T}$ on $X\times S$. For any $r\in(0,1)$, restricting both sides on $X\times \{r\}$ we obtain:
\begin{equation}
 \Lambda_{is}^*(T_r)=T_r.
 \label{eq:BM-imaginary-stabilizer}
\end{equation}
We extend \eqref{eq:BM-current-transport} and
\eqref{eq:BM-imaginary-stabilizer} to \(s,s_0\in[0,1]\) by
Lemma~\ref{lem:BM-Ding-flat}.
If $s_0=0$, then $\Lambda_{s}^*(T_{s+\varepsilon})=T_{\varepsilon}$ for $\varepsilon>0$ small enough.
Lemma~\ref{lem:BM-Ding-flat} implies $|\varphi_s-\varphi_{s+\varepsilon}|,|\varphi_{\varepsilon}-\varphi_0|\leq C\varepsilon$. 
Taking the weak limit of both sides, we get $\Lambda_{s}^*(T_{s})=T_0$.
The same arguments prove $\Lambda_{ir}^*(T_0)=T_0$.
Finally, 
\eqref{eq:BM-current-transport} yields that
\begin{equation}
 \Lambda_{s-s_0}^*\langle T_s^n\rangle
 =\langle \bigl(\Lambda_{s-s_0}^*(T_s)\bigl)^n\rangle
 =\langle T_{s_0}^n\rangle.
 \label{eq:BM-MA-transport}
\end{equation}
The proof is therefore complete.
\end{proof}

\vspace{0.4cm}
As an application of Theorem \ref{thm:main} and Theorem \ref{thm:main-big-BM}, we settle Problem \ref{ques:DR-strictly convex} raised
by Dervan--Reboulet \cite[Remarks~3.6 and~3.14]{DR24}.

\begin{corollary}
\label{cor:DR-uniqueness}
Let $X$ be a smooth projective manifold with $-K_X$ big and klt in
the sense of the multiplier ideal sheaf $\mathcal{I}(V_{\theta})=\mathcal{O}_X$, and let $\theta\in c_1(-K_X)$ be a smooth
representative.

\begin{enumerate}[\upshape(i)]
\item If $H^0(X,T_X^{1,0})=0$, then the $1$-Ding functional is strictly
convex along finite-energy weak geodesics. More precisely, if the $1$-Ding functional
is affine along a weak geodesic $(\varphi_s)\in\mathcal{E}^1(X,\theta)$, then
\[
 \theta+\ddc\varphi_s
 =
 \theta+\ddc\varphi_0
 \qquad\text{for all }s.
\]

\item If $\operatorname{Aut}(X)$ is finite, then there is at most one
weak K\"ahler--Einstein current in $c_1(-K_X)$.
\end{enumerate}
\end{corollary}

\begin{proof}
By $\partial\bar\partial$-lemma, we may assume there is a smooth volume form $\Omega$ so that
$\psi$ such that
\[
 \operatorname{Ric}(\Omega)=\theta.
\]
Thus the $1$-Ding functional is
$\mathcal D^1(u):=-\log(\frac{1}{\mathsf{V}}\int_{X}e^{-u}\Omega)$.

For (i), let $(\varphi_s)$ be a finite-energy weak geodesic along which the $1$-Ding functional is affine.
The Monge--Amp\`ere energy is affine along weak geodesic in $\mathcal{E}^1(X,\theta)$,
hence
\[
 s\longmapsto
 -\log\int_Xe^{-\varphi_s}\Omega
\]
is affine as well. Applying Theorem~\ref{thm:main} to $\theta$-plurisubharmonic subgeodesic
$(\varphi_s)_{s\in(0,1)}$ produces a holomorphic vector field
$\mathcal V\in H^0(X,T_X^{1,0})$ whose flow transports the currents
$\theta+\ddc\varphi_s$. Since
$H^0(X,T_X^{1,0})=0$, one has $\mathcal V=0$, and therefore
\[
 \theta+\ddc\varphi_s
 =
 \theta+\ddc\varphi_0
\]
for all $s$.

\vspace{0.3cm}
For (ii), let $T_{\varphi_0}$ and $T_{\varphi_1}$ be two
K\"ahler--Einstein currents. Theorem~\ref{thm:main-big-BM} with $\eta=0$, gives a holomorphic one-parameter subgroup
\[
 \Lambda:\mathbb C\longrightarrow\operatorname{Aut}^0(X)
\]
such that
\[
 \Lambda(1)^*T_{\varphi_1}=T_{\varphi_0}.
\]
If $\operatorname{Aut}(X)$ is finite, then
$\operatorname{Aut}^0(X)=\{\operatorname{id}_X\}$, so $\Lambda$ is
trivial and $T_{\varphi_0}=T_{\varphi_1}$.
\end{proof}

Combining the preceding corollary with
\cite[Theorem~1.1(ii)]{DR24}, we obtain in particular that if
$\operatorname{Aut}(X)$ is a finite discrete group, then the existence of a
K\"ahler--Einstein current in $c_1(-K_X)$ implies that
$(X,-K_X)$ is uniformly $1$-Ding stable.

\vspace{0.3cm}
\section{Application II of Theorem~\ref{thm:main}: proof of Theorem~\ref{thm:main-DZ-uniqueness}}
\label{sec:DZ-properness-uniqueness}

Throughout this section we use notations in
Section~\ref{sec:twisted-data}:
\[
 \eta=\operatorname{Ric}(\Omega)-\theta,\qquad \eta_\psi=\operatorname{Ric}(\Omega)-\theta+\ddc\psi\geq0,
 \qquad d\mu_\psi:=e^{-\psi}\Omega.
\]
In this section $\eta_\psi$ always denotes the above twisting current as in
\cite{DZ24}.  We use
$\mathcal L_{\mu_\psi}$, $\mathcal D_{\mu_\psi}$ and $d_1$ as defined in
Section~\ref{sec:twisted-data}.

We also recall the divisorial definition of the stability threshold
from Darvas--Zhang \cite[Equation~(3) and Definition~1.1]{DZ24}.
A prime divisor over $X$ is a prime divisor $E$ on a proper modification
$\pi:Y\to X$.  Its log discrepancy with respect to $\psi$ is
\[
 A_\psi(E):=A_X(E)-\nu(\psi,E),
 \qquad
 A_X(E):=1+\operatorname{ord}_E(K_Y-\pi^*K_X),
\]
where $\nu(\psi,E)$ is the generic Lelong number of $\pi^*\psi$ along
$E$.  Put
\[
 \tau_\theta(E)
 :=\sup\{t\geq0:\pi^*[\theta]-t[E]\text{ is big}\}
\]
and
\[
 S_\theta(E)
 :=\frac1{\mathsf V}\int_0^{\tau_\theta(E)}
       \operatorname{Vol}(\pi^*[\theta]-t[E])\,dt.
\]
Then
\[
 \delta_\psi([\theta])
 :=\inf_E\frac{A_\psi(E)}{S_\theta(E)},
\]
where the infimum runs over all prime divisors over $X$.

Darvas--Zhang \cite{DZ24} proved that $\delta_\psi([\theta])>1$ implies $\int_Xe^{-u-\psi}\Omega<\infty$, $\forall u\in \mathcal{E}^1(X,\theta)$ and the $1$-Ding functional $\mathcal{D}_{\mu_\psi}$ is proper on $\mathcal{E}^1(X,\theta)$.
As a result, they prove that there exists $u\in \mathcal{E}^1(X,\theta)$ of twisted Monge-Amp\`ere equation
\begin{equation}
    \langle (\theta+dd^cu)^n\rangle=e^{-u-\psi}\Omega.
    \label{eq:t CMA}
\end{equation}

\vspace{0.4cm}
The following proposition plays a key role in the uniqueness part of equation \eqref{eq:t CMA}.
\begin{proposition}
\label{prop:expanded-big-orbit}
Let $u\in\mathcal{E}^1(X,\theta)$ have minimal singularities and \(T:=\theta+\ddc u\) satisfy
\[
 \mu:=\langle T^n\rangle=e^{-u-\psi}\Omega.
\]
Suppose there is a holomorphic vector field \(V\in H^0(X,T_X^{1,0})\) and a bounded real-valued Borel function
\(h\) satisfying
\begin{equation}
 \iota_VT=i\bar\partial h,
 \qquad
 \iota_V\eta=0,
 \qquad
 \int_Xh\,d\mu=0.
 \label{eq:expanded-orbit-hypotheses}
\end{equation}
Then the real part of \(V\) generates a weak geodesic
line \((w_s)_{s\in\mathbb{R}}\) with minimal singularities such that
\begin{align}
 \langle(\theta+\ddc w_s)^n\rangle
 &=e^{-w_s-\psi}\Omega,
 \label{eq:expanded-orbit-solution}\\
 \mathcal D_{\mu_\psi}(w_s)
 &=\mathcal D_{\mu_\psi}(w_0),
 \label{eq:expanded-orbit-Ding}\\
 d_1(w_r,w_s)
 &=\sigma|s-r|,
 \label{eq:expanded-orbit-speed}
\end{align}
moreover, $ \sigma=0\Longleftrightarrow h=0$.
\end{proposition}

\begin{proof}

\smallskip
\noindent\emph{Step 1: Construction of the real orbit.}
Set
\[
 Y:=\frac{V+\bar V}{2},
 \qquad
 K:=\frac{V-\bar V}{2i}.
\]
Taking the conjugate of the first identity in
\eqref{eq:expanded-orbit-hypotheses} gives
\[
 \iota_YT=d^ch,
 \qquad
 \iota_KT=\frac12dh.
\]
Since \(T\) is closed, Cartan's formula for currents yields
\begin{equation}
 \mathcal L_YT=\ddc h,
 \qquad
 \mathcal L_KT=0.
 \label{eq:expanded-orbit-Lie-identities}
\end{equation}
Contracting \(\iota_KT=\tfrac12dh\) once more with \(K\) gives
\begin{equation}
\mathcal{L}_Kh=\iota_Kdh=0
 \label{eq:expanded-orbit-Kh-zero}
\end{equation}
in the sense of distributions.

\vspace{0.3cm}
Let $F_s$ be the flow of $Y$. Then $F:X\times\mathbb{R}\to X$, $(x,s)\mapsto F_s(x)$, is a submersion. This implies that $H:=F^*h$ is a bounded Borel function on $X\times\mathbb{R}$. 
Now we define 
\begin{equation}
 a(x,s):=\int_0^sH(x,r)\,dr,\qquad a_s(x):=a(x,s)=\int_0^s h\circ F_r(x)\,dr,
 \qquad
 u_s:=u+a_s.
 \label{eq:expanded-orbit-potential-definition}
\end{equation}
Then $a(x,s)\in L^{\infty}_{\mathrm{loc}}(X\times\mathbb{R})$, and $a_s\in L^{\infty}(X)$ for every $s$. 
It is also easy to see that
\begin{equation}
    \|a_s-a_t\|_{L^{\infty}(X)}\leq |s-t|\|h\|_{L^{\infty}(X)}.
    \label{eq:Lip-cont}
\end{equation}
Moreover we have
\begin{equation}
\frac{\partial}{\partial s}a(x,s):=\lim_{\varepsilon\to 0}\frac{a(x,s+\varepsilon)-a(x,s)}{\varepsilon}=H(x,s)\ \textrm{for a.e.}\  (x,s)\in X\times \mathbb{R}.
\label{eq:ab-cont}
\end{equation}
In fact, for every $x\in X$, $s\mapsto H(x,s)$ is a bounded Borel function. As a result, $s\mapsto a(x,s)=\int_0^sH(x,r)\,dr$ is absolutely continuous (in fact Lipschitz) and $\frac{d}{ds}a(x,s)=H(x,s)$ for a.e. $s\in\mathbb{R}$. \eqref{eq:ab-cont} is therefore proved.

We will also need the following:
\begin{equation}
    \partial_su=H(x,s)
    \label{eq:ac-derivative}
\end{equation} 
in the sense of distributions on $X\times\mathbb{R}$.
In fact, take test form $\chi(x,s)ds\Omega$ and compute that
\begin{align*}
    \langle \partial_su,\chi ds\Omega\rangle
    &:=-\int_{\mathbb{R}\times X}u\frac{\partial\chi}{\partial s}\,ds\,\Omega\\
    &=-\int_X\left(\int_{\mathbb{R}}u_s(x)\frac{\partial\chi}{\partial s}(x,s)\,ds\right)\Omega(x)\\
    &=\int_X\left(\int_{\mathbb{R}}H(x,s)\chi\,ds\right)\Omega\\
    &=\left\langle H,\chi ds\Omega\right\rangle.
\end{align*}
Here the third equality holds because, for $x\in \operatorname{Amp}(\theta)$, $s\mapsto u_s(x)=u(x)+a(x,s)$ is absolutely continuous and its derivative is $H(x,s)$.

\vspace{0.3cm}
Finally we prove that
\begin{align*}
 F_s^*T-T
 =\ddc a_s
\end{align*}
in the sense of currents.
In fact, take a test form $\beta$ on $X$. Using \eqref{eq:current-Cartan-flow}, we get
\begin{align*}
    \langle F_s^*T-T,\beta \rangle
    &=\int_0^s \frac{d}{dt}\langle F^*_tT,\beta \rangle\,dt=\int_0^s \langle F_t^*(\mathcal{L}_YT),\beta\rangle\,dt\\
    &=\int_0^s \langle F_t^*(dd^ch),\beta\rangle\,dt=\int_0^s \langle dd^c(F_t^*h),\beta\rangle\,dt\\
    &=\int_0^s \langle F_t^*h,dd^c\beta\rangle\,dt=\left\langle \int_0^sF_t^*h\,dt,dd^c\beta\right\rangle\\ 
  &  =\langle a_s,dd^c\beta\rangle
    =\langle dd^ca_s,\beta \rangle,
\end{align*}
Thus \(u_s:=u+a_s\), initially understood as a Borel function on $X$, has a
canonical \(\theta\)-psh representative $w_s$ and satisfies
\begin{equation}
 T_s:=\theta+\ddc w_s=F_s^*T
 \label{eq:expanded-orbit-derivative}
\end{equation}
in the sense of currents. 
Furthermore, it follows easily from \eqref{eq:Lip-cont} and the property of quasi-psh functions that
\begin{equation}
    \|w_s-w_t\|_{L^{\infty}(X)}\leq |s-t|\|h\|_{L^{\infty}(X)}.
    \label{eq:Lip-ws}
\end{equation}

\vspace{0.3cm}
\smallskip
\noindent\emph{Step 2: Gluing $w_s$ to a subgeodesic.}
Let \(G_t\) be the real flow of \(-K:=-\frac{V-\bar{V}}{2i}\).  Since \(X\) is compact and
\(V\) is holomorphic, the vector fields \(Y\) and \(-K\) are complete and
commute.  The holomorphic flow generated by \(V/2\) is
\begin{equation}
 \mathbb{C}\to \operatorname{Aut}^0(X),s+it\mapsto\mathfrak F_{s+it}:=F_s\circ G_t=G_t\circ F_s.
 \label{eq:expanded-complex-flow-definition}
\end{equation}
In particular, \(\mathfrak F_s=F_s\) is the real flow of $Y$, while
\(\mathfrak F_{it}=G_t\) is the real flow of \(-K\).  
Since the Lie derivatives of $T$ and $h$ with respect to $K$ vanish by \eqref{eq:expanded-orbit-Lie-identities} and
\eqref{eq:expanded-orbit-Kh-zero}, Lemma~\ref{lem:Lie-derivative-and-flow}
implies that, in the sense of currents,
\begin{equation}
 G_t^*T=T,
 \qquad
 h\circ G_t=h.
 \label{eq:expanded-imaginary-flow-invariance}
\end{equation}
Set
\[
 \mathcal F:X\times \mathbb{C}\longrightarrow X,
 \qquad
 \mathcal F(x,\tau):=\mathfrak F_\tau(x),
 \qquad
 U(x,s+it):=u_s(x).
\]
Since \(\mathcal F\) is a holomorphic submersion, \(\mathcal F^*T\) is a well-defined positive closed $(1,1)$-current. We claim that, in the sense of currents,
\begin{equation}
 p_X^*\theta+\ddc_{x,\tau}U=\mathcal F^*T.
 \label{eq:expanded-product-current-identity}
\end{equation}

\vspace{0.3cm}
The problem is local.
A test form $\alpha $ has the form $\alpha_0\wedge d\tau\wedge d\bar{\tau}+\alpha_1+\alpha_2$ where $\alpha_0$ only has $dx$ and $d\bar{x}$, $\alpha_1$ does not contain $d\tau$ and $\alpha_2$ does not contain $d\bar{\tau}$.
By Fubini's theorem,
\begin{align*}
    \langle p_X^*\theta+dd^c_{x,\tau}U,\alpha_0\wedge d\tau\wedge d\bar{\tau}\rangle
    &=\int_{\mathbb{C}}d\tau\wedge d\bar{\tau}\int_X\bigl(\theta+dd^c_xU(\cdot,\tau)\bigr)\wedge (\alpha_0|_{X_\tau}),\\
   \langle \mathcal{F}^*T,\alpha_0\wedge d\tau\wedge d\bar{\tau}\rangle
   &=\int_{X\times\mathbb{C}}\bigl(\mathcal{F}^*\theta+dd^c(u\circ\mathcal F)\bigr)\wedge\alpha_0\wedge d\tau\wedge d\bar{\tau}\\
   &=\int_{\mathbb{C}}d\tau\wedge d\bar{\tau}\int_X\mathfrak{F}_{\tau}^*(T)\wedge(\alpha_0|_{X_{\tau}}).
\end{align*}
For each $s$, $\theta+dd^c_xU(\cdot,s)=T_s=F_s^*T$ by  \eqref{eq:expanded-orbit-derivative}. $\mathfrak{F}_{\tau}^*T=(G_t\circ F_s)^*T=F_s^*(G_t^*T)=F_s^*T$ by \eqref{eq:expanded-complex-flow-definition} and
\eqref{eq:expanded-imaginary-flow-invariance}.
It remains to deal with $\alpha_1,\alpha_2$, which is equivalent to showing that
\[\iota_{\partial_{\tau}}(p_X^*\theta+dd^cU)=\iota_{\partial_{\tau}}(\mathcal{F}^*T),\qquad \iota_{\partial_{\bar \tau}}(p_X^*\theta+dd^cU)=\iota_{\partial_{\bar \tau}}(\mathcal{F}^*T).
\]
Since both currents are real, we only need to check the first identity.

The holomorphic flow $\mathcal{F}$ is generated by $\frac{V}{2}$, so we have
\[
 d\mathcal F_{(x,\tau)}\bigl(\frac{\partial}{\partial\tau}\bigr)=\frac12V_{\mathcal{F}(x,\tau)}.
\]
The second identity in \eqref{eq:expanded-imaginary-flow-invariance} means $G_t^*h=h$ a.e. on $X$ for every $t$. 
So for every $\tau\in\mathbb{C}$,
\[
 \mathcal{F}_{\tau}^*h=F_s^*G_t^*h=F_s^*h  \textrm{ a.e. on  $X$}.
\]
Naturality of contraction gives $\iota_{\partial_{\tau}}(\mathcal{F}^*\xi)=\mathcal{F}^*(\iota_{\frac{V}{2}}\xi)$ for every smooth differential form $\xi$ on $X$. This equality extends to currents on $X$ by smooth approximation. As a result, using \eqref{eq:expanded-orbit-hypotheses}, we obtain
\begin{align}
 \iota_{\partial_{\tau}}(\mathcal F^*T)
 &=\mathcal F^*\!\left(\iota_{V/2}T\right) \notag=\frac i2\,\mathcal F^*(\bar\partial_Xh) \notag\\
 &=\frac i2\,\bar\partial_{X\times \mathbb{C}}
       (\mathcal F^*h) \notag=\frac i2\,\bar\partial_{X\times \mathbb{C}}
       (h\circ F_s).
 \label{eq:expanded-pullback-contraction}
\end{align}
On the right-hand side, \(\iota_{\partial_{\tau}}\bigl(p_X^*\theta\bigr)=0\), and using \eqref{eq:expanded-orbit-hypotheses} we obtain
\begin{align}
\iota_{\partial_{\tau}}\bigl(p_X^*\theta+\ddc_{x,\tau}U\bigr)
 &=i\bar\partial_{X\times \mathbb{C}}(\partial_{\tau}U) \\
 &=i\bar\partial_{X\times \mathbb{C}}\left(
   \frac12(\partial_s-i\partial_t)U\right) \\
 &=\frac i2\bar\partial_{X\times \mathbb{C}} (\partial_sU)\\
 &=\frac i2\bar\partial_{X\times \mathbb{C}}(h\circ F_s).
 \label{eq:expanded-potential-contraction}
\end{align}
We define
\begin{equation}
    \tilde{U}(x,\tau):=w_s(x),\qquad (x,\tau)\in X\times\mathbb{C},
\end{equation}
where $w_s$ is constructed in \emph{Step 1.}
We claim that $\tilde{U}$ is a weak subgeodesic line with minimal singularities.

Indeed, since each $w_s$ is upper semi-continuous and $w_s$ is uniformly Lipschitz in $s$, it is not hard to see from the definition of $\tilde{U}$ that $\tilde{U}$ is upper semi-continuous on $X\times\mathbb{C}$. Since $w_s=U(\cdot,s)$ a.e. on $X$, We have $\tilde{U}(\cdot,\tau)=U(\cdot,\tau)$ a.e. for every $\tau\in\mathbb{C}$. 
Fubini's theorem yields that $\tilde{U}=U$ a.e. on $X\times\mathbb{C}$ and it follows immediately that $p_X^*\theta+dd^c_{x,\tau}\tilde{U}=p_X^*\theta+dd^c_{x,\tau}U\geq0$. 
It remains to show that $\tilde U$ is indeed a $p_X^*\theta$-psh function. 

Recall that if $\tilde{u}$ is a locally integrable function on some open $\Omega\subset \mathbb{C}^n$ and $\tilde{u}$ coincides with some $u\in\mbox{PSH}(\Omega)$ almost everywhere on $\Omega$, then 
\begin{align*}
    u(x)=\lim_{r\to0^+}\bigl(\operatorname*{ess\,sup}_{B_r(x)} \ \tilde{u}\bigl),\qquad \forall x\in \Omega.
\end{align*}

We fix arbitrary \((x_0,\tau_0)\in X\times\mathbb C\) with
\(s_0=\operatorname{Re}\tau_0\).  On one hand, the uniform Lipschitz estimate \eqref{eq:Lip-ws} gives
\[
 \left|
 \operatorname*{ess\,sup}_{B_r(x_0)\times D_r(\tau_0)}\tilde U
 -
 \operatorname*{ess\,sup}_{B_r(x_0)}w_{s_0}
 \right|
 \leq Cr.
\]
In fact, $w_{s_0}\leq \operatorname*{ess\,sup}_{B_r(x_0)}w_{s_0}$ on $B_r(x_0)$ and Lipschitz continuity imply that $\tilde{U}\leq \operatorname*{ess\,sup}_{B_r(x_0)}w_{s_0}+Cr$ on $B_r(x_0)\times D_r(\tau_0)$. This gives one-side control $\operatorname*{ess\,sup}_{B_r(x_0)\times D_r(\tau_0)}\tilde U\leq\operatorname*{ess\,sup}_{B_r(x_0)}w_{s_0}+Cr$.
On the other hand, $E_{\varepsilon}:=\{x\in B_r(x_0):w_r(x)\geq \operatorname*{ess\,sup}_{B_r(x_0)}w_{s_0}-\varepsilon\}$ has positive Lebesgue measure for any $\varepsilon>0$, so $\{(x,\tau)\in B_r(x_0)\times D_r(\tau_0):\tilde{U}(x,\tau)\geq \operatorname*{ess\,sup}_{B_r(x_0)}w_{s_0}-\varepsilon-Cr\}$ contains $E_{\varepsilon}\times D_r(x_0)$ hence has positive measure. This proves the reverse inequality.

\(w_{s_0}\) is \(\theta\)-psh, so we have
\[
 w_{s_0}(x_0)
 =
 \lim_{r\to0^+}
 \operatorname*{ess\,sup}_{B_r(x_0)}w_{s_0}.
\]
Consequently,
\[
 \tilde U(x_0,\tau_0)
 =
 w_{s_0}(x_0)
 =
 \lim_{r\to0^+}
 \operatorname*{ess\,sup}_{B_r(x_0)\times D_r(\tau_0)}\tilde U.
\]
Since $(x_0,\tau_0)$ is arbitrary, we get to know $\tilde{U}$ is $p_X^*\theta$-psh on $X\times\mathbb{C}$.

\vspace{0.3cm}
\smallskip
\noindent\emph{Step 3: Verification of the weak geodesic equation.}
Identity \eqref{eq:expanded-product-current-identity} shows that
\(\tilde{U}\) is a subgeodesic.  For fixed \(\tau\), the path
\[
 [0,1]\ni r\longmapsto\mathfrak F_{r\tau}
\]
is an isotopy from \(\operatorname{id}_X\) to
\(\mathfrak F_\tau\), so
\[
 \mathfrak F_\tau^*[\theta]=[\theta]\in H^{1,1}(X,\mathbb{R}).
\]
Therefore each \(\mathfrak F_\tau\) preserves \(\operatorname{Amp}([\theta])\): $\mathfrak{F}_{\tau}^{-1}(\operatorname{Amp}([\theta]))=\operatorname{Amp}([\theta])$.

On \(\operatorname{Amp}([\theta])\times S_{a,b}\), both \(p_X^*\theta+dd^c\tilde{U}\) and \(\mathcal{F}^*T=\mathcal{F}^*(\theta+dd^cu)\) have locally
bounded potentials, so their Bedford--Taylor products \cite{BT82} are
well-defined. Therefore we obtain that
\begin{equation}
 \bigl(p_X^*\theta+\ddc_{x,\tau}\tilde{U}\bigr)^{n+1}
 =
 \bigl(\mathcal F^*T\bigr)^{n+1}
 =
 0
 \quad\text{on }\operatorname{Amp}([\theta])\times S_{a,b},
 \label{eq:expanded-product-rank}
\end{equation}
because \(\dim_{\mathbb C}X=n\).  

On every strip $S_{a,b}$, the slices have minimal singularities,
\eqref{eq:Lip-cont} gives uniform time-Lipschitz
control, and \(\tilde U\) has the prescribed endpoint values.  
\cite[Proposition~3.3]{DDL18b} shows
that every finite restriction of \(s\mapsto w_s\) is the weak
geodesic between its endpoints.  Hence \(w_s\) is a weak geodesic line.

\smallskip
\noindent\emph{Step 4: Transport of the Monge--Amp\`ere equation.}
Since \(\eta\) is real and closed,
\(\iota_V\eta=0\) and its conjugate give
\(\iota_Y\eta=0\).  Hence
\(\mathcal L_Y\eta=0\), which implies
\begin{equation}
 F_s^*\eta=\eta.
 \label{eq:expanded-twist-flow-invariance}
\end{equation}
Set
\[
 \mu_s:=F_s^*\mu=\langle (F_s^*T)^n\rangle=\langle T_s^n\rangle,
 \qquad
 \nu_s:=e^{-w_s-\psi}\Omega,
\]
where the first equality uses the naturality of non-pluripolar
products under biholomorphisms.  Since
\[
 \operatorname{Ric}(\mu)=T+\eta,
\]
\eqref{eq:expanded-twist-flow-invariance} gives
\[
 \operatorname{Ric}(\mu_s)=F_s^*(\operatorname{Ric}(\mu))=T_s+\eta=\operatorname{Ric}(\nu_s).
\]
Let $q_s:=\log\frac{d\mu_s}{d\nu_s}\in L^1(X)$. Then
\begin{align*}
    -dd^cq_s=\operatorname{Ric}(\mu_s)-\operatorname{Ric}(\nu_s)=0.
\end{align*}
This proves that $q_s$ is constant on $X$.
Hence
\begin{equation}
 \mu_s=c_s\nu_s
 \qquad\text{for some }c_s>0.
 \label{eq:expanded-orbit-proportional-measures}
\end{equation}
Since \(w_s=u+a_s\) a.e. on $X$, we have
\[
 \nu_s=e^{-a_s}\mu \qquad \text{a.e. on }X.
\]
The bound \eqref{eq:Lip-cont} implies that
\(I(s):=\nu_s(X)=\int_X e^{-a_s}\mu\) is locally absolutely continuous. Recall that \eqref{eq:ab-cont} shows $\frac{\partial}{ \partial s }a(x,s)=H(x,s)=h\circ F_s(x)$ almost everywhere on $X\times\mathbb{R}$ and $\|a_s-a_t\|_{L^{\infty}(X)}\leq |s-t|\|h\|_{L^{\infty}}$.
Then Fubini's
theorem and the dominated convergence theorem yield that, for almost every \(s\),
\begin{align*}
 I'(s)
 &=-\int_X(h\circ F_s)e^{-a_s}\mu\\
 &=-\int_X(h\circ F_s)\,d\nu_s\\
 &=-c_s^{-1}\int_X(h\circ F_s)\,d\mu_s\\
 &=-c_s^{-1}\int_Xh\,d\mu
 =0.
\end{align*}
Thus \(I(s)=I(0)=\mathsf V\).  Since
\(\mu_s(X)=\mu(X)=\mathsf V\), equation
\eqref{eq:expanded-orbit-proportional-measures} gives \(c_s=1\).
Together with \(\mu_s=\langle T_s^n\rangle\), this proves
\eqref{eq:expanded-orbit-solution}.

\smallskip
\noindent\emph{Step 5: Constancy of $1$-Ding functional and the \(d_1\)-speed.}
Since each $w_s$ is a normalized solution, \cite[Proposition~5.7]{DZ24} immediately gives \eqref{eq:expanded-orbit-Ding}. Finally,
\cite[Proposition~3.13]{DDL18a} applied on finite intervals gives a number \(\sigma\geq0\) such
that
\[
 d_1(w_r,w_s)=\sigma|s-r|.
\]
If \(h=0\), it follows from the definition \eqref{eq:expanded-orbit-potential-definition}
that \(w_s=u\), hence \(\sigma=0\).  Conversely, if \(\sigma=0\),
then \(w_s=u\) for every \(s\).
We can prove that
\[
 \frac{w_s-u}{s}
 =
 \frac{a_s}{s}
 =\int_0^1h\circ F_{ts}\,dt
 \longrightarrow h
 \quad\text{in }L^1(X,\Omega)\quad \text{as $s\to0$.}
\]
In fact, let $h_j\in C(X)$ be a uniformly bounded sequence with $h_j\to h$ in $L^1(X)$. On the one hand, $\int_0^1h_j\circ F_{ts}\,dt\to h_j$ uniformly on $X$ as $s\to0$. On the other hand,
\begin{align*}
    \int_X\Big(\int_0^1\bigl|(h_j-h)\circ F_{ts}\bigl|dt\Big)\Omega
    &=\int_0^1\Big(\int_X|(h_j-h)\circ F_{ts}|\Omega\Big)dt\\
    &=\int_0^1\Big(\int_X|h_j-h|(F_{-ts})^*\Omega\Big)dt\\
    &\leq C\int_0^1 \Big(\int_X|h_j-h|\Omega\Big) dt\to0.
\end{align*}
Thus \(h=0\), proving
\eqref{eq:expanded-orbit-speed}.
\end{proof}

\vspace{0.3cm}
\begin{lemma}
\label{lem:DZ-ae-compatible-velocity}
Let \(\varphi_0,\varphi_1\) be normalized minimal-singularity
solutions, let \((\varphi_s)_{0\leq s\leq1}\) be the weak geodesic
joining them, and put
\[
 T_s:=\theta+\ddc\varphi_s,
 \qquad
 \mu_s:=\langle T_s^n\rangle
       =e^{-\varphi_s-\psi}\Omega.
\]
Let \(\mathcal V\) be the holomorphic vector field provided by
Theorem~\ref{thm:main-big-BM}, and set \(V:=-\mathcal V\).
There exists a full-measure set \(I_0\subset(0,1)\) such that, for
every \(s\in I_0\), \(h_s:=\dot\varphi_s\) belongs to
\(L^\infty(X,\mathbb R)\) and
\[
 \iota_VT_s=i\bar\partial h_s,
 \qquad
 \iota_V\eta=0,
 \qquad
 \int_Xh_s\,d\mu_s=0.
\]
\end{lemma}

\begin{proof}
By \eqref{eq:BM-time-Lipschitz}, $|\varphi_s-\varphi_t|\leq C|s-t|$, $\forall s,t\in[0,1]$.
The partial derivative
\(H:=\partial_s\varphi\) is well-defined almost everywhere on $X\times(0,1)$ and 
\(\|H\|_{L^{\infty}(X\times(0,1))}\leq C\). Fubini's theorem yields that for almost every \(s\),
\(h_s:=H(\,\cdot\,,s)=\dot\varphi_s\in L^{\infty}(X)\).

Recall from \eqref{eq:BM-split-nullities} that, with
\(\mathcal T:=p_X^*\theta+\ddc_{x,\tau}\varphi\),
\[
 \iota_{\partial/\partial\tau+\mathcal V/2}\mathcal T=0,
 \qquad
 \iota_{\mathcal V}\eta=0.
\]
Since \(\varphi\) is independent of \(\operatorname{Im}\tau\), an easy calculation yields that the
fiberwise \((0,1)\)-component of the first identity is
\[
 \frac i2\bar\partial_XH
 +\frac12\iota_{\mathcal V}T_s=0
\]
in the sense of currents on \(X\) for almost all $s\in(0,1)$. More precisely, for
smooth test form \(\beta\) of on \(X\), \(\chi\in C_c^\infty(0,1)\) and $\eta\in C_c^\infty(\mathbb{R}) $ with $\int_{\mathbb{R}}\eta(t)dt=1$, it holds that
\begin{align*}
    0=\left\langle\iota_{2\frac{\partial}{\partial\tau}+\mathcal V}\mathcal{T},\chi\eta id\tau\wedge d\bar\tau\wedge\beta\right\rangle
 &=\int_{\mathbb{R}_t}\eta(t)dt\int_0^1\chi(s)\langle
   \iota_VT_s-i\bar\partial_X h_s,\beta
 \rangle ds\\
 &=\int_0^1\chi(s)
 \langle
   \iota_VT_s-i\bar\partial_X h_s,\beta
 \rangle ds.
\end{align*}
Since $\chi$ is arbitrary, $\langle \iota_VT_s-i\bar\partial_Xh_s,\beta\rangle=0$ for almost every $s\in(0,1)$.
By choosing a countable dense
family of test forms \(\beta\), we get
\(\iota_VT_s=i\bar\partial_X h_s\) in the sense of currents on \(X\) for almost every
\(s\). The identity \(\iota_V\eta=0\) follows from
\(\iota_{\mathcal V}\eta=0\).

Finally, the Lipschitz continuity of $\varphi_s$
in $s$-variable and dominated convergence theorem implies $\frac{d}{ds}\int_Xe^{-\varphi_s-\psi}\Omega=-\int_X\dot\varphi_s e^{-\varphi_s-\psi}$ holds for a.e. $s\in(0,1)$. 
\eqref{eq:BM-geodesic-mass} gives $\int_Xe^{-\varphi_s-\psi}\Omega=\mathsf V$ for every $s\in(0,1)$. Therefore for almost every \(s\),
\[
 0
 =\frac d{ds}\int_Xe^{-\varphi_s-\psi}\Omega
 =-\int_Xh_s\,d\mu_s.
\]
The proof is complete.
\end{proof}

\vspace{0.4cm}
\begin{proof}[Proof of Theorem~\ref{thm:main-DZ-uniqueness}]
The equivalence between \emph{(i)} and \emph{(ii)} follows easily from \cite[Proposition~5.4]{DZ24}, while the implication \emph{(iii)} \(\Rightarrow\) \emph{(i)} follows from \cite[Proposition~5.9]{DZ24}. It therefore remains only to prove \emph{(i)} \(\Rightarrow\) \emph{(iii)}. To this end, assume that
$$
\delta_\psi([\theta])>1.
$$ By the equivalence between \emph{(i)} and \emph{(ii)}, the $1$-Ding functional is proper, namely there exist \(\varepsilon,C>0\)
such that
\begin{equation}
 \mathcal D_{\mu_\psi}(u)
 \geq\varepsilon\bigl(\sup_Xu-E_\theta(u)\bigr)-C,
 \qquad u\in\mathcal E^1(X,\theta).
 \label{eq:Ding_Proper}
\end{equation}

Let \(\varphi_0,\varphi_1\) be two normalized
minimal-singularity solutions, and let \((\varphi_s)\) be the weak
geodesic joining them. Fix \(s\) in the full-measure set provided by
Lemma~\ref{lem:DZ-ae-compatible-velocity}. Proposition
\ref{prop:expanded-big-orbit}, applied to
\((T_s,V,h_s)\), gives a weak geodesic line
\((w_r)_{r\in\mathbb R}\), with \(w_0=\varphi_s\), consisting of
normalized solutions and satisfying
\[
 \mathcal D_{\mu_\psi}(w_r)
 =\mathcal D_{\mu_\psi}(w_0),
 \qquad
 d_1(w_r,w_0)=\sigma_s|r|,
 \qquad
 \sigma_s=0\Longleftrightarrow h_s=0.
\]

Since \(\mathcal L_{\mu_\psi}(w_r)=0\), both
\(E_\theta(w_r)\) and
\(\mathcal D_{\mu_\psi}(w_r)\) are constant in \(r\). Thus, writing
\(e_s:=E_\theta(w_0)\) and
\(J_r:=\sup_Xw_r-E_\theta(w_r)\), the properness of $1$-Ding functional \eqref{eq:Ding_Proper} implies that \(J_r\)
is uniformly bounded from above.

Indeed, if \(\widetilde w_r:=w_r-\sup_Xw_r\), then
\(\widetilde w_r\leq V_\theta\), and the definition of $d_1$ gives
\(d_1(\widetilde w_r,V_\theta)=J_r\). Since
\(d_1(w_r,\widetilde w_r)=|\sup_Xw_r|\) and
\(\sup_Xw_r=J_r+e_s\), we obtain
\begin{align*}
 d_1(w_r,w_0)&\leq d_1(w_r,\widetilde w_r)+d_1(\widetilde w_r,V_\theta)+d_1(V_\theta,w_0)\\
 &\leq 2J_r+|e_s|+d_1(V_\theta,w_0).
\end{align*}
Hence \(d_1(w_r,w_0)\) is bounded in \(r\), so
\(\sigma_s=0\), and Proposition~\ref{prop:expanded-big-orbit} yields
\(h_s=0\).

As this holds for every \(s\) in a full-measure subset of \((0,1)\),
we have \(\partial_s\varphi=0\) almost everywhere on
\(X\times(0,1)\).
It follows that for a.e. $x\in X$, the derivative of $s\mapsto \varphi_s(x)$ is $0$ almost everywhere.
Since $s\mapsto \varphi_s(x)$ is convex, we know it is constant on $[0,1]$.
Therefore
\(\varphi_0=\varphi_1\) almost everywhere and consequently $\varphi_0=\varphi_1$. The proof is finally concluded.
\end{proof}

\section{Application III of Theorem~\ref{thm:main}: proof of Theorem~\ref{thm:main-big-Matsushima}}
\label{sec:null-automorphism-group}

Throughout this section we assume $(X,\omega)$ is a compact K\"ahler manifold, $\Omega$ is a smooth positive volume form, $[\theta]\in H^{1,1}(X,\mathbb{R})$ is big with smooth representative $\theta$.
Let $\lambda>0$, $\eta:=\operatorname{Ric}(\Omega)-\lambda\theta$ and $\psi$ be a $\eta$-psh function.

Recall that a function $\varphi\in\mathcal{E}^1(X,\theta)$ is called a normalized solution if it satisfies
\[
\langle(\theta+dd^c\varphi)^n\rangle=e^{-\lambda\varphi-\psi}\Omega
\]
and $T_\varphi:=\theta+dd^c\varphi$. $S_{\lambda}(\theta,\psi)$ denotes the set of such currents.
\[
\langle(\theta+dd^c\varphi)^n\rangle=e^{-\lambda\varphi-\psi}\Omega,\qquad \varphi\in\mathcal{E}^1(X,\theta).
\]
We assume
\[
 \mathcal S:=\mathcal S_\lambda(\theta,\psi)\neq\varnothing
\]
and fix \(T\in\mathcal S\).  Recall that
\begin{align*}
 \widehat H_\eta
 &:=\bigl\{g\in\operatorname{Aut}^0(X):g^*\eta=\eta\bigr\},
 &
 H_\eta
 &:=\widehat H_\eta^{\,\circ},\\
 \widehat K_T
 &:=\bigl\{g\in\operatorname{Aut}^0(X):g^*T=T\bigr\},
 &
 K_T
 &:=\widehat K_T\cap H_\eta.
\end{align*}
Here \(\operatorname{Aut}^0(X)\) denotes the connected component of the
identity in \(\operatorname{Aut}(X)\).  Likewise,
\(\widehat H_\eta^{\,\circ}\) denotes the connected component of the
identity in \(\widehat H_\eta\), with respect to the topology inherited
from \(\operatorname{Aut}^0(X)\).  

The null Lie algebra of the twisting current is
\[
 \mathfrak g_\eta
 :=
 \bigl\{
 V\in H^0(X,T_X^{1,0}):\iota_V\eta=0
 \bigr\}.
\]
Let \(G_\eta\) be the connected analytic subgroup of
\(\operatorname{Aut}^0(X)\) integrating \(\mathfrak g_\eta\).
More precisely, \(G_\eta\) is a connected complex Lie group equipped
with a canonical injective holomorphic group immersion
\[
 \jmath_\eta:G_\eta\longrightarrow\operatorname{Aut}^0(X)
\]
such that
\[
 (d\jmath_\eta)_e
 \bigl(\operatorname{Lie}_{\mathbb C}(G_\eta)\bigr)
 =\mathfrak g_\eta.
\]
Here $\operatorname{Lie}_{\mathbb C}(G_\eta)$ is the complex Lie algebra of $G_\eta$. Lie algebras of the real Lie subgroups $L_T,K_T,\widehat{K}_T$ are always understood over $\mathbb{R}$.  Thus, after identifying
\(\operatorname{Lie}_{\mathbb C}(G_\eta)\) with its image under
\((d\jmath_\eta)_e\), one has
\[
 \operatorname{Lie}_{\mathbb C}(G_\eta)=\mathfrak g_\eta.
\]
Equivalently, \(\jmath_\eta(G_\eta)\) is the connected subgroup generated
by the holomorphic one-parameter subgroups
\[
 z\longmapsto
 \exp_{\operatorname{Aut}^0(X)}(zV),
 \qquad V\in\mathfrak g_\eta.
\]
The group \(G_\eta\) is always endowed with its intrinsic Lie-group
topology.  In particular, unless \(\jmath_\eta(G_\eta)\) is closed in
\(\operatorname{Aut}^0(X)\), this intrinsic topology need not agree with
the topology induced from \(\operatorname{Aut}^0(X)\).

Finally, define the stabilizer of \(T\) inside \(G_\eta\) by
$
 L_T
 :=\jmath_\eta^{-1}(\widehat K_T)
 =\bigl\{a\in G_\eta:\jmath_\eta(a)^*T=T\bigr\}.
$
The lemma below identifies $\operatorname{Lie}_{\mathbb R}(L_T)$
with the null compact algebra $\mathfrak k_T^{\mathrm{null}}$ defined
before Theorem~\ref{thm:main-big-Matsushima}.
Whenever no confusion can arise, we suppress \(\jmath_\eta\) and simply
write \(a^*T\) for \(\jmath_\eta(a)^*T\).

The compactness of the stabilizer of $T$ is supplied by the
following proposition.  Its proof makes use of the compactness
result of \cite[Section~5]{BBEGZ19}.

\begin{proposition}
\label{prop:compact-current-stabilizer}
The manifold \(X\) is projective,
\[
H^j(X,\mathcal O_X)=0,
\qquad j\geq1,
\]
and \(\widehat K_T\) is compact.
\end{proposition}

\begin{proof}
Set
\[
q:=\lambda\varphi+\psi,
\qquad
\mu:=\langle T^n\rangle=e^{-q}\Omega,
\qquad
M:=\mu(X).
\]
Then
\[
R:=\operatorname{Ric}(\Omega)+\ddc q
=\lambda T+\eta\geq0,
\qquad
\{R\}=c_1(-K_X).
\]
The openness theorem \cite{Ber15a,GZ15} gives
\begin{equation}
e^{-pq}\in L^1(X,\Omega)
\quad\text{for some }p>1.
\label{eq:compactness-strong-openness-margin}
\end{equation}

If \(S_\theta\) is a K\"ahler current in \([\theta]\), then
\(\lambda S_\theta+\eta\) is a K\"ahler current in
\(c_1(-K_X)\). Hence \(-K_X\) is big, so \(X\) is Moishezon
\cite{JS93}. Since
\(X\) is compact K\"ahler, it is projective.

By Demailly's regularization
\cite[Main Theorem~1.1]{Dem92}, we may choose
\[
S=\operatorname{Ric}(\Omega)+\ddc\rho\in c_1(-K_X),
\qquad
S\geq\varepsilon\omega,
\]
where \(\rho\) has analytic singularities. In particular,
\(e^{-a\rho}\in L^1(X,\Omega)\) for some \(a>0\). It follows from
\eqref{eq:compactness-strong-openness-margin} and H\"older's
inequality that, for all sufficiently small \(\delta>0\),
\[
e^{-((1-\delta)q+\delta\rho)}\in L^1(X,\Omega).
\]
Moreover,
\[
\operatorname{Ric}(\Omega)
+\ddc\bigl((1-\delta)q+\delta\rho\bigr)
=(1-\delta)R+\delta S
\geq\delta\varepsilon\omega.
\]
Thus the corresponding singular metric on \(-K_X\) has strictly
positive curvature and trivial multiplier ideal. Nadel vanishing theorem
\cite{Nad90} (see also \cite[Theorem~5.11]{Dem96}), applied to
\(K_X\otimes(-K_X)\simeq\mathcal O_X\), gives
\begin{equation}
H^j(X,\mathcal O_X)=0,
\qquad j\geq1.
\label{eq:compactness-nadel-vanishing}
\end{equation}

Fix an ample line bundle \(A\) and a K\"ahler form
\(\omega_A\in c_1(A)\). We claim that
\begin{equation}
g^*A\simeq A
\quad\text{for every }g\in\operatorname{Aut}^0(X).
\label{eq:compactness-invariant-polarization}
\end{equation}
Indeed, since \(\operatorname{Aut}^0(X)\) is connected, \(g\) can be
joined to the identity by a path \(g_t\) in
\(\operatorname{Aut}^0(X)\). The map
\[
[0,1]\times X\longrightarrow X,
\qquad
(t,x)\longmapsto g_t(x),
\]
is a homotopy between \(g\) and the identity. Hence \(g^*\) acts as
the identity on \(H^2(X,\mathbb Z)\), and therefore
\[
c_1(g^*A)=g^*c_1(A)=c_1(A).
\]
On the other hand, the exponential exact sequence
\[
0\longrightarrow\mathbb Z
\longrightarrow\mathcal O_X
\longrightarrow\mathcal O_X^*
\longrightarrow0
\]
and \(H^1(X,\mathcal O_X)=0\) imply that a holomorphic line bundle is
determined up to isomorphism by its integral first Chern class. This
proves \eqref{eq:compactness-invariant-polarization}.

Set \(V_A:=\int_X\omega_A^n=c_1(A)^n\) and
\[
\nu:=\frac{V_A}{M}\,\mu.
\]
Then \(\nu\) has mass \(V_A\) and an \(L^p\)-density. By
\cite{Kol98}, there is a unique current
\begin{equation}
S_A=\omega_A+\ddc v\in c_1(A),
\qquad
S_A^n=\nu,
\qquad
\sup_Xv=0,
\label{eq:compactness-auxiliary-MA}
\end{equation}
and \(v\) is continuous.

If \(g\in\widehat K_T\), then \(g^*T=T\), and naturality of the
non-pluripolar product under biholomorphisms gives
\[
g^*\mu
=
g^*\langle T^n\rangle
=
\langle(g^*T)^n\rangle
=
\mu.
\]
Thus \(g^*\nu=\nu\). Since \(g^*A\simeq A\), the current \(g^*S_A\)
lies in \(c_1(A)\) and solves the same Monge--Amp\`ere equation as
\(S_A\). Uniqueness in \eqref{eq:compactness-auxiliary-MA} therefore
gives
\begin{equation}
g^*S_A=S_A,
\qquad g\in\widehat K_T.
\label{eq:compactness-auxiliary-current-invariance}
\end{equation}

Let \(\Phi\) be the bounded psh metric on \(A\) whose curvature is
\(S_A\). Denote by \(\widetilde{\operatorname{Aut}}(X,A)\) the group
of pairs \((g,\widetilde g)\), where
\(g\in\operatorname{Aut}(X)\) and
\[
\widetilde g:A\longrightarrow A
\]
is a holomorphic bundle automorphism covering \(g\), linear on each
fiber. Equivalently, such a lift exists precisely when
\(g^*A\simeq A\). Consequently, \eqref{eq:compactness-invariant-polarization} yields that the projection of \(\widetilde{\operatorname{Aut}}(X,A)\) onto \(\operatorname{Aut}(X)\) contains \(\operatorname{Aut}^0(X)\), hence also \(\widehat K_T\).

For \(g\in\widehat K_T\), choose such a lift \(\widetilde g\). By
\eqref{eq:compactness-auxiliary-current-invariance},
\[
\ddc\bigl(\widetilde g^{\,*}\Phi-\Phi\bigr)
=
g^*S_A-S_A
=
0.
\]
The difference \(\widetilde g^{\,*}\Phi-\Phi\) is a global
pluriharmonic distribution, hence is constant because \(X\) is
compact and connected. Multiplying \(\widetilde g\) by a suitable
nonzero scalar, we may arrange that
\(\widetilde g^{\,*}\Phi=\Phi\). Consequently,
\(\widehat K_T\) is contained in the projection to
\(\operatorname{Aut}(X)\) of
\[
\widetilde G_\Phi
:=
\left\{
(g,\widetilde g)\in\widetilde{\operatorname{Aut}}(X,A):
\widetilde g^{\,*}\Phi=\Phi
\right\}.
\]

By \cite[Lemma~5.3]{BBEGZ19}, the group \(\widetilde G_\Phi\) is
compact, and hence its projection to \(\operatorname{Aut}(X)\) is
compact. Finally, \(\widehat K_T\) is closed in this projection:
\(\operatorname{Aut}^0(X)\) is closed in
\(\operatorname{Aut}(X)\), and convergence in
\(\operatorname{Aut}(X)\) implies smooth convergence of the
corresponding biholomorphisms, so pullback of currents is continuous
in the weak topology. Therefore the condition \(g^*T=T\) is closed.
It follows that \(\widehat K_T\) is compact.
\end{proof}

We retain the notation preceding
Theorem~\ref{thm:main-big-Matsushima}.  For the uniquely normalized
potential \(\varphi\) of \(T\),
\begin{equation}
 \mu_T:=\langle T^n\rangle=e^{-\lambda\varphi-\psi}\Omega,
 \qquad
 \operatorname{Ric}(\mu_T)=\lambda T+\eta.
 \label{eq:intrinsic-twisted-KE}
\end{equation}
Contraction of a current by a smooth vector field and the Lie derivative of
a current are always understood distributionally.  We identify the real
Lie algebra of \(\operatorname{Aut}^0(X)\) with the underlying real vector space of
\(H^0(X,T_X^{1,0})\); multiplication by \(i\) is the complex structure
on this Lie algebra.

Before proving the Matsushima statement, we collect the elementary
relations among the twisting symmetries, normalized pullbacks, and
their Lie algebras.

\begin{lemma}
\label{lem:twist-symmetries-pullback}
The following statements hold.
\begin{enumerate}[\upshape(i)]
\item The group \(\widehat H_\eta\) is a closed embedded real Lie subgroup of
\(\operatorname{Aut}^{0}(X)\).

\item Every \(h\in\widehat H_\eta\) sends a solution current to a
solution current.  More precisely, there exists a unique potential
\(u_h\) such that
\[
 h^*T=\theta+\ddc u_h,
 \qquad
 \big\langle(h^*T)^n\big\rangle
 =e^{-\lambda u_h-\psi}\Omega.
\]

\item The full stabilizer $\widehat K_T$ is contained in \(\widehat H_\eta\).

\item The space \(\mathfrak g_\eta\) is a complex Lie subalgebra of
\(H^0(X,T_X^{1,0})\).  If
\[
 \jmath_\eta:G_\eta\longrightarrow\operatorname{Aut}^{0}(X)
\]
denotes the canonical injective immersion of the connected analytic
subgroup integrating \(\mathfrak g_\eta\), then
\[
 \jmath_\eta(G_\eta)\subset H_\eta.
\]

\item  \(L_T\) is a closed Lie subgroup of the intrinsic Lie group
\(G_\eta\), and
\begin{equation}
 \operatorname{Lie}_{\mathbb R}(L_T)
 =\mathfrak k_T^{\mathrm{null}}
 =
 \mathfrak g_\eta
 \cap
 \operatorname{Lie}_{\mathbb R}(\widehat K_T).
 \label{eq:L-setwise-and-Lie}
\end{equation}
\end{enumerate}
\end{lemma}

\begin{proof}
\begin{enumerate}[\upshape(i)]
\item The natural action of the finite-dimensional Lie group
\(\operatorname{Aut}^{0}(X)\) on \(X\) is smooth.  Since \(X\) is
compact, convergence in \(\operatorname{Aut}^{0}(X)\) implies smooth
convergence of the corresponding biholomorphisms and their inverses.
Consequently, pullback of a fixed current is continuous in the weak
topology. Thus the map 
\[
 h\longmapsto h^*\eta
\]
is continuous, and \(\widehat H_\eta\) is closed in
\(\operatorname{Aut}^{0}(X)\).  Cartan's closed subgroup theorem then
makes \(\widehat H_\eta\) an embedded real Lie subgroup.

\item Fix \(h\in\widehat H_\eta\).  Since \(h\in
\operatorname{Aut}^{0}(X)\) is homotopic to the identity, it acts
trivially on de Rham cohomology.  Hence \(h^*T\) lies in
\([\theta]\). 

The non-pluripolar product is natural under biholomorphisms:
\begin{equation}
 \big\langle(h^*T)^n\big\rangle
 =
 h^*\langle T^n\rangle.
 \label{eq:NPP-pullback-naturality}
\end{equation}
Set
\[
 \nu:=h^*\mu_T.
\]
By the naturality of the Ricci current, the identity
\(h^*\eta=\eta\), and \eqref{eq:intrinsic-twisted-KE}, one has
\[
 \operatorname{Ric}(\nu)
 =
 h^*\operatorname{Ric}(\mu_T)
 =
 \lambda h^*T+\eta.
\]
By the \(\partial\bar\partial\)-lemma, there exists a global potential
\(u\) such that
\[
 h^*T=\theta+\ddc u.
\]
On the other hand,
\[
 \begin{aligned}
 \operatorname{Ric}\bigl(e^{-\lambda u-\psi}\Omega\bigr)
 &=
 \lambda\ddc u+\ddc\psi+\operatorname{Ric}(\Omega)\\
 &=
 \lambda(\theta+\ddc u)
 +\operatorname{Ric}(\Omega)-\lambda\theta+\ddc\psi\\
 &=
 \lambda h^*T+\eta.
 \end{aligned}
\]
Write
\[
 \nu=F\Omega,
 \qquad
 e^{-\lambda u-\psi}\Omega=G\Omega.
\]
Since \(F,G\in(0,+\infty)\) almost everywhere and
\(\log F,\log G\in L^1_{\mathrm{loc}}(X)\), the function
\[
 \chi:=\log F-\log G\in L^1_{\mathrm{loc}}(X)
\]
is well defined up to a set of measure zero; equivalently,
\(\chi=\log(F/G)\) almost everywhere.  Equality of the preceding Ricci
currents gives
\[
 \ddc \chi=0
\]
in the sense of distributions.  On each holomorphic coordinate chart,
taking the Euclidean trace yields \(\Delta \chi=0\) distributionally.
Weyl's lemma therefore shows that \(\chi\) admits a smooth harmonic
representative. 
Hence \(\chi\) is smooth and pluriharmonic.  Since \(X\) is
compact and connected, it is constant.  Hence
\[
 \nu=C e^{-\lambda u-\psi}\Omega
\]
for some \(C>0\).  Setting
\[
 u_h:=u-\lambda^{-1}\log C
\]
gives
\[
 h^*T=\theta+\ddc u_h,
 \qquad
 \big\langle(h^*T)^n\big\rangle
 =\nu=e^{-\lambda u_h-\psi}\Omega.
\]
Finally, two potentials of \(h^*T\) differ by a constant, and the
normalized equation determines that constant uniquely because
\(\lambda>0\).

\item Let \(g\in\widehat K_T\), where \(\widehat K_T\) is the full
stabilizer in \(\operatorname{Aut}^{0}(X)\).  Naturality of the
non-pluripolar product gives
\[
 g^*\mu_T
 =
 g^*\langle T^n\rangle
 =
 \big\langle(g^*T)^n\big\rangle
 =
 \langle T^n\rangle
 =
 \mu_T.
\]
Therefore,
\[
 \begin{aligned}
 \lambda T+\eta
 &=
 \operatorname{Ric}(\mu_T)=
 \operatorname{Ric}(g^*\mu_T)\\
 &=
 g^*\operatorname{Ric}(\mu_T)=
 \lambda g^*T+g^*\eta\\
 &=
 \lambda T+g^*\eta.
 \end{aligned}
\]
It follows that \(g^*\eta=\eta\), and hence
\[
 \widehat K_T\subset\widehat H_\eta.
\]

\item The space \(\mathfrak g_\eta\) is complex linear.  If
\(V,W\in\mathfrak g_\eta\), then distributional Cartan calculus and
\(d\eta=0\) give
\[
 \mathcal L_V\eta
 =
 d(\iota_V\eta)+\iota_Vd\eta
 =
 0
\]
and
\[
 \iota_{[V,W]}\eta
 =
 \mathcal L_V(\iota_W\eta)
 -
 \iota_W(\mathcal L_V\eta)
 =
 0.
\]
Thus \(\mathfrak g_\eta\) is a complex Lie subalgebra.

Since \(\eta\) is real, \(\iota_V\eta=0\) also implies
\(\iota_{\overline V}\eta=0\).  Hence the real vector fields underlying
\(V\) and \(iV\) both have zero contraction with \(\eta\).  They are
complete because \(X\) is compact, and their real flows preserve
\(\eta\) by Cartan's formula.  Thus, regarding
\(\mathfrak g_\eta\) as a real Lie subalgebra (under this identification $V$ is identified with $V+\bar V$ as a real vector field),
\[
 \exp_{\operatorname{Aut}^{0}(X)}(tV)\in H_\eta
 \qquad
 (V\in\mathfrak g_\eta,\ t\in\mathbb R).
\]
Compatibility of \(\jmath_\eta\) with exponential maps gives
\[
 \jmath_\eta\bigl(\exp_{G_\eta}(tV)\bigr)
 =
 \exp_{\operatorname{Aut}^{0}(X)}
 \bigl(t(d\jmath_\eta)_eV\bigr).
\]
After identifying $(d\jmath_\eta)_e
 \bigl(\operatorname{Lie}_{\mathbb R}(G_\eta)\bigr)
 =
 \mathfrak g_\eta$,
the right-hand side lies in \(H_\eta\).  Since the connected Lie group
\(G_\eta\) is generated by a connected neighborhood of its identity and the implicit function theorem yields that $\exp_{G_\eta}$ is a local diffeomorphism from a neighborhood of $0\in \mathfrak g_\eta$ to a neighborhood of the identity $e\in G_\eta$, it follows that
\[
 \jmath_\eta(G_\eta)\subset H_\eta.
\]

\item The same weak-continuity argument used in the first part shows that
\(\widehat K_T\) is closed in \(\operatorname{Aut}^{0}(X)\).
Since \(\jmath_\eta\) is continuous for the intrinsic Lie-group
topology of \(G_\eta\), the subgroup \(L_T\) is closed in \(G_\eta\).
The closed subgroup theorem therefore makes \(L_T\) an embedded real
Lie subgroup of the intrinsic group \(G_\eta\).

Finally, the one-parameter subgroup criterion in Lie theory and compatibility of
\(\jmath_\eta\) with exponential maps give
\[
 \operatorname{Lie}_{\mathbb R}(L_T)
 =
 (d\jmath_\eta)_e^{-1}
 \bigl(
   \operatorname{Lie}_{\mathbb R}(\widehat K_T)
 \bigr).
\]
Under the identification $(d\jmath_\eta)_e
 \bigl(\operatorname{Lie}_{\mathbb R}(G_\eta)\bigr) = \mathfrak g_\eta$, this becomes
\[
 \operatorname{Lie}_{\mathbb R}(L_T)
 =
 \mathfrak g_\eta
 \cap
 \operatorname{Lie}_{\mathbb R}(\widehat K_T).
\]
Together with the definition of $\mathfrak k_T^{\mathrm{null}}$, this
proves \eqref{eq:L-setwise-and-Lie}.
\end{enumerate}
\end{proof}

We now prove the main reductivity theorem.  The transitivity and
polar factorization used below are intermediate consequences of Theorem \ref{thm:main-big-BM}.
The final real-form argument follows the classical
Bando--Mabuchi--Matsushima strategy \cite{BM87,Mat57}.

\begin{proof}[Proof of Theorem~\ref{thm:main-big-Matsushima}]
Let \(T_0,T_1\in\mathcal S\). By Theorem \ref{thm:main-big-BM},
there is a holomorphic vector field \(W\in\mathfrak g_\eta\) such
that the time-one map of its complex flow sends \(T_1\) to \(T_0\).
Since the analytic subgroup integrating \(\mathfrak g_\eta\) is
\(G_\eta\), this proves the transitivity of \(G_\eta\) on
\(\mathcal S\).

We next prove the first factorization in
\eqref{eq:intrinsic-polar-factorization}. Let \(a\in G_\eta\). By the
preceding lemma, \(a^*T\) is again a solution. Applying Theorem \ref{thm:main-big-BM} to \(T\) and \(a^*T\), we obtain
\(W\in\mathfrak g_\eta\) whose complex one-parameter subgroup
$\Lambda(z):=\exp_{\operatorname{Aut}^{0}(X)}(zW)$
satisfies
\begin{equation}
\Lambda(1)^*a^*T=T,
\qquad
\Lambda(it)^*T=T
\quad\text{for every }t\in\mathbb R.
\label{eq:BM-polar-flow}
\end{equation}
Compatibility of the exponential maps with \(\jmath_\eta\) shows
that \(\Lambda(\mathbb C)\subset\jmath_\eta(G_\eta)\). Indeed, $\Lambda(z)=\exp_{\operatorname{Aut}^0(X)}(zW)=\jmath_\eta(\exp_{G_\eta}(zW))\in \jmath_\eta(G_\eta)$.

Set $\xi:=iW$.
The second identity in \eqref{eq:BM-polar-flow} says that the
one-parameter subgroup generated by \(iW\) lies in \(L_T\). Hence
$\xi=iW\in\operatorname{Lie}_{\mathbb R}(L_T)
=\mathfrak k_T^{\mathrm{null}}$.  Moreover,
$\Lambda(1)=\exp(W)=\exp(-i\xi)$.  Therefore
$\ell:=a\Lambda(1)=a\exp(-i\xi)$ belongs to \(G_\eta\), and
$\ell^*T=\Lambda(1)^*a^*T=T$.  Thus \(\ell\in L_T\) and
$a=\ell\exp(i\xi)$.
Applying this factorization to \(a^{-1}\) and taking inverses gives
the opposite order.  Thus
\begin{equation}
 G_\eta
 =L_T\exp\bigl(i\mathfrak k_T^{\mathrm{null}}\bigr)
 =\exp\bigl(i\mathfrak k_T^{\mathrm{null}}\bigr)L_T.
 \label{eq:intrinsic-polar-factorization}
\end{equation}

It remains to identify the compact real form.  We first prove that
\begin{equation}
\mathfrak k_T^{\mathrm{null}}
\cap
i\mathfrak k_T^{\mathrm{null}}
=
\{0\}.
\label{eq:k-null-totally-real}
\end{equation}
Choose a K\"ahler form \(\omega_0\) on \(X\). Since \(\widehat K_T\) is
compact, averaging over its normalized Haar measure gives a
\(\widehat K_T\)-invariant K\"ahler form
\[
\omega_K
:=
\int_{\widehat K_T}g^*\omega_0\,d\chi(g).
\]
Let \(g_K\) be its associated K\"ahler metric,
$g_K(U,V):=\omega_K(U,JV)$.

Suppose that
$\xi\in\mathfrak k_T^{\mathrm{null}}
\cap i\mathfrak k_T^{\mathrm{null}}$.
Then both \(\xi\) and \(i\xi\) belong to
\(\mathfrak k_T^{\mathrm{null}}\): indeed, if
\(\xi=i\zeta\) with
\(\zeta\in\mathfrak k_T^{\mathrm{null}}\), then
\(i\xi=-\zeta\in\mathfrak k_T^{\mathrm{null}}\).

Let \(Y\) be the real holomorphic vector field corresponding to
\(\xi\) (namely, $Y=\xi+\bar\xi$). The real holomorphic vector field corresponding to \(i\xi\)
is \(JY\). By
\eqref{eq:L-setwise-and-Lie}, both \(Y\) and \(JY\) belong to the Lie
algebra of \(\widehat K_T\). The invariance of \(\omega_K\) therefore
gives
\[
\mathcal L_Y\omega_K=0,
\qquad
\mathcal L_{JY}\omega_K=0.
\]
Since these vector fields are holomorphic, they preserve \(J\);
hence they are also Killing fields for \(g_K\).

Cartan's formula gives
$d(\iota_Y\omega_K)=\mathcal L_Y\omega_K=0$.
By \eqref{eq:compactness-nadel-vanishing},
\(H^1(X,\mathcal O_X)=0\). Since \(X\) is compact K\"ahler, it follows from the Hodge decomposition
that \(H^1(X,\mathbb R)=0\). Consequently, there exists a smooth real
function \(f\) such that $\iota_Y\omega_K=df$.  With the above
convention for \(g_K\), this is equivalent to
$\nabla^{g_K}f=JY$.
But \(JY\) is Killing, so
\[
0
=
\mathcal L_{JY}g_K
=
\mathcal L_{\nabla^{g_K}f}g_K
=
2\operatorname{Hess}_{g_K}f.
\]
Taking the trace gives \(\Delta_{g_K}f=0\). Since \(X\) is compact and
connected, \(f\) is constant. Hence \(JY=0\), and therefore
\(\xi=0\). This proves \eqref{eq:k-null-totally-real}.

Set $r:=\dim_{\mathbb R}\mathfrak k_T^{\mathrm{null}}$ and
$m:=\dim_{\mathbb C}\mathfrak g_\eta$.
Equation \eqref{eq:k-null-totally-real} implies \(r\leq m\). On the
other hand, the polar factorization \eqref{eq:intrinsic-polar-factorization} gives a surjective smooth map
\[
F:
L_T\times\mathfrak k_T^{\mathrm{null}}
\longrightarrow G_\eta,
\qquad
F(\ell,\xi)=\ell\exp(i\xi).
\]
Each connected component of its domain has real dimension \(2r\),
whereas \(G_\eta\) has real dimension \(2m\). If \(r<m\), Sard's
theorem would imply that the image of each such component has measure
zero in \(G_\eta\). Since \(L_T\) is second countable and hence has at
most countably many connected components, this would contradict the
surjectivity of \(F\). Thus \(r\geq m\), and consequently \(r=m\).

Together with \eqref{eq:k-null-totally-real}, this proves
\eqref{eq:null-Lie-real-form}. Furthermore,
\(\mathfrak k_T^{\mathrm{null}}\) is a Lie subalgebra of the Lie
algebra of the compact group \(\widehat K_T\), so it is a compact real
Lie algebra.  Since
$\mathfrak g_\eta\simeq
\mathfrak k_T^{\mathrm{null}}\otimes_{\mathbb R}\mathbb C$,
the complex Lie algebra \(\mathfrak g_\eta\) is reductive.
\end{proof}

\end{document}